\documentclass[a4paper, 10pt, twoside, notitlepage]{amsart}

\usepackage[utf8]{inputenc}
\usepackage{color}
\usepackage{amsmath} 
\usepackage{amssymb} 
\usepackage{amsthm}
\usepackage{graphicx}
\usepackage{esint}
\usepackage[colorlinks=true,linkcolor=blue]{hyperref}
\usepackage[export]{adjustbox}
\usepackage{hyphenat}

\theoremstyle{plain}
\newtheorem{thm}{Theorem}
\newtheorem{prop}{Proposition}[section]
\newtheorem{nota}[prop]{Notation}

\newtheorem{lem}[prop]{Lemma}
\newtheorem{cor}[prop]{Corollary}

\newtheorem{defi}[prop]{Definition}
\newtheorem{rmk}[prop]{Remark}

\newtheorem{alg}[prop]{Algorithm}
\newtheorem{claim}{Claim}
\newtheorem{assume}{Assumption}
\newtheorem{obs}{Observation}

\newcommand {\R} {\mathbb{R}} \newcommand {\Z} {\mathbb{Z}}
 \newcommand {\N} {\mathbb{N}}
 
\newcommand {\p} {\partial}

\newcommand {\va} {\varphi}

\newcommand {\supp} {\text{supp}}

\DeclareMathOperator{\tr}{tr}

\DeclareMathOperator {\dist} {dist}

\DeclareMathOperator {\intconv} {intconv}
\DeclareMathOperator {\conv}{conv}

\DeclareMathOperator {\argmin} {argmin}
\DeclareMathOperator {\Skew} {Skew}

\DeclareMathOperator {\rank} {rank}
\DeclareMathOperator {\Per} {Per}
\DeclareMathOperator {\sgn} {sgn}

\begin{document}

\title[Higher Sobolev Regularity of Convex Integration Solutions]{Higher Sobolev Regularity of Convex Integration Solutions in Elasticity}

\author{Angkana R\"uland }
\author{Christian Zillinger}
\author{Barbara Zwicknagl}

\address{
Mathematical Institute of the University of Oxford, Andrew Wiles Building, Radcliffe Observatory Quarter, Woodstock Road, OX2 6GG Oxford, United Kingdom }
\email{ruland@maths.ox.ac.uk}

\address{
Department of Mathematics,
University of Southern California,
Los Angeles, CA 90089-2532, US}
\email{zillinge@usc.edu}

\address{
Institute for Applied Mathematics, Universit\"at Bonn, Endenicher Allee 60, 53115 Bonn, Germany  }
\email{zwicknagl@iam.uni-bonn.de}

\begin{abstract}
In this article we discuss quantitative properties of convex integration solutions arising in problems modeling shape-memory materials. For a two-dimensional, geometrically linearized model case, the hexagonal-to-rhombic phase transformation, we prove the existence of convex integration solutions $u$ with higher Sobolev regularity, i.e. there exists $\theta_0>0$ such that $\nabla u \in W^{s,p}_{loc}(\R^2)\cap L^{\infty}(\R^2)$ for $s\in(0,1)$, $p\in(1,\infty)$ with $0<sp < \theta_0$. We also recall a construction, which shows that in situations with additional symmetry much better regularity properties hold.
\end{abstract}



\thanks{
A.R. acknowledges a Junior Research Fellowship at Christ Church.
C.Z. and B.Z. acknowledge support from the DFG through CRC 1060 ``The mathematics of emergent effects''.}

\maketitle
\tableofcontents

\section{Introduction}
In this article we are concerned with the detailed analysis of certain convex integration solutions, which arise in the modeling of solid-solid, diffusionless phase transformations in shape-memory materials. We seek to precisely analyze the regularity properties of these constructions in a simple, two-dimensional, geometrically linear model case.\\
Shape-memory materials undergo a solid-solid, diffusionless phase transition
upon temperature change (see e.g. \cite{B} and the references given there): In the high temperature phase, the \emph{austenite phase}, the materials form very symmetric lattices. Upon cooling down the material, the symmetry of the lattice is reduced, the material transforms into the \emph{martensitic phase}. Due to the loss of symmetry, there are different \emph{variants of martensite}, which make these materials very flexible at low temperature and give rise to a variety of different microstructures. Mathematically, it has proven very successful to model this behavior variationally in a continuum framework as the following minimization problem \cite{B3}:
\begin{align}
\label{eq:var}
\min \int\limits_{\Omega} W(\nabla y, \theta) dx.
\end{align}
Here $\Omega \subset \R^{n}$ is the reference configuration of the undeformed material. 
The mapping $y:\Omega \rightarrow \R^n$ describes the \emph{deformation} of the material with respect to the reference configuration. It is assumed to be of a suitable Sobolev regularity.
The function $W:\R^{n\times n} \times \R \rightarrow [0,\infty)$ denotes the \emph{energy density} of a given deformation gradient $M\in \R^{n \times n}$ at a certain temperature $\theta \in \R$. Due to frame indifference, $W$ is required to be invariant with respect to rotations, i.e. 
\begin{align*}
W(QM,\theta) = W(M,\theta) \mbox{ for all } Q\in SO(n), \theta \in \R, M\in\R^{n\times n}. 
\end{align*}
Modeling the behavior of shape-memory materials, the energy density further reflects the physical properties of these materials. In particular, it is assumed that at high temperatures $\theta > \theta_c$ the energy density
$W$ has a single minimum (modulo $SO(n)$ symmetry), which (upon normalization)
we may assume to be given by the $SO(n)$ orbit of $\alpha(\theta) Id$, where
$\alpha: \R \rightarrow (0,\infty)$ with $\alpha(\theta_c)=1$ (c.f. \cite{Ball:ESOMAT}). This is the \emph{(austenite) energy well} at temperature $\theta$. Upon lowering the temperature below a critical temperature $\theta_c$, the function $W$ displays a (discrete) multi-well behavior (modulo $SO(n)$): There exist finitely many matrices $U_1(\theta),\dots,U_m(\theta) \in \R^{n\times n}_{+}$, $m\in \N$, such that
\begin{align*}
W(M,\theta) = 0 \Leftrightarrow M \in \bigcup\limits_{j=1}^{m}SO(n) U_j(\theta).
\end{align*} 
The matrices $U_j(\theta)$ represent the variants of martensite at temperature $\theta <\theta_c$ and are referred to as the \emph{(martensite) energy wells}. At the critical temperature $\theta= \theta_c$ both the austenite and the martensite wells are energy minimizers. \\
In the sequel, we assume that $\theta<\theta_c$ is fixed, so that only the variants of martensite are energy minimizers. We seek to study the quantitative behavior of minimizers for energies of the type (\ref{eq:var}). Here we make the following simplifications:
\begin{itemize}
\item[(i)] \emph{Reduction to the $m$-well problem.} Instead of studying the full variational problem (\ref{eq:var}), we only focus on exact minimizers. Restricting to the low temperature regime, this implies that we seek solutions to the differential inclusion
\begin{align}
\label{eq:mwell}
\nabla y \in \bigcup\limits_{j=1}^{m}SO(n) U_j(\theta),
\end{align}
for some $\theta < \theta_c$.
\item[(ii)] \emph{Small deformation gradient case, geometric linearization.} We further modify (\ref{eq:mwell}) and assume that $\nabla y$ is close to the identity. This allows us to linearize the problem around this constant value (c.f. Chapter 11 in \cite{B}). Instead of considering (\ref{eq:mwell}), we are thus lead to the inclusion problem
\begin{align}
\label{eq:mwell_lin}
e(\nabla u):=\frac{\nabla u + (\nabla u)^T}{2} \in \{e_{1},\dots,e_m\}.
\end{align}
The symmetrized gradient $e(\nabla u)$ represents the infinitesimal deformation
\emph{strain} associated with the \emph{displacement} $u$, which is defined as $u(x):=y(x)-x$ (with slight abuse of physical convention in the sequel we do not distinguish between the deformation and the displacement in our use of language and will simply refer to both as a ``deformation''). The symmetric matrices $e_1,\dots,e_m \in \R^{n\times n}$ are the
\emph{exactly stress-free strains} representing the variants of martensite.
While this linearizes the \emph{geometry} of the problem (by replacing the symmetry group $SO(n)$ by an invariance with respect to the linear space $\Skew(n)$), the differential inclusion (\ref{eq:mwell_lin}) preserves the inherent \emph{physical nonlinearity}, which arises from the multi-well structure of the problem.
\item[(iii)] \emph{Reduction to two dimensions and the hexagonal-to-rhombic phase transformation.} In the sequel studying an as simple as possible model case, we restrict to two dimensions and a specific two-dimensional phase transformation, the \emph{hexagonal-to-rhombic} phase transformation (this is for instance used in studying materials such as $\mbox{Mg}_2\mbox{Al}_4\mbox{Si}_5\mbox{O}_{18}$ or Mg-Cd alloys undergoing a (three-dimensional) hexagonal-to-orthorhombic transformation, c.f. \cite{CPL14}, \cite{KK91}, and also for closely related materials such as $\mbox{Pb}_3(\mbox{VO}_4)_2$, which undergo a
(three-dimensional) hexagonal-to-monoclinic transformation, c.f. \cite{MA80a}, \cite{MA80}, \cite{CPL14}). From a microscopic point of view, the hexagonal-to-rhombic phase transformation 
occurs, if a hexagonal atomic lattice is transformed into a rhombic atomic lattice.
From a continuum point of view, we model it as solutions to the differential inclusion
\begin{equation}
\label{eq:incl_1}
\begin{split}
 & u: \R^{2} \rightarrow \R^{2},\\
  &\frac{1}{2}(\nabla u +
  (\nabla u)^{T}) \in K \mbox{ a.e. in }\Omega,
\end{split}
\end{equation}
where $\Omega \subset \R^2$ is a bounded Lipschitz domain and
\begin{equation}
\label{eq:K3}
\begin{split}
&K:=\{e^{(1)}, e^{(2)}, e^{(3)}\} \mbox{ with }\\
&e^{(1)}:= \begin{pmatrix}
1 & 0 \\
0& -1
\end{pmatrix},
e^{(2)}:= \frac{1}{2}\begin{pmatrix}
-1 & \sqrt{3} \\
\sqrt{3}& 1
\end{pmatrix},
e^{(3)}:= \frac{1}{2}\begin{pmatrix}
-1 & -\sqrt{3}\\
-\sqrt{3}& 1
\end{pmatrix}.
\end{split}
\end{equation}
We note that all the matrices in $K$ are trace-free, which corresponds to the (infinitesimal) volume preservation of the transformation.
We note that the set $K $ is ``large" (its convex hull is a two-dimensional set in the three-dimensional ambient space of two-by-two, symmetric matrices, c.f. Lemma \ref{lem:convex_hull}).
\end{itemize} 
In the sequel, we study the problem (\ref{eq:incl_1}), (\ref{eq:K3}) and 
investigate regularity properties of its solutions. 

\subsection{Main result}

The geometrically linearized hexagonal-to-rhombic phase transformation is a very flexible transformation, which allows for numerous exact solutions to the associated three-well problem (\ref{eq:incl_1}) with different types of boundary data. Here the simplest possible solutions are so-called \emph{simple laminates}, for which the strain is a one-dimensional function 
$e(\nabla u)(x)=f(x\cdot n)$ for some vector $n \in S^1$ and for which
$$f(x\cdot n) \in\{e^{(i_1)}, e^{(i_2)}\} \mbox{ a.e. in } \Omega, \
i_1, i_2 \in\{1,2,3\} \mbox{ and } i_1 \neq i_2,
$$ 
i.e. $e(\nabla u)$ only attains two values. The possible directions of these laminates, as given by the vector $n \in S^1$ are (up to sign reversal) six discrete values, which arise as the \emph{symmetrized rank-one directions} between the energy wells: For each $i_1,i_2 \in \{1,2,3\}$ with $i_1 \neq i_2$ there exists (up to sign reversal and exchange of the roles of $a_{i_1,i_2}$ and $n_{i_1,i_2}$ and renormalization) exactly one pair $(a_{i_1,i_2},n_{i_1,i_2}) \in \R^{2} \setminus \{0\}\times S^1$ with the property that
\begin{align*}
e^{(i_1)} - e^{(i_2)} = a_{i_1,i_2} \odot n_{i_1,i_2} 
:= \frac{1}{2}(a_{i_1, i_2}\otimes n_{i_1,i_2} + n_{i_1,i_2} \otimes a_{i_1,i_2}).
\end{align*}
The possible vectors are collected in Lemma \ref{lem:geo2}. \\
In addition to these ``simple" constructions, there are further exact solutions
to the three-well problem associated with the hexagonal-to-rhombic phase
transformation, e.g. there are patterns involving all three variants as depicted
in Figures \ref{fig:zeroh} and \ref{fig:cross} in the Appendix (Section \ref{sec:append}).\\
In the sequel, we study solutions to the hexagonal-to-rhombic phase transformation with \emph{affine boundary conditions}, i.e. we consider $u\in W^{1,p}_{\text{loc}}(\R^2)$ with $p\in(2,\infty]$ such that
\begin{equation}
\label{eq:incl}
\begin{split}
 & u: \R^{2} \rightarrow \R^{2},\\
  &\nabla u = M \mbox{ a.e. in } \R^2\setminus \Omega, \\
  &\frac{1}{2}(\nabla u +
  (\nabla u)^{T}) \in K \mbox{ a.e. in }\Omega.
\end{split}
\end{equation}
Here we investigate the rigidity/ non-rigidity of the problem by asking whether it has non-affine solutions:
\begin{itemize}
\item[(Q1)] Are there (non-affine) solutions to (\ref{eq:incl}) with $M \in \R^{2\times 2}$?
\end{itemize}
Clearly, a necessary condition for this is that $e(M)\in \conv(K)$. Using the method of convex integration, Müller and {\v{S}}ver{\'a}k \cite{MS} (c.f. also the Baire category arguments of \cite{D}, \cite{DaM12}) constructed multiple solutions to related differential inclusions, displaying the existence of a variety of solutions to the problem. Noting that these techniques are applicable to our set-up of the three-well problem, ensures that for any $M$ with $e(M)\in \intconv(K)$ there exists a non-affine solution to (\ref{eq:incl}).\\
In general these convex integration solutions are however very ``wild" in the sense that they do not possess very strong regularity properties (c.f. \cite{DM1}). As our inclusion (\ref{eq:incl}) is motivated by a physical problem, a natural question addresses the relevance of this multitude of solutions:
\begin{itemize}
\item[(Q2)] Are all the convex integration solutions physically relevant? Or are they only mathematical artifacts? Is there a mechanism distinguishing between the ``only mathematical" and the ``really physical" solutions?
\end{itemize}
Guided by the physical problem at hand and the literature on these problems,
natural criteria to consider are surface energy constraints and surface
energy regularizations. For our differential inclusion these translate into
regularity constraints and lead to the question, whether \emph{unphysical}
convex integration solutions have a natural regularity threshold. Here an
immediate regularity property of solutions to (\ref{eq:incl}) is that $e(\nabla
u)\in L^{\infty}(\R^2)$. With slightly more care, it is also possible to
obtain solutions with the property that $u\in W^{1,\infty}_{\text{loc}}(\R^2)$. However, prior to
this work it was not known whether these solutions can enjoy more regularity, i.e.
whether for instance there are convex integration solutions with $\nabla u \in W^{s,p}(\Omega)$ for some $s>0$, $p\geq 1$.
\\
Motivated by these questions, in this article, we study the regularity of a specific convex integration construction and obtain \emph{higher Sobolev regularity} properties for the resulting solutions: 

\begin{thm}
\label{thm:main}
Let $\Omega\subset \R^2$ be a bounded Lipschitz domain. Let 
$K$ be as in (\ref{eq:K3}) and let $M\in \R^{2\times 2}$ be such that 
$e(M):= \frac{M+M^T}{2} \in \intconv(K)$.
Then there exist a value $\theta_0\in(0,1)$, depending only on $\frac{\dist(e(M), \partial \conv(K))}{\dist(e(M),K)}$, and a deformation
$u: \R^2 \rightarrow \R^2$ with $u\in W^{1,\infty}_{loc}(\R^2)$ such that (\ref{eq:incl}) holds and such that $\nabla u\in W^{s,p}_{loc}(\R^2)\cap L^{\infty}(\R^2)$ for all $s\in(0,1)$, $p\in (1,\infty)$ with $s p< \theta_0$.
\end{thm}

Let us comment on this result: To the best of our knowledge it represents the first $W^{s,p}$ higher regularity result for convex integration solutions arising in differential inclusions for shape-memory materials. 
The given \emph{quantitative} dependences for $\theta_0$ are certainly not optimal in the specific constants. While it is certainly possible to improve on these numeric values, a more interesting question deals with the \emph{qualitative} expected dependences: Is it necessary that $\theta_0$ depends on $\frac{\dist(e(M), \partial \conv(K))}{\dist(e(M),K)}$?\\
Since for $M\in \R^{2\times 2}$ with $e(M)\in \partial \conv(K)$ there are no non-affine solutions to (\ref{eq:incl}), it is natural to expect that convex integration constructions deteriorate for matrices $M$ with $e(M)$ approaching the boundary of $\conv(K)$. The precise dependence on the behavior towards the boundary however is less intuitive.
In this context, it is interesting to note that the regularity threshold $\theta_0>0$ does not depend on the \emph{distance to the boundary} of $K$, but rather on the \emph{angle}, which is formed between the initial matrix $e(M)$ and the boundary of $K$. This is in agreement with the intuition that the larger the angle is, the better the convex integration algorithm becomes, as it moves the values of the iterations, which are used to construct the displacement $u$, further into the interior of $K$. In the interior of $K$ it is possible to use larger length scales, which increases the regularity of solutions. Whether this dependence is necessary in the value of the product of $sp$ or whether the product $sp$ should be independent of this and only the value of the corresponding norm should deteriorate with a smaller angle, is an interesting open question.\\
We remark that in the special case of \emph{additional symmetries} it is possible to construct much better solutions. An example is given in the appendix for the case $M=0$ (c.f. also \cite{Pompe} and \cite{CPL14}).
It is an important and challenging open question, whether it is possible to exploit further symmetries and thus to construct further solutions with these much better regularity properties.

\subsection{Literature and context}
A fascinating problem in studying solid-solid, diffusionless phase transformations modeling shape-memory materials is the dichotomy between rigidity and non-rigidity. Since the work of M\"uller and {\v{S}}ver{\'a}k \cite{MS}, who adapted the convex integration method of Gromov \cite{G}, \cite{EM} and Nash-Kuiper \cite{N1}, \cite{Ku} to the situation of solid-solid phase transformations, and the work of Dacorogna and Marcellini \cite{D}, \cite{DaM12}, it is known that under suitable conditions on the convex hulls of the energy wells, there is a very large set of possible minimizers to (\ref{eq:var})  (c.f. also \cite{S} and \cite{K1} for a comparison of these two methods).
More precisely, the set of minimizers forms a residual set (in the Baire sense) in the associated function spaces. However, in general convex integration solutions are ``wild"; they do not enjoy very good regularity properties.
This has rigorously been proven for the case of the geometrically nonlinear two-well problem \cite{DM1}, \cite{DM2}, the geometrically nonlinear three-well problem in three dimensions (the ``cubic-to-tetragonal phase transformation") \cite{K}, \cite{CDK} and (under additional assumptions) for the geometrically linear six-well problem (the ``cubic-to-orthorhombic phase transformation") \cite{R16}. In these works it has been shown that on the one hand convex integration solutions exist, if the deformation gradient is only assumed to be $L^{\infty}$ regular. If on the other hand, the deformation gradient is $BV$ regular (or a replacement of this), then solutions are very rigid and for most constant matrices $M$ the analogue of (\ref{eq:incl}) does not possess a solution.\\
Thus, convex integration solutions cannot exist at $BV$ regularity for the deformation gradient; at this regularity solutions are \emph{rigid}. At $L^{\infty}$ regularity they are however \emph{flexible} and a multitude of solutions exist. Similarly as in the related (though much more complicated) situation of the Onsager conjecture for Euler's equations \cite{S}, \cite{DS16} or the situation of isometric embeddings \cite{CoS}, it is hence natural to ask whether there is a regularity threshold, which distinguishes between the rigid and the flexible regime.\\
It is the purpose of this article to make a first, very modest step into the
understanding of this dichotomy by analyzing the $W^{s,p}$ regularity of a
(known) convex integration scheme in an as simple as possible model case.

\subsection{Main ideas}
In our construction of solutions to the differential inclusion \eqref{eq:incl} we follow the ideas of Müller and {\v{S}}ver{\'a}k \cite{MS} (in the version of \cite{Otto}) and argue by an iterative convex integration algorithm. For the hexagonal-to-rhombic transformation this is particularly simple, since the laminar convex hull equals the convex hull of the wells and since all matrices in the convex hull are symmetrized rank-one-connected with the wells (c.f. Lemma \ref{lem:convex_hull}). As a consequence it is possible to construct piecewise affine solutions (in the language of \cite{K1}, Chapter 4). This simplifies the convergence of the iterative construction drastically. It is one of the reasons for studying the hexagonal-to-rhombic phase transformation as a model problem.\\
Yet, in spite of the (relative) simplicity of obtaining \emph{convergence} of the iterative construction to a solution of (\ref{eq:incl}) and hence of showing \emph{existence}, substantially more care is needed in addressing \emph{regularity}. In this context we argue by an interpolation result (c.f. Theorem \ref{thm:interpol} and Proposition \ref{prop:reg}): While our approximating deformations $u_k:\R^2 \rightarrow \R^2$ are such that the $BV$ norms of the iterations increase (exponentially), the $L^1$ norm of their difference decreases exponentially. If the threshold $\theta_0>0$ is chosen appropriately, the $W^{s,p}$ norm for $0<sp< \theta_0$ is controlled by an interpolation of the $BV$ and the $L^1$ norms, which can be balanced to be uniformly bounded. 
To ensure this,
we have to make the iterative algorithm \emph{quantitative} in several ways:
\begin{itemize}
\item[(i)] \emph{Tracking the error in strain space.} In order to iterate the convex integration construction, it is crucial not to leave the interior of the convex hull of $K$ in the iterative modification steps. In \emph{qualitative} convex integration algorithms, it suffices to use 
errors, which become arbitrarily small and to invoke the openness of $\intconv(K)$. As the admissible error in strain space is however coupled to the length scales of the convex integration constructions (c.f. Lemma \ref{lem:conti_deformed}) and as these in turn are directly reflected in the solutions' regularity properties, in our \emph{quantitative} algorithm we have to keep track of the errors in strain space very carefully. Here we seek to maximize the possible length scales (and hence the error) without leaving $\intconv(K)$ in each iteration step. This leads to the distinction of various possible cases (the ``stagnant", the ``push-out", the ``parallel" and the ``rotated" case, c.f. Notation \ref{nota:conti}, Definition \ref{defi:parallel_rot} and Algorithm \ref{alg:construction}). In these we quantitatively prescribe the admissible error according to the given geometry in strain space.
\item[(ii)] \emph{Controlling the skew part without destroying the structure of (i).} Seeking to construct $W^{1,\infty}$ solutions, we have to control the skew part of our construction. Due to the results of Kirchheim, it is known that this is generically possible (c.f. \cite{K1}, Chapter 3). However, in our quantitative construction, we cannot afford to arbitrarily change the direction of the rank-one connection, which is chosen in the convex integration algorithms, at an arbitrary iteration step. This would entail $BV$ bounds, which could not be compensated by the exponentially decreasing $L^1$ bounds in the interpolation argument. Hence we have to devise a detailed description of controlling the skew part (c.f. Algorithm \ref{alg:skew}). 
\item[(iii)] \emph{Precise covering construction.} In order to carry out our convex integration scheme we have to prescribe an iterative covering of our domain by constructions, which successively modify a given gradient. As our construction in Lemma \ref{lem:conti_deformed} relies on triangles, we have to ensure that there is a class of triangles, which can be used for these purposes (c.f. Section \ref{sec:covering}). In particular, we have to quantitatively control the overall perimeter (which can be viewed as a measure of the BV norm of $\nabla u_k$) of the covering at a given iteration step of the convex integration algorithm. This crucially depends on the specific case (``rotated" or ``parallel"), in which we are in.
\end{itemize}

\subsection{Organization of the article}
The remainder of the article is organized as follows: After briefly collecting
preliminary results in the next section (interpolation results, results on the
convex hull of the hexagonal-to-rhombic phase transition), in Section
\ref{sec:Conti} we begin by
describing the convex integration scheme, which we employ. Here we first recall the main ingredients of the qualitative
scheme (Section \ref{sec:conv}) and then introduce our more quantitative
algorithms in Sections \ref{sec:alg}-\ref{sec:skew}. As this algorithm crucially
relies on the existence of an appropriate covering, we present an explicit
construction of this in Section \ref{sec:covering}. Here we also address quantitative covering estimates for the perimeter and the volume. The ingredients from Sections \ref{sec:Conti} and \ref{sec:covering} are then combined in Section \ref{sec:quant}, where we prove Theorem \ref{thm:main} for a specific class of domains. In Section \ref{sec:generaldomains} we explain how this can be generalized to arbitrary Lipschitz domains. Finally, in the Appendix, Section \ref{sec:append}, we recall a symmetry based construction for a solution to (\ref{eq:incl}) with $M=0$ with much better regularity properties.

\section{Preliminaries}
In this section we collect preliminary results, which will be relevant in the sequel. We begin by stating the interpolation results of \cite{CDDD03}, on which our $W^{s,p}$ bounds rely. Next, in Section \ref{sec:hex} we recall general facts on matrix space geometry and in particular apply this to the hexagonal-to-rhombic phase transformation and its convex hulls.

\subsection{An interpolation inequality and Sickel's result}
\label{sec:interpol}

Seeking to show higher Sobolev regularity for convex integration solutions, we rely on the characterization of $W^{s,p}$ Sobolev functions. Here we recall the following two results on an interpolation
characterization \cite{CDDD03} and on a geometric characterization of the regularity of characteristic functions \cite{Si}:

\begin{thm}[Interpolation with BV, \cite{CDDD03}]
\label{thm:interpol}
We have the following interpolation results:
\begin{itemize}
\item[(i)] Let $p\in[2,\infty)$ and assume that
$\frac{1}{q} = \frac{1-\theta}{p} + \theta$ for some $\theta \in(0,1)$. Then
\begin{align}
\label{eq:int}
\|u\|_{W^{\theta,q}(\R^{n})} \leq C \|u\|_{L^p(\R^{n})}^{1-\theta}\|u\|_{BV(\R^n)}^{\theta} .
\end{align}
\item[(ii)] Let $p\in(1,2]$ and let $\frac{1}{q} = \frac{1-\theta}{p} + \theta$ for some $\theta \in (0,1)$. Let further $(\theta_1, q_1) \in (0,1)\times (1,\infty)$ be such that 
\begin{align*}
\frac{1}{q_1} &= \frac{1-\theta_1}{2} + \theta_1,\\
(\theta, q^{-1}) &= \tau (0,1-) + (1-\tau)(\theta_1, q_1^{-1}),
\end{align*}
for some $\tau \in (0,1)$, where $1-$ denotes an arbitrary positive number slightly less than $1$. Then,
\begin{align}
\label{eq:int_a}
\| u \|_{W^{\theta, q}(\R^n )} 
\leq C \left(\|u\|_{L^{1+}(\R^n)}^{\frac{\tau}{1-\theta}} \|u\|_{L^2(\R^n)}^{1-\frac{\tau}{1-\theta}} \right)^{1-\theta} \|u\|_{BV(\R^n)}^{\theta},
\end{align}
with $1+ := (1-)^{-1}$.
\end{itemize}
\end{thm}

Before proceeding to the proof of Theorem \ref{thm:interpol}, we present an immediate corollary of it: For functions, which are ``essentially" characteristic functions, we obtain the following unified result:

\begin{cor}
\label{cor:int}
Let $u:\R^n \rightarrow \R^n$ be a function, such that
\begin{align}
\label{eq:char_a}
\|u\|_{L^{\infty}(\R^n)} <\infty \mbox{ and } |u(x)| \geq c_0>0 \mbox{ for a.e. } x \in \supp(u).
\end{align} 
Then, for any $p\in(1,\infty)$ we have that
\begin{align}
\label{eq:int__c}
\|u\|_{W^{\theta,q}(\R^n)} \leq C \left( \frac{\|u\|_{L^{\infty}(\R^n)}}{c_0} \right)^{\left(1-\frac{1}{p}\right)(1-\theta)} \|u\|_{L^p(\R^n)}^{1-\theta} \|u\|_{BV(\R^n)}^{\theta},
\end{align}
where $\frac{1}{q} = \frac{1-\theta}{p} + \theta$ and $\theta \in(0,1)$.
\end{cor}

In the sequel, we will mainly rely on Corollary \ref{cor:int}, since in our applications (e.g. in Propositions \ref{prop:reg}, \ref{prop:reg1}), we will mainly deal with functions, which are ``essentially" characteristic functions.

\begin{figure}[t]
\includegraphics[width=0.6 \textwidth, page=40]{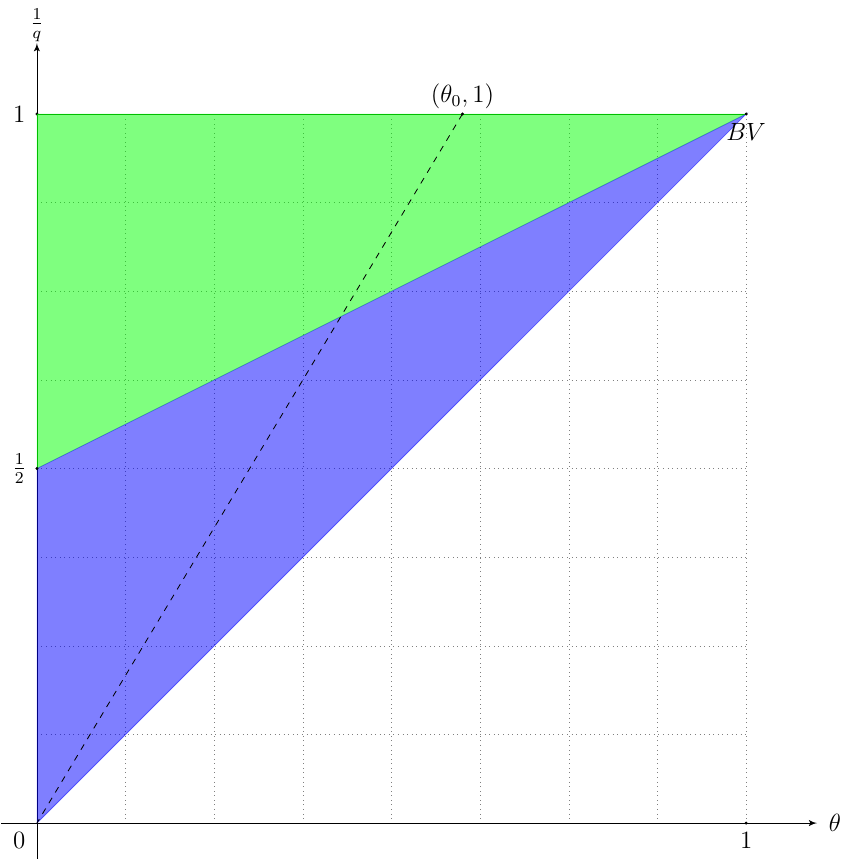}
\caption{For functions, which are ``essentially" characteristic functions in the sense that condition (\ref{eq:char_a}) of Corollary \ref{cor:int} holds, we obtain the interpolation inequality (\ref{eq:int__c}), which is valid in the whole coloured region in the figure (green and blue). Here the blue region is already covered in Theorem \ref{thm:interpol} (i). In order to also obtain the green region, we have to be able to simplify the statement of (\ref{eq:int_a}), which in general is only valid for functions, which are ``essentially" characteristic functions. In our application of Corollary \ref{cor:int} (c.f. Proposition \ref{prop:reg}, \ref{prop:reg1}), we will restrict to the region to the left of the dashed line. Remark \ref{rmk:prod} shows that having a bound for the product of the right hand side of (\ref{eq:int__c}) for a \emph{specific} value $(\theta_0,1)$ already allows to deduce a bound for \emph{all} exponents $(\tilde{\theta},q)$ on the associated dashed line connecting $(\theta_0,1)$ with $(0,
\infty)$.}
\label{fig:interpol}
\end{figure}

\begin{proof}[Proof of Corollary \ref{cor:int}]
By virtue of Theorem \ref{thm:interpol} (i) and equation (\ref{eq:int}), it suffices to consider the regime, in which $p\in(1,2)$. In this case the statement follows from a combination of equation  \eqref{eq:int_a} and the fact that for functions satisfying \eqref{eq:char_a} we have
\begin{align}
\label{eq:Lp}
 \|u\|_{L^{p_1}(\R^n)}^{\sigma} \|u\|_{L^{p_2}(\R^n)}^{1-\sigma}
 \leq \left( \frac{\|u\|_{L^{\infty}(\R^n)}}{c_0} \right)^{1-\frac{1}{r}} \|u\|_{L^r(\R^n)},
\end{align}
for $1< p_1 \leq r \leq p_2$, $r^{-1}= \sigma p_1^{-1} + (1-\sigma)p_2^{-1} $ and $\sigma \in (0,1)$.
We postpone a proof of (\ref{eq:Lp}) to the end of this proof, and observe first that it indeed suffices to show \eqref{eq:Lp} to conclude the claim of (\ref{eq:int__c}). To this end, we note that the exponents in \eqref{eq:int_a} obey the relation
\begin{align*}
\frac{1}{p} = \frac{1}{1+} \frac{\tau}{1-\theta} + \frac{1}{2} \left(1- \frac{\tau}{1-\theta}\right).
\end{align*}
This in turn is a consequence of the three identitites
\begin{align*}
\frac{1}{q} = \frac{1}{1+} \tau + (1-\tau) \frac{1+\theta_1}{2},\ 
\frac{1}{q} = \frac{1-\theta}{p} + \theta, \ \theta = (1-\tau)\theta_1.
\end{align*}
Here we note that $\frac{\tau}{1-\theta}=1-\frac{(1-\tau)(1-\theta_1)}{1-\theta}\in(0,1)$. Hence (\ref{eq:Lp}) (applied to $r=p$, $p_1 = 1+$, $p_2 = 2$ and $\sigma = \frac{\tau}{1-\theta}$) together with (\ref{eq:int_a}) yields the claim of (\ref{eq:int__c}).\\
It thus remains to prove (\ref{eq:Lp}).
To this end, we observe that for any $r \in [1,\infty]$
\begin{align*}
\|u\|_{L^r(\R^n) } &\geq c_0^{1-\frac{1}{r}} \|u\|_{L^1(\R^n)}^{\frac{1}{r}},\\
\|u\|_{L^r(\R^n) } &\leq C_1^{1-\frac{1}{r}} \|u\|_{L^1(\R^n)}^{\frac{1}{r}},
\end{align*}
where for abbreviation we have set $C_1:=\|u\|_{L^{\infty}(\R^n)}$. With this we infer
\begin{align*}
\|u\|_{L^{p_1}(\R^n)}^{\sigma} \|u\|_{L^{p_2}(\R^n)}^{1-\sigma}
&\leq C_1^{\left( 1-\frac{1}{p_1}\right) \sigma} \|u\|_{L^1(\R^n)}^{\frac{\sigma}{p_1}}
C_1^{\left(1-\frac{1}{p_2}\right) (1-\sigma)} \|u\|_{L^1(\R^n)}^{\frac{1-\sigma}{p_2}}\\
& \leq C^{1-\frac{1}{r}}_1 \|u\|_{L^1(\R^n)}^{\frac{1}{r}}
 \leq C^{1-\frac{1}{r}}_1 c_0^{\frac{1}{r}-1} \|u\|_{L^r(\R^n)}\\
& = \left(\frac{C_1}{c_0} \right)^{1-\frac{1}{r}}  \|u\|_{L^r(\R^n)}.
\end{align*}
This concludes the argument.
\end{proof}

After this discussion, we come to the proof of Theorem \ref{thm:interpol}:

\begin{proof}[Proof of Theorem \ref{thm:interpol}]
If $p\geq 2$, the interpolation result is a special case of Theorem 1.4 in \cite{CDDD03} (where in the notation of \cite{CDDD03} we have chosen $s=0$, $t=\theta$): Indeed, for $\gamma< 1-\frac{1}{n}$ and $(s,p)$ satisfying $(s-1)p^{\ast} \frac{1}{n} = \gamma -1$ with $p^{\ast}$ being the dual exponent of $p$, the estimate in Theorem 1.4 from \cite{CDDD03} reads
\begin{align}
\label{eq:int_gen}
\|u\|_{B^{t}_{q,q}(\R^n)} \leq C\|u\|_{B^s_{p,p}(\R^n)}^{1-\theta}
\|u\|_{BV(\R^n)}^{\theta},
\end{align}
where
\begin{align*}
\frac{1}{q} = \frac{1-\theta}{p} + \theta, \ t=(1-\theta) s + \theta.
\end{align*}
We note that in the setting of Theorem \ref{thm:interpol} the estimate (\ref{eq:int_gen}) is applicable, as in the notation of \cite{CDDD03} and with dimension $n$ we have that $\gamma := -\frac{p}{p-1}\frac{1}{n} + 1= 1- \frac{1}{n} - \frac{1}{p-1}\frac{1}{n}< 1- \frac{1}{n}$, which implies the validity of (\ref{eq:int}). The simplification from (\ref{eq:int_gen}) to (\ref{eq:int}) is then a consequence of the facts that 
\begin{itemize}
\item for $s\notin \Z$ we have $W^{s,p}(\R^n) = B_{p,p}^s(\R^n)$ (c.f. \cite{Cohen}, \cite{BM12}),
\item and for $p\geq 2$ the embedding
$L^p(\R^n) \hookrightarrow B_{p,p}^0(\R^n)$ is valid (Theorem 2.41 in \cite{BCD}). 
\end{itemize}
This concludes the argument for (i).
\\
To obtain (ii), we combine (i) with an additional interpolation inequality, which becomes necessary, as the inclusion $L^p(\R^n) \hookrightarrow B_{p,p}^0(\R^n)$ is no longer valid for $p\in(1,2)$. 
Hence, we rely on the following interpolation estimate (c.f. Lemma 3 in \cite{BM12})
\begin{align}
\label{eq:Fsrl}
\|u\|_{\tilde{F}^{s}_{r,l}(\R^n)} \leq C \|u\|_{\tilde{F}^{s_0}_{p_0,l_0}(\R^n)}^{\tau} \|u\|_{\tilde{F}^{s_1}_{p_1,l_1}(\R^n)}^{1-\tau},
\end{align}
which is valid for $-\infty< s_0 <s_1<\infty$, $0<q_0,q_1\leq \infty$, $0<p_0,p_1\leq \infty$, $0< \tau < 1$ with 
\begin{align*}
s = \tau s_0 + (1-\tau) s_1, \ r^{-1} = \tau p_0^{-1} + (1-\tau) p_1^{-1}.
\end{align*}
Here the spaces $\tilde{F}^{s}_{r,l}$ denote the (modified) Triebel-Lizorkin spaces from \cite{BM12}.
The main advantage of the estimate \eqref{eq:Fsrl}, which goes back to Oru \cite{Or98}, is that there are no conditions on the relations between $l,l_1,l_2$ in this estimate. In particular, we can choose $l_1=2$, $l_2 = p_0$ and $l=r$. Using that
\begin{itemize}
\item $\tilde{F}^{s}_{r,2}(\R^n) = L^{s,r}(\R^n)$ for $s\in \R$, $1<r<\infty$ and that for this range $L^{0,r}(\R^n)=L^{r}(\R^n)$, 
\item $\tilde{F}^{s}_{r,r}(\R^n)= W^{s,r}(\R^n)$ for $0<s<\infty$, $s\notin \Z$, $1\leq p <\infty$,
\end{itemize}
we can simplify \eqref{eq:Fsrl} to yield
\begin{align}
\label{eq:Wsp_int}
\|u\|_{W^{s,r}(\R^n)} \leq C \|u\|_{L^{p_0}(\R^n)}^{\tau} \|u\|_{W^{s_1,p_1}(\R^n)}^{1-\tau},
\end{align}
which is valid for $0< s_1 < \infty$, $1<p_0 < \infty$, $1<p_1 <\infty$ with
\begin{align*}
s = (1-\tau) s_1, \ r^{-1} = \tau p_0^{-1} + (1-\tau) p_1^{-1}.
\end{align*}
We apply $(\ref{eq:Wsp_int})$
 with $p_0=1+$,  $s=\theta$, $r=q$ and $(s_1,p_1)=(\theta_1,q_1)$
  lying on the boundary of the interpolation region from (i) (c.f. the blue region in Figure \ref{fig:interpol}), i.e. 
  \begin{align*}
  \left(\frac{1}{q},\theta\right)&= \tau \left(1-,0\right) + (1-\tau)\left(\frac{1}{q_1},\theta_1\right), \\
    \left(\frac{1}{q_1},\theta_1\right)&= (1-\theta_1) \left(\frac{1}{2},0\right) + \theta_1(1,1).
  \end{align*}
  In particular, these equations uniquely determine $\tau \in (0,1)$.
  Hence, we obtain
  \begin{align*}
    \|u\|_{W^{s,q}} \lesssim \|u\|_{L^{1+}}^{\tau} \|u\|_{W^{\theta_1,q_1}}^{1-\tau}  \lesssim   \|u\|_{L^{1+}}^{\tau} \|u\|_{L^2}^{(1-\tau)(1-\theta_1)} \|u\|_{BV}^{(1-\tau)\theta_1}.
  \end{align*}
  We conclude the proof of (ii) by noting that $(1-\tau)\theta_1=\theta$ and that
  \begin{align*}
    0<\frac{(1-\tau)(1-\theta_1)}{1-\theta}= 1-
  \frac{\tau}{1-\theta}.
  \end{align*}
\end{proof}

As an alternative to the interpolation approach, a more geometric criterion for regularity is given by Sickel:

\begin{thm}[Sickel, \cite{Si}] 
\label{thm:Sickel}
Let $\theta \in(0,1)$, $q\in[1,\infty)$ and let $E\subset \R^{n}$ be a bounded set satisfying
\begin{align}
\label{eq:Sickel}
\int\limits_{0}^{1}\delta^{-\theta q}|(\partial E_i)_{\delta}| \frac{d \delta}{\delta}<\infty,
\end{align}
where 
\begin{align*}
(\partial E)_{\delta}:= \{ x \in E: \dist(x,\partial E)\leq \delta\}.
\end{align*}
Then, $\chi_{E}\in W^{\theta,q}(\R^{n})$.
\end{thm}

Although this theorem provides good geometric intuition and could have been used as an alternative means of proving Theorem \ref{thm:main}, we do not pursue this further in the sequel, but postpone its discussion to future work.

\begin{rmk}
\label{rmk:prod}
We note that the estimate \eqref{eq:Sickel} in Theorem \ref{thm:Sickel} yields a condition on the \emph{product} $\theta q>0$, while, at first sight, Theorem \ref{thm:interpol} and Corollary \ref{cor:int} pose a restriction on $\theta,q$ individually. 
As we are dealing with bounded (or even characteristic) functions, we however observe that it is also possible to obtain an analogous condition on the product $\theta q$ in Theorem \ref{thm:interpol} and Corollary \ref{cor:int}:
Indeed, assume that $u\in L^{\infty}(\R^2)$ is such that for some $\theta_0 \in (0,1)$ the product
\begin{align}
\label{eq:L1_bound}
\|u\|_{L^1(\R^2)}^{1-\theta_0}\| u\|_{BV(\R^2)}^{\theta_0}
\end{align}
is bounded. Then, we claim that for 
\begin{align}
\label{eq:spec_exp}
q\in (1,\infty), \ \tilde{\theta}:=\theta_0 q^{-1} \mbox{ and for } p= \frac{1-\tilde{\theta}}{\tilde{\theta}} \frac{\theta_0}{1-\theta_0},
\end{align} 
also the product
\begin{align*}
\|u\|_{L^p(\R^2)}^{1-\tilde{\theta}}\| u\|_{BV(\R^2)}^{\tilde{\theta}}
\end{align*}
is bounded. To derive this, we first observe that the $L^{\infty}$ bound for $u$ allows us to infer that for any $p\in(1,\infty)$
\begin{align}
\label{eq:Linf}
\|u\|_{L^p(\R^2)} \leq \|u\|_{L^{\infty}(\R^2)}^{1-\frac{1}{p}} \|u\|_{L^1(\R^2)}^{\frac{1}{p}}.
\end{align}
As a consequence, we deduce that
\begin{equation}
\label{eq:control_int}
\begin{split}
\|u\|_{L^p(\R^2)}^{1-\tilde{\theta}}\| u\|_{BV(\R^2)}^{\tilde{\theta}}
&\leq \|u\|^{(1-\frac{1}{p})(1-\tilde{\theta})}_{L^{\infty}(\R^2)} \|u\|_{L^1(\R^2)}^{\frac{1-\tilde{\theta}}{p}}\| u\|_{BV(\R^2)}^{\tilde{\theta}} \\
& = \|u\|^{1-\frac{\tilde{\theta}}{\theta_0}}_{L^{\infty}(\R^2)} \left( \|u\|_{L^1(\R^2)}^{1-\theta_0}\| u\|_{BV(\R^2)}^{\theta_0} \right)^{\frac{\tilde{\theta}}{\theta_0}}.
\end{split}
\end{equation}
Here we have made use of the specific choices of exponents from (\ref{eq:spec_exp}) and the boundedness of $u$, which allowed us to invoke (\ref{eq:Linf}). This concludes the argument for the claim.\\
Thus, relying on the bound (\ref{eq:control_int}), we infer that given a bound on (\ref{eq:L1_bound}), we obtain that for all exponents $q,\tilde{\theta},p$ from (\ref{eq:spec_exp})
\begin{align}
\label{eq:higher}
\|u\|_{W^{\tilde{\theta},q}(\R^n)} \leq C \|u\|^{1-\frac{\tilde{\theta}}{\theta_0}}_{L^{\infty}(\R^2)} \left( \|u\|_{L^1(\R^2)}^{1-\theta_0}\| u\|_{BV(\R^2)}^{\theta_0} \right)^{\frac{\tilde{\theta}}{\theta_0}}.
\end{align}
Here we applied Theorem \ref{thm:interpol} (or Corollary \ref{cor:int}), for which we noted that the respective exponents are admissible.
On the one hand, this is the desired analogue of the condition from Theorem \ref{thm:Sickel} and allows us to obtain a whole family of $W^{\theta,q}$ bounds for $u$, where $\theta q < \theta_0$. On the other hand, it shows that although $p=1$ is \emph{not admissible} in Theorem \ref{thm:interpol} and Corollary \ref{cor:int}, for our purposes, it still suffices to consider the case $p=1$ and to prove a control for (\ref{eq:L1_bound}), which then gives the full range of expected exponents in the form of the estimate (\ref{eq:higher}). 
\end{rmk}

\begin{rmk}[Fractal packing dimension]
\label{rmk:frac}
Following Sickel \cite{Si}, Proposition 3.3 (c.f. also \cite{JM96}, Theorem 2.2) we remark that for a characteristic function its $W^{s,p}$ regularity has direct consequences on the packing dimension (c.f. \cite{JM96}, \cite{Mat}), which we denote by $\dim_P$, of its boundary: If for some set $E \subset \R^{n}$ its characteristic function $\chi_E$ satisfies $\chi_{E}\in W^{s,p}(\R^n)$ for some $s>0$ and $1\leq p <\infty$, then
\begin{align*}
\dim_P(S_{\delta}(\partial E)) \leq \min\{n,n-sp+\delta\}.
\end{align*}
Here 
\begin{align*}
S_{\delta}(\partial E):= \left\{
x \in \partial E: \ \exists \mu >0 \mbox{ such that } 
\forall \epsilon, \ 0 < \epsilon \leq 1, \ \exists A_{\epsilon}, A_{\epsilon}' \mbox{ satisfying } \right. \\
\left. A_{\epsilon} \subset B_{\epsilon}(x)\cap E, \ A_{\epsilon}' \subset B_{\epsilon}(x)\cap E^{c} \mbox{ and } 
|A_{\epsilon}||A_{\epsilon}'| \geq \mu \epsilon^{2n+\delta}
\right\},
\end{align*}
$B_{\epsilon}(x):=\{x'\in \R^n: |x-x'|\leq \epsilon\}$ and $E^c$ denotes the complement of $E$.
\end{rmk}

\subsection{Matrix space geometry}
\label{sec:hex}
Before discussing our convex integration scheme, we recall some basic notions and properties of the hexagonal-to-rhombic phase transformation, which we will use in the sequel. \\

We begin by introducing notation for the symmetric and antisymmetric part of two matrices.

\begin{defi}[Symmetric and antisymmetric parts]
\label{defi:sym}
Let $M\in \R^{n\times n}$. We denote the uniquely determined symmetric and antisymmetric parts of $M$ by
\begin{align*}
M= e(M) + \omega(M), \ e(M) := \frac{1}{2}(M^T + M), \ 
\omega(M) := \frac{1}{2}(M - M^T).
\end{align*}
\end{defi}

\subsubsection{Lamination convexity notions}

Relying on the notation from Definition \ref{defi:sym}, in the sequel we discuss the different notions of lamination convexity. Here we distinguish between the usual \emph{lamination convex hull} (defined by successive rank-one iterations) and the \emph{symmetrized lamination convex hull} (defined by successive symmetrized rank-one iterations):

\begin{defi}[Lamination convex hull, symmetrized lamination convex hull] We define the following notions of lamination convex hulls:
\begin{itemize}
\item[(i)] Let $U\subset \R^{n\times n}$. Then we set
\begin{align*}
\mathcal{L}^0(U)&:= U,\\
\mathcal{L}^k(U) &:= \{M\in \R^{2\times 2}: M= \lambda A+ (1-\lambda) B \mbox{ with } A-B = a \otimes n, \lambda\in[0,1],\\
& \quad \quad A,B \in \mathcal{L}^{k-1}(U)\}, \ k\geq 1,\\
U^{lc}&:= \bigcup\limits_{k=0}^{\infty} \mathcal{L}^k(U).
\end{align*} 
We refer to $U^{lc}$ as the \emph{laminar convex hull of $U$} and to $\mathcal{L}^k(U)$ as the \emph{laminates of order at most $k$}.
\item[(ii)] Let $U\subset \R^{n \times n}_{sym}$. Then we define
\begin{align*}
\mathcal{L}^0_{sym}(U)&:= U,\\
\mathcal{L}^k_{sym}(U) &:= \{M\in \R^{2\times 2}: M= \lambda A+ (1-\lambda) B \mbox{ with } A-B = a \odot n, \lambda\in[0,1],\\
& \quad \quad A,B \in \mathcal{L}^{k-1}_{sym}(U)\}, \ k\geq 1,\\
U^{lc}_{sym}&:= \bigcup\limits_{k=0}^{\infty} \mathcal{L}^k_{sym}(U).
\end{align*} 
Here $a\odot b:= \frac{1}{2}(a\otimes b + b\otimes a)$.
We refer to $U^{lc}_{sym}$ as the \emph{symmetrized laminar convex hull of $U$} and to $\mathcal{L}^k_{sym}(U)$ as the \emph{symmetrized laminates of order at most $k$}.
\item[(iii)] We denote the \emph{convex hull} of a set $U\subset\R^{m}$ by $\conv(U)$.
\end{itemize}
\end{defi}

\begin{rmk}
We note that if $U \subset \R^{n\times n}$ or $U\subset \R^{n\times n}_{sym}$ is (relatively) open, then also $U^{lc}$ or $U^{lc,sym}$ is (relatively) open.
\end{rmk}

\begin{lem}[Convex hull = laminar convex hull]
\label{lem:convex_hull}
Let $K$ be as in (\ref{eq:K3}). Then 
\begin{align*}
K_{sym}^{lc} = \conv(K) = \mathcal{L}^2_{sym}(K).
\end{align*}
Moreover, each element $e\in \intconv(K)$ is symmetrized rank-one connected with each element in $K$.
\end{lem}

\begin{proof}
The first point follows from an observation of Bhattacharya (c.f. \cite{B} and also Lemma 4 in \cite{R16}). The second point either follows from a direct calculation or by an application of Lemma \ref{lem:rk1} below.
\end{proof}

The following lemma establishes a relation between rank-one connectedness and symmetrized rank-one connectedness. It in particular shows that in two dimensions all symmetric trace-free matrices are pairwise symmetrized rank-one connected.

\begin{lem}[Rank-one vs symmetrized rank-one connectedness]
\label{lem:rk1}
Let $e_1,e_2 \in \R^{n\times n}_{sym}$ with $\tr(e_1)=0=\tr(e_2)$. Then the following statements are equivalent:
\begin{itemize}
\item[(i)] There exist vectors $a\in \R^{n}\setminus\{0\}, n \in S^{n-1}$ such that 
\begin{align*}
e_1-e_2 = a\odot n.
\end{align*}
\item[(ii)] There exist matrices $M_1,M_2 \in \R^{n\times n}$ and vectors $a\in \R^{n}\setminus\{0\}, n \in S^{n-1}$ such that
\begin{align*}
M_1 - M_2 &= a\otimes n,\\
e(M_1) &= e_1, e(M_2) = e_2.
\end{align*}
\item[(iii)] $\rank(e_{1}-e_2)\leq 2$.
\end{itemize}
\end{lem}

\begin{proof} We refer to \cite{R16}, Lemma 9 for a proof of this statement.
\end{proof}

This lemma allows us to view symmetrized rank-one connectedness essentially as equivalent to rank-one connectedness. 

\subsubsection{Skew parts}

We discuss some properties of the associated skew symmetric parts of rank-one connections, which occur between points in the interior of $\intconv(K)$. To this end, we introduce the following identification:

\begin{nota}[Skew symmetric matrices]\label{not:skew} As the two dimensional skew symmetric matrices are all of the form
\begin{align*}
S=\begin{pmatrix}
0 & \tilde{\omega} \\ 
-\tilde{\omega} & 0 
\end{pmatrix} \mbox{ for some } \tilde{\omega} \in \R,
\end{align*}
we use the mapping $S\mapsto \tilde{\omega}$ to identify $\Skew(2)$ with $\R$.
We define an ordering on $\Skew(2)$ by the corresponding ordering on $\R$, i.e. 
\begin{align*}
S_1 = \begin{pmatrix}
0 & \tilde{\omega}_1 \\ 
-\tilde{\omega}_1 & 0 
\end{pmatrix} \leq S_2 = \begin{pmatrix}
0 & \tilde{\omega}_2 \\ 
-\tilde{\omega}_2 & 0 
\end{pmatrix}
\end{align*}
if $\tilde{\omega}_1 \leq \tilde{\omega}_2$.
\end{nota}

We begin by estimating the symmetric and skew-symmetric parts of a symmetrized rank-one connection:

\begin{lem}
  \label{lem:skewcontrol2}
  Let $ a \in \R^{2}, n \in S^{1}$ with $a \cdot n=0$. Then
  \begin{align*}
    \|a \odot n \| =|a|/2= \|\omega(a \otimes n)\|.
  \end{align*}
  Here $\|\cdot\|$ denotes the spectral norm, i.e. $\|A\|:= \sup\limits_{|e|=1}\{|e\cdot A e|\}$, where $|\cdot|$ denotes the $\ell_2$ norm.
\end{lem}

\begin{proof}
  Since $a \bot n$, we obtain that
  \begin{align*}
    (a \odot n) n  &=a/2 = \frac{|a|}{2} \frac{a}{|a|} ,\\
    (a \odot n) \frac{1}{|a|} a &= \frac{|a|}{2} n.
  \end{align*}
  As $n, \frac{1}{|a|}a$ forms an orthonormal basis, this shows that $\|a \odot n\|=|a|/2$.

  Similarly, we obtain that 
  \begin{align*}
    \omega(a \otimes n) n &= \frac{|a|}{2} \frac{a}{|a|}, \\
    \omega(a \otimes n) \frac{1}{|a|}a &= - \frac{|a|}{2} n, 
  \end{align*}
  and hence $\| \omega(a \otimes n)\|=|a|/2$.
\end{proof}

Using the previous result, we can control the size of the skew part which occurs in rank-one connections with $K$:

\begin{lem}
  \label{lem:skewcontrol}
For all matrices $N$ with $e(N)\in \intconv(K)$ and with
$N$ being rank-one connected with a matrix $e^{(j)} \in K$ it holds
\begin{align*}
\|\omega(N)\| \leq 10.
\end{align*}
\end{lem}

\begin{proof}
For each $e\in \intconv(K)$ there are exactly two matrices 
$M_{e,i}^{\pm}$ such that  
\begin{align*}
e(M_{e,i}^{\pm})= e \mbox{ and } \rank(M_{e,i}^{\pm} - e^{(i)}) = 1 \mbox{ for } i\in\{1,2,3\}.
\end{align*}
Let $e-e^{(i)}= \frac{1}{2}(a \otimes n + n \otimes a)$ for some 
$a \in \R^{2}\setminus\{0\}, n \in S^{1}$.
Then, $\omega(M_{e,i}^{\pm})= \omega(M_{e,i}^{\pm}) - \omega(e^{(i)})$ is explicitly given by 
$\pm \frac{1}{2} (a \otimes n -n \otimes a)$. 
Thus, Lemma \ref{lem:skewcontrol2} implies
  \begin{align*}
    |a|/2= |(e-e^{(i)}) n| \leq \|e-e^{(i)}\|,  
  \end{align*}
  since $n \in S^{1}$ and $a \cdot n =0$ by the trace-free condition.
As $\conv(K)$ is a compact set, $e-e^{(i)}$ is uniformly
  bounded. Moreover, the diameter of $\conv(K)$ is less than five, which yields the desired bound.
\end{proof}

\subsubsection{Geometry of the hexagonal-to-rhombic phase transformation}
\label{sec:geo}

In this subsection, we discuss the specific matrix space geometry of the hexagonal-to-rhombic phase transformation. To this end we decompose each matrix of the form $\begin{pmatrix} \alpha & \beta \\ \beta & -\alpha \end{pmatrix}$ into a component in $v_1=\begin{pmatrix} 1 & 0 \\ 0 & -1 \end{pmatrix}$ and a component in $v_2=\begin{pmatrix} 0 & 1 \\ 1 & 0 \end{pmatrix}$ direction.

With this notation we make the following observations:

\begin{lem}
\label{lem:geo1}
Let $v_1, v_2 \in \R^{2\times 2}_{sym}$ be as above. Then,
\begin{align*}
\cos(\varphi)
\begin{pmatrix} 1 & 0 \\ 0 &- 1 \end{pmatrix} + \sin(\varphi) \begin{pmatrix} 0 & 1 \\ 1 & 0 \end{pmatrix}& = \begin{pmatrix} \cos(\varphi) & -\sin(\varphi) \\ \sin(\varphi) & \cos(\varphi) \end{pmatrix} \begin{pmatrix} 1 & 0 \\ 0 & -1 \end{pmatrix}\\
&= \begin{pmatrix} \cos(\varphi) & -\sin(\varphi) \\ \sin(\varphi) & \cos(\varphi) \end{pmatrix} v_1.
\end{align*}
Furthermore, we have that
\begin{align*}
\begin{pmatrix} \cos(\varphi) & -\sin(\varphi) \\ \sin(\varphi) & \cos(\varphi) \end{pmatrix}v_1 &= \begin{pmatrix}\cos(\frac{\varphi}{2}+\frac{\pi}{4}) \\ \sin(\frac{\varphi}{2}+\frac{\pi}{4}) \end{pmatrix}\otimes \begin{pmatrix} \sin(\frac{\varphi}{2} + \frac{\pi}{4}) \\ -\cos(\frac{\varphi}{2}+\frac{\pi}{4}) \end{pmatrix} \\
& \quad +  \begin{pmatrix} \sin(\frac{\varphi}{2} + \frac{\pi}{4}) \\ -\cos(\frac{\varphi}{2}+\frac{\pi}{4}) \end{pmatrix} \otimes \begin{pmatrix} \cos(\frac{\varphi}{2}+\frac{\pi}{4}) \\ \sin(\frac{\varphi}{2}+\frac{\pi}{4}) \end{pmatrix}.
\end{align*}
\end{lem}

\begin{proof}
Using the trigonometric identities
    \begin{align*}
      \cos(\phi)&=\sin(\phi+\frac{\pi}{2})= 2 \cos(\frac{\phi}{2}+\frac{\pi}{4})\sin(\frac{\phi}{2}+\frac{\pi}{4}),\\
      \sin(\phi)&=-\cos(\phi+\frac{\pi}{2})= \sin^2(\frac{\phi}{2}+\frac{\pi}{4})-\cos^2(\frac{\phi}{2}+\frac{\pi}{4}),
    \end{align*}
an immediate computation shows the claim.
\end{proof}

In other words, Lemma \ref{lem:geo1} allows us to identify all lines in matrix
space (through the origin) by their rotation angle. In particular, this
gives a simple description of the possible rank-one connections between the
energy wells (c.f. also Figure \ref{fig:normals} 
  (a)). In our application to the hexagonal-to-rhombic phase transformation we have to take into account that non-trivial differences of the matrices $e^{(1)}, e^{(2)}, e^{(3)}$ lie on the sphere of radius $\sqrt{3}$ in matrix space (with respect to the spectral norm), which yields slightly different normalization factors for $a$:

\begin{figure}
\centering
\includegraphics[width=4cm,page=41]{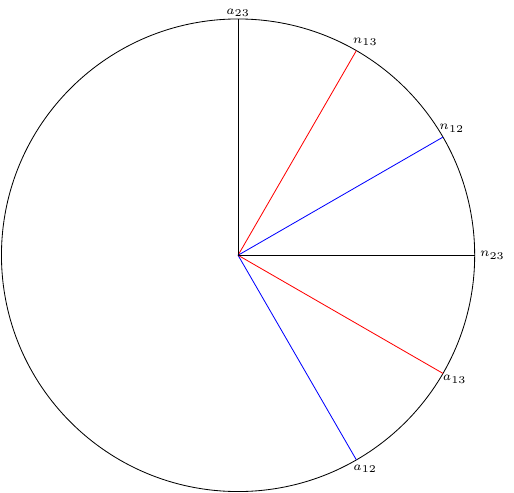}
\includegraphics[width=4cm, page=44]{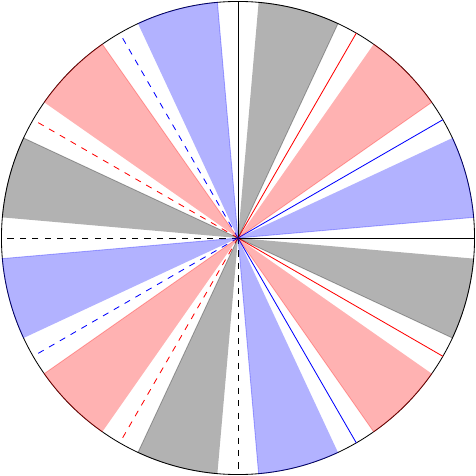}
\caption{The possible normals between the wells (c.f. Lemma \ref{lem:geo2}) (left) and the angles that arise in Lemma \ref{lem:angle_a} (right). The figure on the left depicts the possible normals between the wells, the respective pairs $(a_{ij},n_{ij})$ are marked in the same color. We note that
by symmetry it is also possible to pass from $(a,n)$ to $(-a,-n)$. After
normalizing appropriately, symmetry further allows to exchange the roles of $a$,
$n$.
The figure on the right depicts the normals, which arise in the decomposition of differences $e-e^{(i)}$ with $e^{(i)}\in K$ and $e\in C_d$ (c.f. Lemma \ref{lem:angle_a}). The colors represent the well, with which $e\in C_d$ is connected: Red corresponds to the cone at $e^{(1)}$, blue to the cone at $e^{(2)}$ and black to the cone at $e^{(3)}$.
As $C_d$ does not contain the
full convex hull of $K$ (c.f. Figure \ref{fig:star}), only vectors, which lie within the colored zones arise as
possible normals (in particular, there is a gap between these vectors and the vectors, which arise as decomposition of differences of the wells). 
}
\label{fig:normals}
\end{figure}

\begin{lem}
\label{lem:geo2}
Let $K$ be as in (\ref{eq:K3}). Then we have that
\begin{align*}
e^{(1)} - e^{(2)} &= a_{12}\odot n_{12},\\
e^{(1)}-e^{(3)} &= a_{13}\odot n_{13},\\
e^{(2)}-e^{(3)} &= a_{23}\odot n_{23},
\end{align*}
with (up to rotation symmetry by an angle of $\pi$)
\begin{align*}
a_{12} &:= -2\sqrt{3} \begin{pmatrix} \cos(\frac{5 \pi}{12} + \frac{\pi}{4}) \\ \sin(\frac{5 \pi}{12} + \frac{\pi}{4}) \end{pmatrix}
=\begin{pmatrix} \sqrt{3} \\ -3 \end{pmatrix}, \
n_{12}:= \begin{pmatrix} \sin(\frac{5 \pi}{12} + \frac{\pi}{4}) \\ -\cos(\frac{5 \pi}{12} + \frac{\pi}{4}) \end{pmatrix}
=\frac{1}{2}\begin{pmatrix} \sqrt{3} \\ 1 \end{pmatrix} ,\\
a_{13}&:= -2\sqrt{3}\begin{pmatrix} \cos(\frac{7 \pi}{12} + \frac{\pi}{4}) \\ \sin(\frac{7 \pi}{12} + \frac{\pi}{4}) \end{pmatrix}
= \begin{pmatrix} 3 \\ -\sqrt{3} \end{pmatrix}, \
n_{13}:= \begin{pmatrix} \sin(\frac{7 \pi}{12} + \frac{\pi}{4}) \\ -\cos(\frac{7 \pi}{12} + \frac{\pi}{4}) \end{pmatrix}
=\frac{1}{2}\begin{pmatrix} 1\\\sqrt{3} \end{pmatrix} ,\\
a_{23} &:= 2\sqrt{3} \begin{pmatrix} 0 \\ 1 \end{pmatrix} , \
n_{23} :=\begin{pmatrix} 1 \\ 0 \end{pmatrix} .
\end{align*}
\end{lem}

\begin{proof}
This is a consequence of Lemma \ref{lem:geo1} and the form of the matrices in (\ref{eq:K3}).
\end{proof}

With Lemma \ref{lem:geo1} at hand, we can also compute the possible (symmetrized) rank-one connections, which occur between each well and any possible matrix in $\conv(K)$:

\begin{lem}
\label{lem:geo3}
Let $e^{(i)}\in K$ and let $e\in \conv(K)$. Let $\va\in(-\pi,\pi]$ denote the angle from the decomposition from Lemma \ref{lem:geo1} for the matrix $\frac{e^{(i)}-e}{\|e^{(i)}-e\|}$, where $\|\cdot\|$ denotes the spectral matrix norm.
Then, 
\begin{align}
\label{eq:cone_angle}
\va \in \left\{ 
\begin{array}{ll}
(\frac{5\pi}{6}, \frac{7\pi}{6}) &\mbox{ if } i=1,\\
(-\frac{\pi}{2},-\frac{\pi}{6}) &\mbox{ if } i=2,\\
(\frac{\pi}{6},\frac{\pi}{2}) &\mbox{ if } i=3,\\
\end{array}
\right.
\end{align}
and
\begin{align*}
e^{(i)}- e = \|e^{(i)}-e\| a_i(\va) \otimes n_i(\va), 
\end{align*}
with 
\begin{align*}
a_i(\va) :=    \begin{pmatrix} \sin(\frac{\varphi}{2} + \frac{\pi}{4}) \\ -\cos(\frac{\varphi}{2}+\frac{\pi}{4}) \end{pmatrix}, \ 
n_{i}(\va) := \begin{pmatrix} \cos(\frac{\varphi}{2}+\frac{\pi}{4}) \\ \sin(\frac{\varphi}{2}+\frac{\pi}{4}) \end{pmatrix}.
\end{align*}
\end{lem}

\begin{proof}
This is a direct consequence of Lemma \ref{lem:geo1} and of the fact that the set $K$ forms an equilateral triangle in strain space.
\end{proof}

As an immediate consequence of Lemma \ref{lem:geo2} and Lemma \ref{lem:geo3}, we infer the following result, which is graphically illustrated in Figure \ref{fig:normals} (b): 

\begin{figure}[t]
\centering
\includegraphics[scale=0.7, page=45]{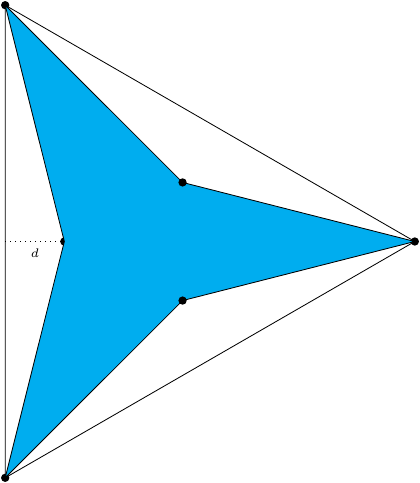}
\caption{The colored domain represents the interior of the star $C_d$ from Lemma \ref{lem:angle_a}. As not the full cone of rank-one directions between the wells are possible (c.f. \ref{eq:angle_cone1}), the difference in the angle between different cones is bounded strictly away from zero and $\pi$. }
\label{fig:star}
\end{figure}

\begin{lem}
\label{lem:angle_a}
Let $K$ be as in (\ref{eq:K3}) and let 
\begin{align*}
C_{d}:=\conv\{K, P_d, Q P_d Q^T, Q^2 P_d (Q^2)^T\},
\end{align*}
with $Q$ being a rotation by $\frac{\pi}{6}$ and
\begin{align*}
P_d := \frac{1-d}{2}\begin{pmatrix} -1 & 0 \\ 0 & 1 \end{pmatrix},
\end{align*}
where $d \in (0,\frac{1}{2})$.
Let $\hat{e}, \bar{e} \in C_d$. Suppose that
\begin{align*}
\hat{e}-e^{(i_1)} &= \|\hat{e}-e^{(i_1)} \| a_{i_1}\odot n_{i_1},\\
\bar{e}-e^{(i_2)} &= \|\hat{e}-e^{(i_2)} \| a_{i_2}\odot n_{i_2},
\end{align*}
with $e^{(i_1)}, e^{(i_2)} \in K$ and $e^{(i_1)} \neq e^{(i_2)}$.
Define $\alpha_{m_1,m_2} \in (0,\pi)$ as
\begin{align*}
\cos(\alpha_{m_1,m_2}) := (m_1,m_2),
\end{align*}
where
\begin{align*}
m_l \in \left\{ \frac{a_{i_l}}{\|a_{i_l}\|}, \frac{n_{i_l}}{\|n_{i_l}\|} \right\}, \ l \in \{1,2\}.
\end{align*}
Then there exists a constant $C=C(d)\in (0,\pi/4)$ such that
\begin{align*}
|\alpha_{m_1,m_2}|  \in \bigcup\limits_{l=1}^{6}\left(C(d) + \frac{\pi}{6}(l-1), \frac{\pi}{6}l - C(d)\right).
\end{align*}
\end{lem}

\begin{proof}
Arguing as for (\ref{eq:cone_angle}) by using the definition of the set $C_d$, we infer that the angles $\va$ that occur in the representation from Lemma \ref{lem:geo1}
for $\frac{e-e^{(i)}}{\|e-e^{(i)}\|}$ with $e\in C_d$ and $e^{(i)}\in K$ satisfy 
\begin{align}
\label{eq:angle_cone1}
\va \in \left\{ 
\begin{array}{ll}
(\frac{5\pi}{6}-C(d), \frac{7\pi}{6} + C(d)) &\mbox{ if } i=1,\\
(-\frac{\pi}{2}+C(d),-\frac{\pi}{6}-C(d)) &\mbox{ if } i=2,\\
(\frac{\pi}{6}-C(d),\frac{\pi}{2}+C(d)) &\mbox{ if } i=3.
\end{array}
\right.
\end{align}
Here $C(d)\in(0,\pi/3)$ is a constant, which depends only on $d$.
The associated symmetrized rank-one connection is determined by $\va$ as stated in Lemma \ref{lem:geo3}. Applied to the situation in Lemma \ref{lem:angle_a} this implies that $a_{i_1}, n_{i_1}, a_{i_2}, n_{i_2}$ are expressed in terms of $\va$ (as in Lemma \ref{lem:geo3}). Since $d>0$, the sectors parametrized by $\va$ however do not overlap for $i_1\neq i_2$. As a consequence of this and of the options in (\ref{eq:angle_cone1}), only the claimed angles $\alpha_{m_1,m_2}$ occur.
\end{proof}

\section{The Convex Integration Algorithm}
\label{sec:Conti}
In this section we present and analyze our convex integration algorithm (c.f. Algorithms \ref{alg:construction}, \ref{alg:skew}). 
Our discussion of this consists of four parts: First in Section \ref{sec:conv} we introduce a replacement construction, in which a deformation gradient can be modified (c.f. Lemma \ref{lem:conti_undeformed}-\ref{lem:convex_int}). Here we follow
Otto's Minneapolis lecture notes \cite{Otto} and refer to this construction as a version of Conti's construction (c.f. \cite{C}, \cite{CT05} (Appendix) but also \cite{K1}). \\
Next in Section \ref{sec:alg} we explain how the Conti construction can be exploited to formulate the convex integration algorithm (c.f. Algorithms \ref{alg:construction}, \ref{alg:skew}). Here we deviate from the more common \emph{qualitative} algorithms by precisely prescribing error estimates in strain space, by specifying a covering construction and by controlling the skew part \emph{quantitatively}.\\
In Section \ref{sec:well} we analyze our algorithms and show that they are well-defined (c.f. Proposition \ref{prop:symcontrol}). We further provide a control on the skew part of the resulting construction (c.f. Proposition \ref{prop:skewcontrol1}).\\
Finally, in Section \ref{sec:exist} we use Algorithms \ref{alg:construction}, \ref{alg:skew} to deduce the existence of solutions to the inclusion problem (\ref{eq:incl}), c.f. Proposition \ref{prop:convex_int}.
\\ 

We remark that our version of the convex integration scheme is based on particular properties
of our set of strains: For the hexagonal-to-rhombic phase transition 
the laminar convex hull equals the convex hull (c.f. Lemma \ref{lem:convex_hull}). Moreover, we can
connect any matrix in $\intconv(K)$ with the wells $K$ (c.f. Lemma
\ref{lem:rk1}). For a general inclusion problem this is no longer possible and
hence more sophisticated arguments are necessary. In spite of the restricted
applicability of the scheme, we have decided to focus on the hexagonal-to-rhombic phase transformation, as it yields one of the simplest instances
of convex integration and illustrates the difficulties and ingredients, which have to be dealt with in proving higher Sobolev regularity in the simplest possible set-up.

\subsection{The replacement construction}
\label{sec:conv}
In this section we describe the replacement construction that allows to modify constant gradients by replacing them with an affine construction that preserves the boundary values. Moreover, the resulting new gradients are controlled (c.f. Lemma \ref{lem:convex_int}).\\

We begin by recalling Otto's variant of Conti's construction \cite{Otto}, see also the video at \cite{minneapolis}:

\begin{figure}[t]
  \centering
  \includegraphics[width=0.8\linewidth, page=3]{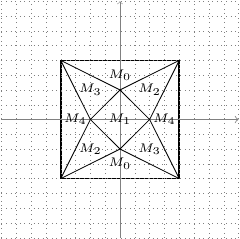}
  \caption{The level sets in Conti's construction}
  \label{fig:conti}
\end{figure}
\begin{lem}[Undeformed Conti construction]
  \label{lem:conti_undeformed}
Let $\Omega=(-1,1)^{2}$ and let
\begin{equation}
\label{eq:contimatrices}
\begin{split}
&  M_{0}=2
  \begin{pmatrix}
    0 & 1\\ 0 &0
  \end{pmatrix},
                M_{1}= 2
                \begin{pmatrix}
                  0 & -1 \\ 1 & 0
                \end{pmatrix},
                                M_{2}= \frac{2}{3}
                                \begin{pmatrix}
                                  2 & -1 \\ 1 & -2
                                \end{pmatrix}, \\
 & M_{3}= \frac{2}{3}
  \begin{pmatrix}
    -2 & -1 \\ 1 & 2
  \end{pmatrix},
                 M_{4}=2
                   \begin{pmatrix}
                     0 & 0 \\ -1& 0
                   \end{pmatrix}.
\end{split}
\end{equation}
Then there exists $u: \R^{2}\rightarrow \R^{2}$ Lipschitz such that
\begin{align*}
 & u=0 \mbox{ on } \R^{2}\setminus \Omega, \\
&  \nabla u \in \{M_{0}, \dotsi, M_{4}\} \text{ in }\Omega, \\
 & \{x: \nabla u(x) = M_{i}\} \mbox{ is a finite union of rectangles and triangles}, \\
 & |\{x: \nabla u(x) = M_{i}\}|=
  \begin{cases}
    1 & i=0,4, \\
    \frac{1}{2} & i=1, \\
    \frac{3}{4} & i=2,3.
  \end{cases}               
\end{align*}
\end{lem}

Following Otto \cite{Otto}, we generalize this construction slightly by allowing variable volume fractions:

\begin{lem}[Variable Conti construction]
  \label{cor:var_conti}
  Let $\Omega=(-1,1)^2$.
  Let $\lambda \in (0,1)$. Define
   \begin{align*}
 &     M_{0}=2
  \begin{pmatrix}
    0 & 1\\ 0 &0
  \end{pmatrix},
                 M_{4}=2
                   \begin{pmatrix}
                     0 & 0 \\ -1& 0
                   \end{pmatrix}, \\
  &   M_{1}^{\lambda}= 2
                \begin{pmatrix}
                  0 & \frac{-1+\lambda}{\lambda} \\ \frac{1-\lambda}{\lambda} & 0
                \end{pmatrix},
                      M_{2}^{\lambda}=\frac{-2(1-\lambda)}{1-(1-\lambda)^{2}}
    \begin{pmatrix}
      1 & \lambda -1\\
      1-\lambda & -1
    \end{pmatrix},\\
   &               M_{3}^{\lambda}=Q^{T}M_{2}^{\lambda}Q, Q=
                  \begin{pmatrix}
                    0 & 1 \\ -1 & 0
                  \end{pmatrix}.
  \end{align*}
Then there exists $u: \R^{2}\rightarrow \R^{2}$ Lipschitz such that
\begin{align*}
 & u=0 \mbox{ on } \R^{2}\setminus \Omega, \\
&  \nabla u \in \{M_{0}^{\lambda}, \dotsi, M_{4}^{\lambda}\}\text{ in }\Omega, \\
 & \{x: \nabla u(x) = M_{i}^{\lambda}\} \mbox{ is a finite union of rectangles and triangles}, \\
&  |\{x: \nabla u(x)=M_{i}^{\lambda}\}|=
    \begin{cases}
      2(1-\lambda) & i=0,4, \\
      2\lambda^{2} & i =1,\\
      \frac{1}{2}(4-4(1-\lambda)-2\lambda^{2})=\lambda(2-\lambda) & i=2,3.
    \end{cases}              
\end{align*}	  
Multiplying $u$ with $\frac{1}{2\lambda'}$ and setting $\lambda=1-\lambda'$, we obtain the
matrices on page 56 of \cite{Otto}:
\begin{align}
  \label{eq:matrices}
  \begin{split}
  &M_{0}=\frac{1}{\lambda'}
  \begin{pmatrix}
    0 & 1 \\ 0 &0
  \end{pmatrix},
                 M_{1}=\frac{1}{1-\lambda'}
  \begin{pmatrix}
    0 & -1 \\ 1 &0
  \end{pmatrix},
                  M_{2}=\frac{1}{1-\lambda'^{2}}
  \begin{pmatrix}
    1 & -\lambda' \\ \lambda' &-1
  \end{pmatrix},\\
  &                            M_{3}=\frac{1}{1-\lambda'^{2}}
  \begin{pmatrix}
    -1 & -\lambda' \\ \lambda' &1
  \end{pmatrix},
                               M_{4}=\frac{1}{\lambda'}
  \begin{pmatrix}
    0 & 0 \\ -1 &0
  \end{pmatrix},
  \end{split}
\end{align}
and volumes (total volume $4$)
\begin{align*}
& |\{x: \nabla u(x)=M_{i}\}|=
  \begin{cases}
    2\lambda' & i=0,4, \\
    2(1-\lambda')^{2} & i=1, \\
    (1-\lambda')(1+\lambda') & i=2,3.
  \end{cases}
\end{align*}
\end{lem}

\begin{proof}[Proof of Lemma \ref{cor:var_conti} and \ref{lem:conti_undeformed}]
Lemma \ref{lem:conti_undeformed} is a special case of Lemma \ref{cor:var_conti} and (\ref{eq:matrices}) with $\lambda =1/2$. We thus describe the general construction for arbitrary $\lambda \in(0,1)$ and then show that this reduces to the matrices from Lemma \ref{lem:conti_undeformed} in the case $\lambda=\frac{1}{2}$.\\
  In order to construct the function $u$, we prescribe the value of $u$ at the points
  $(0,\pm \lambda), (\pm \lambda, 0)$ for $\lambda \in (0,1)$ to be determined
  and then consider linear interpolations. This then yields a 
  piecewise affine Lipschitz map. It remains to verify that all matrices are
  divergence-free and are as claimed in the lemmata (c.f. also Figure \ref{fig:conti}).
  
  \begin{figure}[t]
    \centering
  \includegraphics[width=0.8\linewidth, page=4]{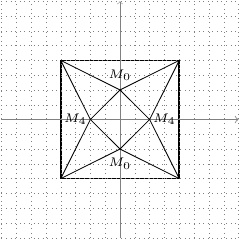}.
    \caption{A symmetric ansatz}
    \label{fig:conti_step1}
  \end{figure}

  We start with the ansatz given in Figure \ref{fig:conti_step1}. The value
  of $u$ at the points $(\pm \lambda,0), (0,\pm \lambda)$ is chosen in
  such a way that linear interpolation in the triangles on the sides of the square in Figure \ref{fig:conti} yields $\nabla u \in \{M_0,M_4\}$, i.e.
  \begin{align*}
    u(\pm \lambda,0)&= \pm 2
    \begin{pmatrix}
      0 \\ 1-\lambda
    \end{pmatrix}, \\
    u(0, \pm \lambda)&= \pm 2
    \begin{pmatrix}
      -1+\lambda \\0
    \end{pmatrix}.
  \end{align*} 
  By linear interpolation on the inner diamond we hence obtain that
  \begin{align*}
  &  \nabla u
    \begin{pmatrix}
      2 \lambda  & 0 \\
      0 & 2 \lambda
    \end{pmatrix}
=4
          \begin{pmatrix}
            0  & -1 +\lambda\\
            1-\lambda & 0
          \end{pmatrix}\\
   & \Rightarrow
    \nabla u =2
    \begin{pmatrix}
      0 & \frac{-1+\lambda}{\lambda} \\
      \frac{1-\lambda}{\lambda} & 0
    \end{pmatrix}.
  \end{align*}
  Choosing $ \lambda =\frac{1}{2}$, we thus obtain $\nabla u = M_{1}$.
  It remains to check the value of $\nabla u$ on the triangles, which interpolate between the sides of the inner diamond and the corners of the outer square. 
  By symmetry it suffices to consider the
  lower left triangle.
  Using again linear interpolation, there $\nabla u$ has to satisfy
  \begin{align*}
   & \nabla u
    \begin{pmatrix}
      1  & 1-\lambda\\
      1-\lambda & 1
    \end{pmatrix}
=
                  \begin{pmatrix}
                    -2(1-\lambda) & 0 \\
                    0 & -2(1-\lambda)
                  \end{pmatrix} \\
    &\Rightarrow \nabla u =\frac{-2(1-\lambda)}{1-(1-\lambda)^{2}}
    \begin{pmatrix}
      1 & \lambda -1\\
      1-\lambda & -1
    \end{pmatrix}.
    \end{align*}
    Setting $\lambda = 1/2$ this equals
    \begin{align*}
\nabla u                  =
\frac{2}{3}
                  \begin{pmatrix}
                    2& -1 \\
                    1 & -2
                  \end{pmatrix},
  \end{align*}
  which is the matrix $M_2$ from Lemma \ref{lem:conti_undeformed}.
\end{proof}

Using this construction as a basic building block, the following lemma allows us to
replace a general matrix $M \in \R^{2\times 2}$ and to restrict
the replacement matrices to an $\epsilon$-neighborhood of a rank-one line passing through $M$. 

\begin{lem}[Deformed Conti construction, page 57 of \cite{Otto}]
  \label{lem:conti_deformed}
  Let $M,M_{0},M_{1}$ be given matrices such that
  \begin{align*}
   & M = \frac{1}{4}M_{0} + \frac{3}{4}M_{1}, \\
   & M_{1}-M_{0}= a \otimes n, \quad a \cdot n =0.
  \end{align*}
Then, for every $\epsilon>0$ there exist matrices
$\tilde{M}_{1},\tilde{M}_{2},\tilde{M}_{3},\tilde{M}_{4}$ with 
\begin{equation}
\label{eq:error}
\begin{split}
 & |\tilde{M}_{1}- M_{1}|< \epsilon, |\tilde{M}_{2}-M_{2}|< \epsilon, |\tilde{M}_{4}-M|<\epsilon, \\
  &|\tilde{M}_{3}-M_{2}|<\epsilon, \\
 & M_{2}=\frac{1}{5}M_{0} + \frac{4}{5}M_{1},
\end{split} 
\end{equation}
a rectangle $\Omega \subset \R^{2}$ of 
aspect ratio $\delta =\frac{\epsilon}{20 |a|}$
and a Lipschitz map $u:
\R^{2}\rightarrow \R^{2}$ such that
\begin{align*}
  &\nabla u = M \mbox{ in } \R^{2} \setminus \Omega, \\
  &\nabla u \in \{M_{0}, \tilde{M}_{1}, \dotsi, \tilde{M}_{4}\}, \\
  &|\{x: \nabla u (x)=M_{0}\}| \geq \frac{1}{8}|\Omega|.
\end{align*}
Furthermore, the level sets of $\nabla u$ are given by the union of at most 16
triangles.
\end{lem}

\begin{figure}[t]
  \centering
  \includegraphics[width=0.8\textwidth, page=5]{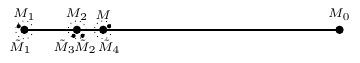}
  \caption{The horizontal axis corresponds to the upper right component of the
    matrix:
  $M\sim 0, M_{0}\sim \frac{1}{\lambda}, M_{1}\sim -\frac{1}{1-\lambda},
  M_{2}\sim - \frac{\lambda}{1-\lambda^{2}}$.
Here for $\lambda=\frac{1}{4}$.}
  \label{fig:conti_deformed_matrices}
\end{figure}

\begin{proof}
  Mapping $u \mapsto u - Mx$, it suffices to consider the case $M=0$.
Furthermore, rotating the rectangular domain by $x \mapsto Qx, Q \in
\text{SO}(2)$ and scaling $u \mapsto \frac{1}{|a|} u$, we may assume that $n= (1,0)$
and $a=(0,-1)$.
Hence, 
\begin{align*}
  M_{0}= \frac{1}{\lambda}
  \begin{pmatrix}
    0 & 1 \\ 0 & 0
  \end{pmatrix},
                 M_{1}= \frac{1}{1-\lambda}
                 \begin{pmatrix}
                   0 & -1 \\ 0 & 0 
                 \end{pmatrix}, 
                                 M_{2}= \frac{1}{1-\lambda^{2}}
                                 \begin{pmatrix}
                                   0 & -\lambda \\ 0 & 0
                                 \end{pmatrix},
\end{align*}
where $\lambda=\frac{1}{4}$ and $-\frac{\lambda}{1-\lambda^{2}}=
-\frac{4}{15}=\frac{1}{5} (4) + \frac{4}{5}(-\frac{4}{3})$.
Applying the construction of Lemma \ref{lem:conti_undeformed} (with
$\lambda=\frac{1}{4}$ instead) rescaled by
$\frac{2}{\lambda}$, we obtain a Lipschitz function $v : \R^{2} \rightarrow
\R^{2}$, which vanishes outside the rectangle $\tilde{\Omega}=(-1,1)^{2}$ and
satisfies
\begin{align*}
  |\{x: \nabla v(x)=M_{0}\}| \geq \frac{1}{8}|\tilde{\Omega}|.
\end{align*}
However, the values of $\nabla v$ as given in \eqref{eq:matrices} in Lemma \ref{cor:var_conti} are
not yet in an $\epsilon$-neighborhood of
$\{M,M_{1},M_{2}\}$.
Hence, we consider the following change of coordinates and the following modified deformation:
\begin{align*}
  (y_{1},y_{2}) &= (y_{1}(x),y_{2}(x)):=(x_{1}, \delta x_{2}), \\
  (u_{1}(x),u_{2}(x))&:= (\delta v_{1}(x), \delta^{2} v_{2}(x)).
  \end{align*}
  We remark that this transforms the domain $\tilde{\Omega}=(-1,1)^2$ into the domain $(-1,1)\times (-\delta,\delta)$ and moreover note that this scaling preserves volume fractions.
Rewriting $\nabla_{x} v(x)$ into $\nabla_{y} u$ yields
\begin{align}
\label{eq:err}
\nabla_y u (y) =  
  \left. \begin{pmatrix}
    \delta \p_{x_1} v_{1} & \p_{x_2} v_1 \\
    \delta^{2} \p_{x_1} v_2 & \delta \p_{x_2} v_2
  \end{pmatrix}\right|_{(x_1, x_2/\delta)}
  =
  \begin{pmatrix}
   0 & \p_{x_2} v_1 \\ 0 & 0
  \end{pmatrix}
 +\mathcal{O}(\delta|\nabla_x v|),
\end{align}
which in particular leaves $M_{0}$ invariant.
Letting $\delta$ be sufficiently small, we thus obtain
the desired $\epsilon$-closeness. Undoing the initial rescaling with
$|a|$ leads to the precise requirement
\begin{align*}
\delta |\nabla_x v| |a| \leq  \epsilon.
\end{align*}
This implies the claimed ratio for $\Omega$ by noting that $|\nabla_x v| \leq 20$.
\end{proof}

\begin{rmk}
\label{rmk:depend}
We remark that both the side ratio $\delta$ as well as the error $\epsilon$ remain unchanged under rescalings of the form $\mu u(\frac{x}{\mu})$ (as this leaves the gradient invariant).
\end{rmk}

We now show how to apply Lemma \ref{lem:conti_deformed} to the setting of symmetric matrices in our three-well problem (\ref{eq:K3}):

\begin{lem}[Application to the three-well-problem, pages 60 ff. of \cite{Otto}]
\label{lem:convex_int}
Suppose that $M \in \R^{2\times 2}$ with 
$$e(M):=\frac{1}{2}(M+M^{T}) \in \intconv(K)$$ 
and let $e^{(i)}$ with $i\in\{1,2,3\}$ be such that 
\begin{align}
\label{eq:dist}
  |e(M)-e^{(i)}| \leq \dist(e(M),K) + 4 \epsilon_0. 
\end{align}
Let $\epsilon_{0}\leq \frac{\dist(e(M), \p \conv(K))}{100}$. Then, for every $0<\epsilon<\epsilon_{0}$ there exist a Lipschitz function
$u:\R^{2}\rightarrow \R^{2}$, a rectangular domain $\Omega$
(with ratio $1:\delta$ and 
$\delta =\frac{\epsilon}{20|e(M)-e^{(i)}|}$) 
with symmetric parts $e^{(i)},\tilde{e}_1,\dots,\tilde{e}_4 \in \intconv(K)$, such that
\begin{align*}
&  u(x)= Mx \text{ on } \R^{2} \setminus \Omega, \notag \\
& e(\nabla u)\in \{e^{(i)},\tilde{e}_{1}, \dotsi , \tilde{e}_{4}\} \subset \conv(K) \mbox{ in } \Omega, \notag \\
&      |\{x\in \Omega: e(\nabla u)(x)= e^{(i)}\}|/|\Omega|=\frac{1}{4}, \notag \\
&\nabla u \in \{\tilde{M}_0, \dots, \tilde{M}_4\} \mbox{ with }
|\tilde{M}_4-M|\leq \epsilon.
\end{align*}
\end{lem}

\begin{proof}
Let $M$ and $e^{(i)}$ be given. Since $e^{(1)},e^{(2)},e^{(3)}$
are arranged in an equilateral triangle with side lengths $\sqrt{3}$ (with respect to the spectral norm)
and as (\ref{eq:dist}) holds, there exists $\tilde{e}_{1} \in \intconv(K)$ such that 
\begin{align*}
   e(M)=\frac{1}{4}e^{(i)} + \frac{3}{4}\tilde{e}_{1}. 
\end{align*}
  Next let $S:=\omega(M) \in \Skew(2)$ and let $\tilde{S} \in
  \Skew(2)$ to be determined.
  Then we obtain
  \begin{align*}
  &  M= \frac{1}{4}(e^{(i)}+S+3\tilde{S})+ \frac{3}{4}(\tilde{e}_{1}+S-\tilde{S}), \\
   & (e^{(i)}+S+3\tilde{S}) - (\tilde{e}_{1}+S-\tilde{S})= e^{(i)}- \tilde{e}_{1} +4 \tilde{S}.
  \end{align*}
  Since we are in two dimensions, any two symmetric, trace-free matrices are symmetrized rank-one connected (c.f. Lemma \ref{lem:rk1}).
  Thus, there exist vectors $a\in \R^{2}\setminus\{0\}$, $n\in S^1$ such that
  \begin{align*}
    e^{(i)}- \tilde{e}_{1}
    = \frac{1}{2}(a \otimes n + n \otimes a).
  \end{align*}
  Furthermore, as $\tr(e^{(i)})= \tr(\tilde{e}_{1})$, $a$ and $n$ are orthogonal.
  Choosing
  \begin{align}
  \label{eq:choice}
    \tilde{S}:= \frac{1}{8}(a \otimes n - n \otimes a) = \frac{1}{4} \omega(a\otimes n)
    \mbox{ or } \tilde{S}:=- \frac{1}{4} \omega(a\otimes n),
  \end{align}
 we thus obtain that the matrices
  \begin{align}
  \label{eq:matrices1}
    M_{0}:=(e^{(i)}+S+3\tilde{S}), \ M_{1}:=(\tilde{e}_{1}+S-\tilde{S})
  \end{align}
  are rank-one connected (with difference $a \otimes n$ or $n \otimes a$,
  respectively) and
  \begin{align*}
    M= \frac{1}{4}M_{0} + \frac{3}{4}M_{1}.
  \end{align*}
  We may hence apply the construction of Lemma \ref{lem:conti_deformed} with $M,
  M_0, M_1$ as defined above. 
  Noting that $\|e(M)-e^{(i)}\|= \frac{|a|}{2} $ (c.f. Lemma \ref{lem:skewcontrol2}), Lemma \ref{lem:conti_deformed} implies the statement on the side ratio for $\Omega$.
  Finally, we note that the $\epsilon$-closeness of the matrices $\tilde{M}_{1},
  \dotsi, \tilde{M}_{4}$ also implies that their symmetric parts are $\epsilon$-close. 
\end{proof}

\begin{figure}[t]
  \centering
  \includegraphics[page=6]{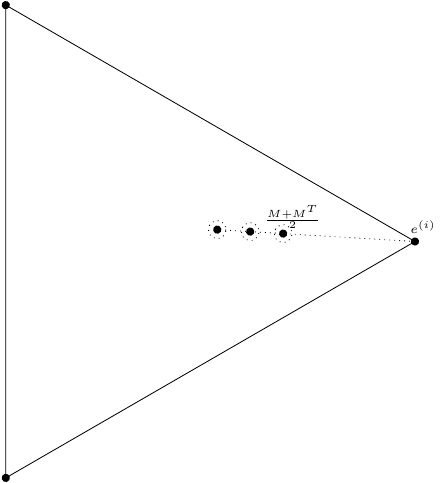}
  \caption{Relative positions of the symmetric part of the matrices inside the
    convex hull.}
  \label{fig:matrices_in_triangle}
\end{figure}

\begin{nota}
\label{nota:conti}
In the preceding Lemma \ref{lem:convex_int} the matrices $\tilde{M}_0, \cdots, \tilde{M}_4$ obey the same (convexity) relations as the ones in Lemma \ref{lem:conti_deformed}, where for the matrices $M_0$ and $M_1$ we insert the ones from \eqref{eq:matrices1}, c.f. Figures \ref{fig:conti_deformed_matrices}, \ref{fig:matrices_in_triangle}. The error estimates in \eqref{eq:error} thus 
\begin{itemize}
\item motivate us to refer to the matrix $\tilde{M}_4$ as \emph{stagnant} (with respect to the replaced matrix $M$).
\item The matrices $\tilde{M}_1,\tilde{M}_2,\tilde{M}_3$ will also be called \emph{pushed-out} matrices (with the factors $\frac{4}{3}$ and $\frac{16}{15}$ respectively), since by construction 
\[\frac{4}{3}\left|e(M)-e^{(i)}\right|-\epsilon\leq\left|e(\tilde{M}_1)-e^{(i)}\right|\leq \frac{4}{3}\left|e(M)-e^{(i)}\right|+\epsilon,\]
and similarly for the other matrices.
\end{itemize}
In order to emphasize the dependence on $M$, we also use the notation
\begin{align*}
\tilde{M}_0(M),\dots,\tilde{M}_4(M).
\end{align*}
Although the matrices $\tilde{M}_0,\dots,\tilde{M}_4$ also depend on the choice of $e^{(i)}$, in the sequel we will often suppress this additional dependence for convenience as the reference well will be clear in most of our applications.\\
We refer to the construction of Lemma \ref{lem:convex_int} as 
\emph{the $(\epsilon$, $\delta)$ Conti construction with respect to $M, e^{(i)}$}. If some of the parameters of this are self-evident from the context, we also occasionally omit them in the sequel.
\end{nota}

We emphasize that in our construction in Lemma \ref{lem:convex_int}, we have the choice between two different solutions, which
differ in the sign of their skew symmetric component and thus in the choice of the corresponding rank-one connection (c.f. (\ref{eq:choice})). This freedom of
choice is a central ingredient in the control over the skew symmetric part of the iterated
constructions. We summarize this observation in the following corollary.

\begin{cor}
  \label{cor:plusminus_construction}
  Let $M,e^{(i)},\epsilon_{0},\epsilon$ be as in Lemma \ref{lem:convex_int}. Then
  there exist two Lipschitz functions $u_{+}, u_{-}: \R^{2}\rightarrow \R^{2}$
  such that on the set where $e(\nabla u_{\pm}) =e^{(i)}$
  \begin{align*}
    \omega(\nabla u_{\pm})= \omega(M) \pm 
    \frac{3}{4} \omega(a \otimes n) =: \omega(M) \pm \hat{S}.
  \end{align*}
  Furthermore, up to an error of size $\epsilon$ the skew parts on the other
  level sets are given by
  \begin{align*}
    \omega(M), \ \omega(M) \pm \frac{1}{3}\hat{S} , \ \omega(M) \pm \frac{1}{15} \hat{S}.
  \end{align*}
\end{cor}

\begin{proof}
From \eqref{eq:matrices1} we read off the skew symmetric parts of $M_0, M_1$. The skew symmetric part of 
$M_2:= \frac{1}{5}M_0 + \frac{4}{5}M_1$ is a consequence of that. The result then follows from Lemma \ref{lem:conti_deformed}.
\end{proof}

\subsection{The convex integration algorithm}
\label{sec:alg}
In this subsection we formulate our convex integration algorithm. It consists of two parts, Algorithms \ref{alg:construction} and \ref{alg:skew}. The first part (Algorithm \ref{alg:construction}) determines the symmetric part of the iterated deformation vector field, while the second part (Algorithm \ref{alg:skew}) deals with the choice of the ``correct" skew component.\\
After formulating the algorithms, we prove their well-definedness (i.e. show that it is indeed possible to iterate this construction as claimed). \\

In the whole section we assume that the domain $\Omega$ and the matrix $M$ in (\ref{eq:incl}) fit together in the sense that $\Omega = Q_{\beta}[0,1]^2$, where $Q_{\beta}$ is the rotation of the Conti construction from Lemma \ref{lem:conti_deformed} for $M$ (and the closest energy well $e^{(i)}$). These ``special" domains will play the role of the essential building blocks of the construction of convex integration solutions in general Lipschitz domains (c.f. Section \ref{sec:generaldomains}).\\

We define our convex integration scheme:

\begin{alg}[Quantitative convex integration algorithm, I]
  \label{alg:construction}
We consider the following construction:
  \begin{itemize}
  \item [Step 0:] State space and data. 
\begin{itemize}
\item[(a)] State space.  
  Our state space is given by
	\begin{align}  
	\label{eq:state_space}
  SP_j:=(j,u_{j},\{\Omega_{j,k}\}_{k\in \{1,\dots,J_j\}},
  e_{j}^{(p)},\epsilon_j, \delta_j).
  \end{align}
Here $j \in \N$ and $u_{j}: \Omega \rightarrow \R^2$ is a piecewise affine
function. The sets
\begin{align*}
\Omega_{j,k}\subset \Omega\cap\{\nabla u_j=const\}\cap\{e(\nabla u_j)\notin K\}
\end{align*} 
are closed triangles, which form a (up to null sets) disjoint, finite partition of the level sets of $\nabla u_j$, for which $e(\nabla u_j) \notin K$. Let $\Omega_j:= \bigcup\limits_{k=1}^{J_j}\Omega_{j,k}$ denote the set, on which $e(\nabla u_j)$ is not yet in one of the energy wells.\\
The function
\begin{align*}
e^{(p)}_{j}:\Omega \rightarrow K
\end{align*}
is constant on each of the sets $\Omega_{j,k}$. It essentially keeps track of the well closest to $e(\nabla u_j|_{\Omega_{j,k}})$ for each $j,k$. \\
The functions 
\begin{align*}
\epsilon_j, \delta_j: \Omega \rightarrow \R,
\end{align*}
are constant on each set $\Omega_{j,k}$ and vanish in $\Omega \setminus \Omega_j$. They correspond to the error and side ratio in the Conti construction, which is to be applied in $\Omega_{j,k}$. The functions $\epsilon_j, \delta_j$ are coupled by the relation
\begin{align*}
\delta_j = \frac{\epsilon_j}{10^2 d_K}, 
\mbox{ where } d_K:=\dist(e(M),K).
\end{align*}
Hence, in the following (update) steps, we will mainly focus on $\epsilon_j$ and assume that $\delta_j$ is modified accordingly.
\item[(b)] Data. Let $M\in \R^{2\times 2}$ with $e(M)\in \intconv(K)$. Let $\Omega= Q_{\beta}[0,1]^2$ with $Q_{\beta}$ denoting the rotation associated with $M$ (c.f. explanations above). Further set
\begin{align*}
d_0&:= \dist(e(M), \p \conv(K)),\\
\epsilon_0 &:= \min\left\{\frac{d_0}{100},\frac{1}{1600}\right\},\ 
\delta_0 := \frac{\epsilon_0}{10^2 d_K}.
\end{align*}
\end{itemize}
\item[Step 1:] Initialization, definition of $SP_1$. We consider the data from Step 0 (b) and in addition define
\begin{align*}
u_{0}(x)&=Mx- \omega(M)x, \\
e^{(p)}_{0}&=\argmin \dist{(e(M), K)}.
\end{align*}
In the case of non-uniqueness in the above minimization problem, we arbitrarily choose any of the possible options.\\
Possibly dividing $\delta_0$ by a factor up to $100$, we may assume that $K_{0,0}:=\delta_0^{-1}\in \N$. We cover $
\Omega=Q_{\beta}[0,1]^2$ by $K_{0,0}$ many (translated) up to null-sets 
disjoint $(\epsilon_0,\delta_0)$ Conti constructions with respect to $
\nabla u_0$ and $e^{(p)}_0$ (c.f. Notation \ref{nota:conti}). We denote these sets by $R_{0,1}^{1},\dots, R_{0,K_{0,0}}^{1}$. We remark that 
$\Omega = \bigcup\limits_{l=1}^{K_{0,0}} R_{0,l}^1$ is possible with (up to null sets) disjoint choices of $R_{0,l}^1$, $l\in\{1,\dots,K_{0,0}\}$, as by definition of the domain $\Omega$ the
sets $R_{0,l}^{1}$, $l\in\{1,\dots,K_{0,0}\}$, are parallel to one of the 
sides of $\Omega$ and as $\delta_0^{-1}\in \N$.
We apply Step 2 (b) on these sets. As a consequence we obtain 
$SP_1$.

\item[Step 2:] Update. Let $SP_j$ be given.
	Let $M_{j,k}:=\nabla u_j|_{\Omega_{j,k}}$ for some 
	$k\in \{1,\dots, J_j\}$. We explain how to update $u_{j}$ and 
	$\epsilon_j, \delta_j$ on 
	$\Omega_{j,k}$. \\
	We seek to
  	apply the construction of Lemma 									\ref{lem:convex_int}
    with $\epsilon_{j,k}:=\epsilon_{j}|_{\Omega_{j,k}}$, 
    $\delta_{j,k}:=\delta_{j}|_{\Omega_{j,k}}$ and
    \begin{align}
    \label{eq:matrices_1}
     e^{(p)}_{j,k}, \ M_{j,k}
     \end{align}
     in a part of $\Omega_{j,k}$.
     To this end, we cover the
     domain $\Omega_{j,k}$ by a union of finitely many 
     (up to null sets) disjoint triangles and rectangles. 
     The rectangles are chosen as
     translated and rescaled
     versions of the domains in the $(\epsilon_{j,k}, \delta_{j,k})$ 
     Conti construction with respect to the matrices from
     (\ref{eq:matrices_1}). We denote these rectangles by
     $R_{j,l}^{k}$, $l\in\{1,\dots,K_{j,k}\}$, for some $K_{j,k}\in \N$
     and require that they cover at least 
     a fixed volume fraction $v_0>0$ of the overall
     volume of $\Omega_{j,k}$ (which is always possible, c.f. Section 
     \ref{sec:covering} for our precise covering algorithm). \\
     We define new
     sets 
     $\tilde{\Omega}_{j+1,l}^k$, $l\in\{1,\dots, \tilde{K}_{j,k}\}$: 
     These are given by the triangles which are in 
     $\Omega_{j,k}\setminus 
     \bigcup\limits_{l=1}^{K_{j,k}}R_{j,l}^k$ and by the triangles
     which form the level sets of the deformed Conti rectangles
     $R_{l}^k$. \\
     \begin{itemize}
     \item[(a)]
     For $x\in \Omega_{j,k}\setminus 
     \bigcup\limits_{l=1}^{K_{j,k}}R_{j,l}^k$ we define 
     \begin{align*}
     u_{j+1}(x)&:=u_{j}(x),\\
     \epsilon_{j+1}(x) &:= \epsilon_{j}(x) \ \
     (\mbox{and hence } \delta_{j+1}(x):= \delta_{j}(x)),\\
     e^{(p)}_{j+1}(x)&:= e^{(p)}_{j}(x).
     \end{align*}
   
     Further we set $\Omega_{j+1,l}^k:=\tilde{\Omega}_{j+1,l}^k$.
     Carrying this out for all $k\in\{1,\dots,J_{j}\}$ hence yields a
     collection of
     triangles 
     $$\{\Omega_{j+1,l}^k\}_{k\in\{1,\dots,J_{j}\},l\in\{1,\dots,K_{j,k}\}}$$ 
     covering 
     $\Omega_{j}\setminus 
     \bigcup\limits_{l=1}^{K_{j,k}}R_{j,l}^k$.\\
     \item[(b)]
     In the sets $R_{j,l}^k$ we apply the Conti construction
     with the matrices from (\ref{eq:matrices_1}). 
     In this application we choose the skew part according to Algorithm
      \ref{alg:skew}. With $\tilde{\Omega}_{j+1,l}^k \subset 
      \bigcup\limits_{k=1}^{J_j}\bigcup\limits_{l=1}^{K_{j,k}} R_{j,l}^k$
     as defined in Step 2 (a),
     we define $u_{j+1}|_{\tilde{\Omega}_{j+1,l}^k}$ as the
     function from the corresponding Conti construction. 
     More precisely, in each of the rectangles $R_{j,l}^k$ the 
     matrix $M_{j,k}$ has been replaced by the
     matrices
     \begin{align*}
     \tilde{M}_0(M_{j,k}),  \dots, \tilde{M}_4(M_{j,k}),
     \end{align*}
     with $e(\tilde{M}_0(M_{j,k})) = e^{(p)}_{j,k}$.
     For each $x\in \tilde{\Omega}_{j+1,l}^k$ with 
     $\tilde{\Omega}_{j+1,l}^k$  
     as above, we define 
     \begin{align*}
       \epsilon_{j+1}(x):=
       \begin{cases}
     \epsilon_{0}
     &\mbox{ for } \nabla u_{j+1}|_{\tilde{\Omega}_{j+1,k}} \in 
     \{\tilde{M}_1(M_{j,k}),\dots,\tilde{M_3}(M_{j,k})\},\\
      \epsilon_{j}(x)/2
     &\mbox{ for } \nabla u_{j+1}|_{\tilde{\Omega}_{j+1,k}} 
     = \tilde{M}_4(M_{j,k}),\\
     0 &\mbox{ for } \nabla u_{j+1}|_{\tilde{\Omega}_{j+1,k}} 
     = \tilde{M}_0(M_{j,k}).    
       \end{cases}
     \end{align*}
     For the definition of $\delta_{j+1}$ we recall its coupling with $\epsilon_{j+1}$.
   We further set
     \begin{align*}
     e^{(p)}_{j+1}(x)
       :=
       \begin{cases}
         \operatorname*{argmin}\limits_{i\in\{1,2,3\}}
     & \{|e(\nabla u_{j+1})|_{\tilde{\Omega}_{j+1,k}} - e^{(i)}|\}  \\
     &\mbox{ for } \nabla u_{j+1}|_{\tilde{\Omega}_{j+1,k}} \in 
      \{\tilde{M}_1(M_{j,k}),\dots,\tilde{M_3}(M_{j,k})\},\\
      e^{(p)}_{j}(x) 
     &\mbox{ for } \nabla u_{j+1}|_{\tilde{\Omega}_{j+1,k}} 
     = \tilde{M}_4(M_{j,k}),\\
     e^{(p)}_{j}(x)  &\mbox{ for } \nabla u_{j+1}|_{\tilde{\Omega}_{j+1,k}} 
     = \tilde{M}_0(M_{j,k}).    
       \end{cases}
     \end{align*}
     Here we choose an arbitrary possible minimizer if there is 
     non-uniqueness. Finally, we possibly split each of the sets 
     $\tilde{\Omega}_{j+1,l}^k \in \bigcup\limits_{l=1}^{K_{j,l}}R_{j,l}^k$ into at most four smaller triangles 
     (c.f. Section \ref{sec:cov_Cont}) and add them to the collection 
	$\{\Omega_{j+1,l}^k\}_{k\in \{1,\dots, J_{j}\}, l \in\{1,\dots, K_{j,k}\}}$. 
	Upon relabeling this yields a new collection 
	$\{\Omega_{j+1,k}\}_{k\in \{1,\dots, J_{j+1}\}}$.
  \end{itemize}
    As a result of Steps 2 (a) and (b) we obtain $SP_{j+1}$.
  \end{itemize}
\end{alg}

While this algorithm prescribes the \emph{symmetric} part of the iteration, we complement it with an algorithm, which defines the choice of the \emph{skew} part. Here the main objectives are to keep the resulting skew parts uniformly bounded (which is necessary, if we seek to obtain bounded solutions to (\ref{eq:incl})) and simultaneously to ensure the choice of the ``right" rank-one direction (c.f. Section \ref{sec:quant}, Lemma \ref{lem:BV}). Here the rank-one direction has to be chosen ``correctly" in the sense that the successive Conti constructions are not rotated too much with respect to one another (which corresponds to the ``parallel" case, c.f. Definition \ref{defi:parallel_rot}). \\
In order to make this precise, we introduce two definitions: The first (Definition \ref{defi:family}) allows us to introduce an ``ordering" on the triangles in $\{\Omega_{j,k}\}_{k\in\{1,\dots,J_j\}}$ for different values of $j\in \N$. With this at hand, we then define the notions of being parallel or rotated (c.f. Definition \ref{defi:parallel_rot}).

\begin{defi}
\label{defi:family}
Let $D\in \{\Omega_{j,k}\}_{k\in\{1,\dots,J_{j}\}}$ for $j\geq 1$. Then a
triangle $\hat{D} \subset D$ is a \emph{descendant of $D$ of order $l$}, if
$\hat{D}\in \{\Omega_{j+l,k}\}_{k\in\{1,\dots,J_{j+l}\}}$ is (part of) a level
set of $\nabla u_{j+l}$ and is obtained from $D$ by an $l$-fold application of
the update step of Algorithm \ref{alg:construction} (where we specify the covering to be the one, which is described in Section \ref{sec:covering}). The \emph{set of descendants of $D$ of order $l$} is denoted by $\mathcal{D}_l(D)$. We define $\mathcal{D}(D):=\bigcup\limits_{l=1}^{\infty}\mathcal{D}_l(D)$.\\
A triangle $\bar{D} \in \{\Omega_{j,k}\}_{k\in\{1,\dots,J_{j}\}}$ is a \emph{predecessor of order $l$ of $D$}, if 
$D\in \mathcal{D}_{l}(\bar{D})$. We then write $\bar{D}\in \mathcal{P}_l(D)$ and also use the notation $\mathcal{P}(D)$ 
for the \emph{set of all predecessors} of $D$. 
\end{defi}

With this we define the parallel and the rotated cases:

\begin{defi}
\label{defi:parallel_rot}
Let $e^{(p)}_{j,k}$ be as in Algorithm \ref{alg:construction}.
Let $D\in \{\Omega_{j,k}\}_{k\in\{1,\dots,J_j\}}$ for $j\geq 1$.
Let $j_0\neq 0$ be the smallest index, for which $\mathcal{P}_{j_0}(D)\ni \bar{D} \neq D$ (i.e. $\mathcal{P}_{j_0}(D)$ was the last triangle in Algorithm \ref{alg:construction}, to which Step 2 (b) was applied instead of Step 2 (a)).
 Then, if for a.e. $x\in D$
\begin{align}
\label{eq:parallel}
e^{(p)}_{j}(x) = e^{(p)}_{j-j_0}(x),
\end{align}
we say that \emph{in step $j$ the triangle $D$ is in the parallel case}. If there is no possible confusion, we also just refer to $D$ as \emph{in the parallel case}.\\
If for a.e. $x \in D$
\begin{align}
\label{eq:rot}
e^{(p)}_{j}(x) \neq e^{(p)}_{j-j_0}(x),
\end{align}
we say that \emph{in step $j$ the triangle $D$ is in the rotated case}. If there is no possible confusion, we also just refer to $D$ as \emph{in the rotated case}.
\end{defi}

Let us comment on this definition: Intuitively, its objective is to describe whether successive Conti constructions can be chosen as essentially parallel or whether they are necessarily substantially rotated with respect to each other (hence, these notions will also play a crucial role in Section \ref{sec:covering}, where we construct our precise covering). More precisely, let $SP_j$ be as in Algorithm \ref{alg:construction} and let $j,j_0, D, \bar{D}$ be as in Definition \ref{defi:parallel_rot}. Then, at the iteration step $j_0$ the triangle $\bar{D}$ was a subset of one of the Conti rectangles $R_{j-j_0,l}^k$. Thus, $u_{j-j_0}$ is modified according to the Conti construction with respect to $\nabla u_{j-j_0}|_{\bar{D}}$, $e_{j-j_0}^{(p)}|_{\bar{D}}$ in this domain. In particular, the difference of the matrices $e(\nabla u_{j-j_0}|_{\bar{D}})$, $e_{j-j_0}^{(p)}|_{\bar{D}}$ determines a direction $e$ in strain space (up to a choice of the skew direction (c.f. Corollary \ref{cor:plusminus_construction}) this 
is directly related to the orientation of the Conti rectangle $R_{j-j_0,l}^k$). By virtue of Lemma \ref{lem:conti_undeformed} all of the new matrices $e(\tilde{M}_0(\nabla u_{j-j_0}|_{\bar{D}})),\dots, e(\tilde{M}_4(\nabla u_{j-j_0}|_{\bar{D}}))$ essentially lie on the line $e$ in strain space. Hence the direction, which is determined by the difference of $e^{(p)}_{j-j_0+1}|_{D}$ and $e(\nabla u_{j-j_0+1}|_{D})$, is still essentially parallel to the directions $e$ (in strain space). As by definition (we are now in Step 2(a) of Algorithm \ref{alg:construction}) the values of $e^{(p)}_{j-j_0 + l}|_{D}$ and of $\nabla u_{j-j_0 +l}|_{D}$ do not change further until $l=j_0$ is reached, the requirement in (\ref{eq:parallel}) implies that the direction $e$ spanned by $e(\nabla u_{j-j_0}|_{\bar{D}}), e^{(p)}_{j-j_0}|_{\bar{D}}$ and the one spanned by $e(\nabla u_{j}|_{D}), e^{(p)}_{j}|_{D}$ are essentially parallel (c.f. Lemma \ref{lem:angle_1} and Remark \ref{rmk:angle_1} for the precise statements). If we choose 
the correct skew directions in Step 2(b) of Algorithm \ref{alg:construction}, we can hence ensure that the successive Conti constructions are essentially parallel, if (\ref{eq:parallel}) is satisfied.\\
We remark that for this argument to hold and for it to yield new, significant information, it was necessary in Definition \ref{defi:parallel_rot} to mod out the cases, in which Step 2(a) was active, i.e. $\mathcal{P}_l(D) = \{D\}$, as during these there are no changes.
\\
If (\ref{eq:rot}) holds, then the directions of the successive Conti constructions are necessarily substantially rotated with respect to each other (c.f. Lemma \ref{lem:angles} for the precise bounds). In this case we cannot substantially improve the situation to being more parallel by choosing the skew part appropriately in Corollary \ref{cor:plusminus_construction}. Thus, in the sequel, we will exploit these instances as possibilities to control the size of the skew part and to use this, if necessary, to change the sign of the skew direction. The precise formulation of this is the content of Algorithm \ref{alg:skew}.

\begin{alg}[Quantitative convex integration algorithm, II]
\label{alg:skew}
Let $\Omega$, $u_j:\Omega \rightarrow \R^2$ and $SP_j$ for $j\geq 1$ be as in Algorithm \ref{alg:construction}. 
We further consider
\begin{align*}
\omega_j:\Omega \rightarrow \Skew(2).
\end{align*}
This function will be defined to be piecewise constant on $\Omega$ and to be constant on each triangle $\Omega_{j,k}$. It will define the skew part of $\nabla u_j$ on $\Omega_{j,k}$, i.e.
\begin{align*}
\omega(\nabla u_{j}|_{\Omega_{j,k}}) = \omega_j|_{\Omega_{j,k}}.
\end{align*}
\begin{itemize}
\item[Step 1:] Initialization. Let $M$ be as in Step 1 in Algorithm \ref{alg:construction}. Then we define
\begin{align*}
\omega_0(x) = 0 \mbox{ for a.e. } x \in \Omega. 
\end{align*}
In the initialization step of Algorithm \ref{alg:construction} we choose
$\omega_1$ arbitrarily.
\item[Step 2:] Update. Let $j\in \N, j\geq 1$. Let $\omega_j$ and $\Omega_{j,k}$ be given. Suppose that $\tilde{\Omega}_{j+1,l}^k$ with $\tilde{\Omega}_{j+1,l}^k\in \mathcal{D}_1(\Omega_{j,k})$ is constructed from $\Omega_{j,k}$ by our covering argument (c.f. Step 2 in Algorithm \ref{alg:construction}). Then we define $\omega_{j+1}$ as follows:
\begin{itemize}
\item[(a)] If $\tilde{\Omega}_{j+1,l}^k$ is not part of one of the Conti constructions in the covering, then we set 
\begin{align*}
\omega_{j+1}|_{\tilde{\Omega}_{j+1,l}^k} = \omega_{j}|_{\Omega_{j,k}}.
\end{align*}
\item[(b)] If $\tilde{\Omega}_{j+1,l}^k$ is part of one of the Conti constructions in the covering, then by Algorithm \ref{alg:construction} we seek to apply the construction of Lemma \ref{lem:convex_int} with scale $\epsilon_j|_{\Omega_{j,k}}$ and $e^{(p)}_j|_{\Omega_{j,k}}$, $\nabla u_j|_{\Omega_{j,k}}$. Thus, by Corollary \ref{cor:plusminus_construction} we have two possible choices for the skew part of $\nabla u_{j+1}$. These are determined by their sign. To define the sign, let $j_0\in \N$ be the smallest integer such that $D:=\mathcal{P}_{j_0}(\Omega_{j,k}) \neq \Omega_{j,k}$.
We then choose the sign of the new skew direction $\omega_{j+1}|_{\tilde{\Omega}_{j+1,l}^k}$ (and hence determine the whole corresponding skew part) according to
\begin{align*}
&\sgn(\omega_{j+1}|_{\tilde{\Omega}_{j+1,l}^k} - \omega_{j}|_{\tilde{\Omega}_{j+1,l}^k})\\
&:=
\left\{
\begin{array}{ll}
\sgn(\omega_{j}|_{\Omega_{j,k}}- \omega_{j-j_0}|_{\Omega_{j,k}}) 
& \mbox{ if } e^{(p)}_{j}|_{\Omega_{j,l}^k} 
= e^{(p)}_{j-j_0}|_{D},\\
-1 & \mbox{ if } e^{(p)}_{j}|_{\Omega_{j,l}^k} 
\neq e^{(p)}_{j-j_0}|_{D}\\
& \quad \wedge \ \omega_j|_{\Omega_{j,k}}\geq 0,\\ 
1 & \mbox{ if } e^{(p)}_{j}|_{\Omega_{j,l}^k} 
\neq e^{(p)}_{j-j_0}|_{D} \\
& \quad \wedge \ \omega_j|_{\Omega_{j,k}}> 0.
\end{array}
\right.
\end{align*}
\end{itemize}
After having carried out the relabeling step, in which we pass from $\tilde{\Omega}_{j+1,l}^k$ to $\Omega_{j+1,l}$, the function $\omega_{j+1}$ is constant on each of the triangles in $\Omega_{j+1,l}$. Together with Algorithm \ref{alg:construction} this completes the construction of $\nabla u_{j+1}$.
\end{itemize}
\end{alg}

Let us comment on these algorithms: Due to the structure of the convex hulls
(Lemma \ref{lem:convex_hull}), our convex integration algorithm produces a
(countably) piecewise affine solution (in contrast to the solutions obtained by
means of an in-approximation scheme). This is reflected in the fact that the deformation $u_j$ is not further modified in the piecewise polygonal domains in $\Omega \setminus \Omega_j$. The preceding algorithm differs from a non-quantitative version of a convex integration scheme in several aspects:
\begin{itemize}
\item We consider \emph{finite} coverings of $\Omega \setminus \Omega_j$ instead of directly covering the whole domain.
\item We prescribe the choice of $\epsilon_j$ quantitatively.
\item We prescribe the skew part quantitatively.
\end{itemize}
These points are central in our higher regularity argument: 
As we seek to prove higher regularity by means of the interpolation result from Theorem \ref{thm:interpol} or Corollary \ref{cor:int}, we have to control the BV norm of the resulting deformation gradients. However, by a countably infinite (self-similar) covering of the whole domain, this is in general not possible (the total perimeter of the covering triangles is not bounded in general). Hence we only consider \emph{finite} coverings, which produce a controlled (but growing) BV norm and simultaneously allows us to cover a sufficiently large volume fraction $v_0$ of our domain $\Omega_j$. That it is possible to satisfy these two competing aims is content of the covering results of the next sections (c.f. Propositions \ref{prop:parallel}, \ref{prop:rot}). This finite covering of $\Omega_{j,k}$ is the cause for the splitting of Step 2 into two parts. Part (a) deals with the triangles which are not covered by Conti constructions and are in this sense ``errors" (in the sense that $u_j$ is not modified here), 
while part (b) deals with the 
part of the domain that is covered by Conti constructions, on which $u_j$ is modified.
\\
The specification of $\epsilon_j$ is of key relevance as well. It distinguishes in a quantitative way whether a new rank-one connection is rotated or not with respect to the corresponding last rank-one connection. In our BV estimate this leads to different bounds (c.f. the perimeter estimates in Propositions \ref{prop:parallel}, \ref{prop:rot}). In particular we cannot afford substantial rotations, as long as $\epsilon_j\ll \epsilon_0$ is very small, since this would yield superexponential growth for the BV norms, which cannot be compensated in our estimates (c.f. Figure \ref{fig:rectangle_coverings} and the corresponding explanations for the intuition behind this).\\ 
Due to the relation between the size of the scales $\delta_j$ (which itself is directly coupled to the admissible error $\epsilon_j$) and our regularity estimates, we in general seek to choose the value of $\epsilon_j$ as large as possible without leaving $\intconv(K)$. By the intercept theorem, it is always possible to choose $\epsilon_j$ to be ``relatively large" in the push-out steps (c.f. Notation \ref{nota:conti}). However, for stagnant matrices, this is no longer possible. Here we have to ensure a choice of $\epsilon_j$, which is summable in $j\in \N$ (in Algorithm \ref{alg:construction} we choose it geometrically decaying), in order to avoid leaving $\intconv(K)$. These considerations lead to the case distinction in the definition of $\epsilon_{j+1}$ in Step 2 (b) of Algorithm \ref{alg:construction}. 
\\
Finally, the quantitative prescription of the skew part is central to deduce the quantitative BV bound of Lemma \ref{lem:BV}, as we have to take care that, as long as we remain ``parallel" in strain space (c.f. Definition \ref{defi:parallel_rot}), we approximately preserve the same skew direction. This is necessary to prevent the Conti constructions from being substantially rotated with respect to each other if $\epsilon_j$ is very small and constitutes a crucial ingredient in the derivation of our perimeter and BV estimates in Sections \ref{sec:covering} and \ref{sec:quant} (c.f. Figure \ref{fig:rectangle_coverings} for the intuition behind this).\\
The normalization of the initial skew part is convenient (though not necessary).

\subsection{Well-definedness of Algorithms \ref{alg:construction}, \ref{alg:skew}}
\label{sec:well}

We now proceed to prove that Algorithms \ref{alg:construction} and \ref{alg:skew} are well-defined. Here in particular, it is crucial to show that with our choice of the admissible error $\epsilon_j$, we do not leave $\intconv(K)$ in the iteration except to attain one of the energy wells in $K$ (c.f. Proposition \ref{prop:symcontrol}). Moreover, we seek to construct solutions to (\ref{eq:incl}), which are Lipschitz regular.
These points are the content of the following two Propositions \ref{prop:symcontrol}, \ref{prop:skewcontrol1}, which deal with the symmetric and anti-symmetric parts respectively. To show these we will rely on several auxiliary observations.

\subsubsection{Symmetric part}
We begin by discussing the symmetric part and by showing that in our
construction it does not leave $\intconv(K)$, except to reach $K$.

\begin{prop}[Symmetric part]
  \label{prop:symcontrol}
  Let 
  $$d: \R^{2 \times 2} \rightarrow [0,\infty], \ N \mapsto \dist(e(N), \p
  \conv (K) ),$$
   and let $SP_{j}$ and $M$ be as in Algorithm \ref{alg:construction}. 			Then
  for every $j,k \in \N$ and every domain $\Omega_{j,k}\in \{\Omega_{j,k}\}_{k\in\{1,\dots,J_j\}}$ there holds 
  \begin{align*}
    d(\nabla u_{j}|_{\Omega_{j,k}}) 
    \geq \min\left\{\frac{1}{16},d(M)\right\} - 2(\epsilon_0 -\epsilon_{j}|_{\Omega_{j,k}}).
  \end{align*}
  In particular, for all $j\geq 1$ it holds that $\nabla u_j(x) \in \intconv(K)$ for almost all $x\in \Omega_j$.
\end{prop}

\begin{proof}
  We prove the statement inductively. For $j=0$, we note that this holds since
  $\epsilon_{j}=\epsilon_{0}$ and $\nabla u_{0}=M$.

  Let thus $\nabla u_{j}|_{\Omega_{j,k}}=:M_{j,k}$ be given. 
  We only show that the result remains true for $j+1$ in the regions, in which
  the Conti construction is applied, as in the other regions it holds by the
  induction hypothesis (as $\epsilon_{j+1}=\epsilon_j$ for these regions). Let
  $\tilde{M}_0(M_{j,k}),\dots, \tilde{M}_4(M_{j,k})$ be the matrices, by which
  $M_{j,k}$ is replaced in the application of the Conti construction of Lemma
  \ref{lem:convex_int}.
We consider first the pushed out matrices (see also Notation \ref{nota:conti}). If the edge of $\partial \conv (K)$ closest to $\tilde{M}_l(M_{j,k})$, $l=1,2,3$, is different from the edge closest to $M_{j,k}$, then by construction $d(\tilde{M}_l(M_{j,k}))\geq 1/16$. It thus remains to discuss the situation, in which this is not the case. In this situation the intercept theorem and the induction hypothesis, for $l=1,2,3,$ (for which $\epsilon_{j+1}|_{\tilde{\Omega}_{j+1,l}^k}=\epsilon_0$) it holds
  
  \begin{align*}
    d(\tilde{M}_{l}(M_{j,k})) 
    &\geq \frac{16}{15}d(M_{j,k})- \epsilon_{0} \\
    &\geq \frac{1}{15}d(M_{j,k}) - \epsilon_{0} +  \min\left\{\frac{1}{16},d(M)\right\}-2(\epsilon_0 - \epsilon_j|_{\Omega_{j,k}}) \\
    &\geq \min\left\{\frac{1}{16},d(M)\right\} + \frac{98}{15}\epsilon_0 - 3\epsilon_0\\
    &\geq \min\left\{\frac{1}{16},d(M)\right\}.
  \end{align*}
Here we used the definition of $\epsilon_0$ (c.f. Step 0 (b)) and the induction hypothesis combined with the bound
  \begin{align*}
    d(M_{j,k}) &\geq \min\left\{d(M),\frac{1}{16}\right\}-2\epsilon_0=98 \epsilon_{0}.
  \end{align*}
  Finally, for $\tilde{M}_{4}(M_{j,k})$ we estimate
  \begin{align*}
    d(\tilde{M}_{4}(M_{j,k})) 
    &\geq d(M_{j,k})- \epsilon_{j}|_{\Omega_{j,k}}
    \geq  \min\left\{d(M), \frac{1}{16} \right\}  +(2\epsilon_{j}|_{\Omega_{j,k}}-\epsilon_j|_{\Omega_{j,k}})-2\epsilon_{0}\\
    & =  \min\left\{d(M), \frac{1}{16} \right\} + 2(\epsilon_{j+1}|_{\tilde{\Omega}_{j,l}^k} -\epsilon_{0}).
  \end{align*}
  This concludes the proof.
\end{proof}

\subsubsection{Skew symmetric part}
\label{sec:skew}
In order to deal with the skew part and to show its boundedness, we need  several auxiliary results. These are targeted at controlling the maximal number of push-out steps in the parallel case (c.f. Lemma \ref{lem:skew1}), where the notions ``parallel" and ``rotated" are used as in Definition \ref{defi:parallel_rot}. With the control of the maximal number of push-out steps at hand, we can then present a bound on the skew part of the gradients from Algorithms \ref{alg:construction} and \ref{alg:skew} (c.f. Proposition \ref{prop:skewcontrol1}). Together with the boundedness of the symmetrized gradient this yields the uniform $L^{\infty}$ bounds on $\nabla u_j$.\\

We begin by estimating the distance to the wells. 

\begin{lem}
\label{lem:dist}
  Let $SP_{j}$ be the j-th step of the convex integration construction obtained
  in Proposition \ref{prop:convex_int}. Then for every level set $\Omega_{j,k}$
  it holds
  \begin{align*}
    &\dist(e(\nabla u_{j}|_{\Omega_{j,k}}), K) 
    \geq \min\{d_K, 1/8\} - 2(\epsilon_0-\epsilon_{j}|_{\Omega_{j,k}}),
   \end{align*} 
   where $d_0$, $d_K$, $\epsilon_0$ are as in Step 0(b) in Algorithm \ref{alg:construction}.
\end{lem}

The statement of this lemma is very similar to the result of Proposition \ref{prop:symcontrol}. However, instead of controlling the distance to the boundary, we here estimate the distance to the wells. This can be substantially larger than the distance to the boundary.

\begin{proof}
The proof follows along the same lines of the one of Proposition \ref{prop:symcontrol}.
  We note that the statement is true for $j=0$ (by the definition of
  $\epsilon_{0}$) and proceed by induction. 
  Let thus $M_{j,k}:=\nabla u_{j}|_{\Omega_{j,k}}$ be given. With slight abuse
  of notation we set $\epsilon_{j}:=\epsilon_j|_{\Omega_{j,k}}$. It suffices to
  show that the values of $\nabla u_{j+1}$, which were obtained from $M_{j,k}$ by
  an application of the Conti construction, still satisfy the desired estimates (in the domains, in which $u_j$ is unchanged the estimate holds by the induction assumption). The application of the Conti construction yields matrices $\tilde{M}_0(M_{j,k}), \cdots,\tilde{M}_4(M_{j,k})$.
  As $e(\tilde{M}_0(M_{j,k}))\in K$, we only consider the other matrices. We
  consider the matrices $\tilde{M}_1(M_{j,k}),\cdots,\tilde{M}_3(M_{j,k})$,
  which are constructed by ``pushing-out" (c.f. Notation \ref{nota:conti}). Without loss of generality (c.f. the argument in Proposition \ref{prop:symcontrol}), we only discuss the case that the closest well for $e(\tilde{M}_i(M_{j,k}))$ is the same as for $e(M_{j,k})$. For $i\in\{1,2,3\}$ we have
  \begin{align*}
  \dist(e(\tilde{M}_i(M_{j,k})),K) &\geq \frac{16}{15}\dist(e(M_{j,k}),K)-\epsilon_0\\
  & \geq d_K + \frac{1}{15}\dist(e(M_{j,k}),K) - \epsilon_0 - 2(\epsilon_0-\epsilon_{j})\\
  & \geq d_K + \frac{1}{15} d(M_{j,k}) - \epsilon_0 - 2(\epsilon_0-\epsilon_{j})\\
  & \geq d_K.
  \end{align*}
  Here we used the induction assumption for $M_{j,k}$ as well as the estimate for $d(M_{j,k}) $ from Proposition \ref{prop:symcontrol} and the definition of $\epsilon_0$.\\
  For $\tilde{M}_4(M_{j,k})$ we estimate
  \begin{align*}
   \dist(e(\tilde{M}_4(M_{j,k})),K) 
   &\geq \dist(e(M_{j,k}),K)-\epsilon_j\\
  & \geq d_K  - \epsilon_j - 2(\epsilon_0-\epsilon_{j})\\
 & \geq d_K - 2(\epsilon_0-\epsilon_{j+1}).
  \end{align*}
  Here we used the definition of $\epsilon_{j+1}:= \epsilon_{j}/2$ on the subset of the Conti construction, on which $\tilde{M}_4(M_{j,k})$ is attained. 
\end{proof}

Using Lemma \ref{lem:dist} and recalling Definitions \ref{defi:family}, \ref{defi:parallel_rot}, we bound the maximal number of possible push-out steps in the parallel situation:

\begin{lem}
\label{lem:skew1}
Let $SP_j$ and $\omega_j$ be as in Algorithms \ref{alg:construction} and \ref{alg:skew}. Assume that $D\in \{\Omega_{j_0 + n,k}\}_{k\in\{1,\dots,J_{j_0+n}\}}$ and suppose that the construction of $D$ from $\bar{D}\in \mathcal{P}_{n}(D)$ involves $k$ with $k\in \N\cup \{0\}$ push-out steps (c.f. Notation \ref{nota:conti}). Further assume that for a.e. $x\in D$ and for all $r \in\{1,\dots,n\}$
\begin{align}
\label{eq:parallel_1}
e^{(p)}_{j_0 + r}(x)= e^{(p)}_{j_0}(x) .
\end{align}
Then there exists a number $N_0=N_0(d_K)$ such that $0\leq k \leq N_0$.
\end{lem}

\begin{proof}
The proof relies on the definition of $\epsilon_0$ and the control on the distance to the wells, which was obtained in Lemma \ref{lem:dist}. Indeed, let $\Omega_{j_0+l,m}\in \{\Omega_{j_0+l,\tilde{m}}\}_{\tilde{m}\in\{1,\dots,J_{j_0+l}\}}$ with $\Omega_{j_0+l,m} \subset \bar{D}$ be arbitrary but fixed. Without loss of generality, we assume that in all the iteration steps $j_0,\dots,j_0+l$ Step 2(b) occurs on our respective domain (as there is no change, if Step 2(a) occurs, and as we are only interested in the maximal number of steps, in which a specific change, i.e. a push-out, occurs).
Let $M_j:= \nabla u_{j}|_{\Omega_{j_0+l,m}}$ be given.
Suppose that a matrix $M_{j_0+n+1}$ is obtained from $M_{j_0+n}$ for some $n\in\{1,\dots,l\}$ by push-out and that $M_{j_0+n}$ is obtained from $M_{j_0}$ by stagnating $n$-times. Then,
\begin{align*}
\dist(e(M_{j_0+n+1}),e_{j_0}^{(p)})
&\geq \frac{16}{15}\dist(e(M_{j_0+n}),e_{j_0}^{(p)}) - \epsilon_0\\
&\geq \frac{16}{15} \dist(e(M_{j_0}), e_{j_0}^{(p)}) 
- \frac{16}{15}\epsilon_{j_0} \sum\limits_{j=1}^{n}2^{-j} - \epsilon_0\\
&\geq \frac{16}{15} \dist(e(M_{j_0}), e_{j_0}^{(p)}) - \frac{32}{15}\epsilon_{j_0}- \epsilon_0\\
& \geq \frac{101}{100}\dist(e(M_{j_0}), e_{j_0}^{(p)}). 
\end{align*}
Here we have used (\ref{eq:parallel_1}), the result of Lemma \ref{lem:dist}, the fact that each consecutive stagnation step decreases the value of $\epsilon_{j}$ by a factor $2^{-1}$ and the definition of $\epsilon_0$. Thus, defining $k$ as the number of push-out steps, we infer that
\begin{align*}
\dist(e(M_{j_0+l}),e_{j_0}^{(p)})
 &\geq \left( \frac{101}{100} \right)^k \dist(e(M_{j_0}),e_{j_0}^{(p)})
 \geq \left( \frac{101}{100} \right)^k (d_K - \epsilon_0)\\
 &\geq \left( \frac{101}{100} \right)^k \frac{99}{100} d_K,
\end{align*}
where $d_K$ is defined as in Step 0 (b) in Algorithm \ref{alg:construction}. 
Therefore, by Step 2 of Algorithm \ref{alg:construction} (i.e. the update for $e^{(p)}_j$) after at most
\begin{align*}
N_0:=  \frac{\log(\frac{3}{d_K})}{\log\left( \frac{101}{100} \right)}
\end{align*}
push-out steps, we are no longer in the parallel case. This yields the desired upper bound.
\end{proof}

Relying on the previous lemma, we obtain a uniform bound on the skew part:

\begin{prop}[Skew symmetric part]
\label{prop:skewcontrol1}
Let $SP_j$ and $\omega_j$ be as in Algorithms \ref{alg:construction} and \ref{alg:skew}. Suppose that $N_0>0$ is the number from Lemma \ref{lem:skew1}. Define $\bar{C}:=\max\{ 100, 20 (N_0 +1) (1+\epsilon_0) \}$. Then, 
\begin{align}
\label{eq:skewcontrol}
|\omega_j(x)| \leq \bar{C} + 2(\epsilon_0 - \epsilon_j) \mbox{ for all } x\in \Omega_j,
\end{align}
and 
\begin{align}
\label{eq:skewcontrol1}
|\omega_j(x)| \leq 2 \bar{C}  \mbox{ for all } x\in \Omega \setminus\Omega_j.
\end{align} 
\end{prop}

\begin{proof}
We prove the claims inductively and note that $\omega_0=0$ satisfies them.
We first discuss (\ref{eq:skewcontrol}) and show that it remains true for $\omega_{j}$ with $j\in \N$. To this end, let $l\in \N$ and 
$D\subset \{\Omega_{j+l,k}\}_{k\in\{1,\dots,J_{j+l}\}}$.
For abbreviation we set
$M_{j}:= \nabla u_j|_{D}$, $\tilde{\omega}_{j}:=\omega_j|_{D}$ (and recall that $\omega_j|_D = \omega(\nabla u_j|_D)$) and first assume that $\tilde{\omega}_{j} \leq 0$ (see Notation \ref{not:skew}).  We begin by making the following additional assumption:

\begin{assume}
\label{as:1}
We suppose that the skew matrix 
$\tilde{\omega}_{j+l}$ is derived from $\tilde{\omega}_{j}$ by an $l$-fold application of Algorithms \ref{alg:construction} and \ref{alg:skew}, where in the Conti construction of Corollary \ref{cor:plusminus_construction} we always choose the \emph{positive} skew direction. 
\end{assume}

We point out that this assumption can occur both in the parallel and in the rotated case, but ensures that the skew direction was not changed in this process. In other words, Assumption \ref{as:1} implies that the sign of the skew direction, which is chosen in Corollary \ref{cor:plusminus_construction}
remains fixed. We hence refer to this situation as the ``fixed sign case".
We further introduce the auxiliary functions
\begin{align*}
N_{l,1}, N_{l,2}, N_{l}: \Omega \rightarrow \N\cup \{0\},
\end{align*} 
with $N_l:= N_{l,1}+N_{l,2}$. Here for given $l\in \N$ and $\tilde{\omega}_{j+l}$, we define $N_{1,l}$ as the number of $4/3$ push-out steps in the process of obtaining $\tilde{\omega}_{j+l}$ from $\tilde{\omega}_{j}$, and $N_{2,l}$ as the number of $16/15$ push-out steps. By Lemma \ref{lem:skew1} we know that $0\leq N_{l}\leq N_0$. \\

\emph{Step 1: Upper bound in the fixed sign case.}
We first deal with the upper bound for $\tilde{\omega}_{j+l}$. To this end we note that
\begin{align*}
\tilde{\omega}_{j+1} \leq \frac{4}{3} \dist(e(\nabla u_j)|_{D},K) + \epsilon_j|_{D}  + \tilde{\omega}_{j}.
\end{align*}
We iterate this estimate:
\begin{align*}
\tilde{\omega}_{j+l} \leq \frac{4}{3} N_{l,1} 3 + \frac{16}{15} N_{l,2} 3 + (N_{l,1}+N_{l,2}) \epsilon_0 + \epsilon_0 N_l \sum\limits_{m=1}^{l} 2^{-m} + \tilde{\omega}_{j}.
\end{align*}
Here we used the estimate $\dist(e(\nabla u_j)|_{D},K) \leq 3$, the fact that in each push-out step an error $\epsilon_0$ is possible, while in each stagnant step the error is decreased by a factor two. Recalling the definition of $\bar{C}$ and the fact that $N_{l}\leq N_0$ hence implies
\begin{align}
\label{eq:upper_bound}
\tilde{\omega}_{j+l} \leq \bar{C}/2 + \tilde{\omega}_{j}.
\end{align}
Using that $\tilde{\omega}_{j}\leq 0$, therefore allows us to conclude that
\begin{align*}
\tilde{\omega}_{j+l} \leq \bar{C}/2.
\end{align*}
\emph{Step 2: Lower bound in the fixed sign case.}
Still working under the assumptions from above, we now bound the negative part of $\tilde{\omega}_{j+l}$. Here we show that
\begin{align*}
\tilde{\omega}_{j+l} \geq -\bar{C} - 2(\epsilon_0 - \epsilon_{j+l}|_D).
\end{align*}
We first consider the push-out steps. Let $\tilde{M}_1(M_{j+l-1}), \tilde{M}_2(M_{j+l-1}), \tilde{M}_3(M_{j+l-1})$ be the push-out matrices in the corresponding Conti construction of Algorithms \ref{alg:construction}, \ref{alg:skew}. Their skew parts are contained in
  $\epsilon_{0}$ neighborhoods of
  \begin{align*}
    \omega(M_{j+l-1}) + \frac{1}{3} \hat{S}, \omega(M_{j+l-1}) + \frac{1}{15} \hat{S}.
  \end{align*}
By Lemma \ref{lem:skewcontrol2}, Lemma \ref{lem:skewcontrol} and Proposition
  \ref{prop:symcontrol}, we obtain that
  \begin{align*}
    10 \geq \hat{S} \geq \frac{3}{8}|a \odot n| 
    \geq \frac{3}{4}\dist(e(M_{j+l-1}), e^{(p)}_{j+l-1}|_{D}) 
    \geq  \frac{3}{4} d(M_{j+l-1}) 
    \geq \frac{3}{4}98\epsilon_0.
  \end{align*}
  Here $d:\R^{2\times 2} \rightarrow \R$ denotes the function from Proposition \ref{prop:symcontrol}, and $a\otimes n$ is the rank-one connection, which appears in the Conti construction. In the last estimate we have used the estimate from Proposition \ref{prop:symcontrol}.
  Hence the skew parts of $\tilde{M}_{1}(M_{j+l-1})$, $ \tilde{M}_{2}(M_{j+l-1}),\tilde{M}_{3}(M_{j+l-1})$
  are respectively bounded by
  \begin{align*}
 \omega(\tilde{M}_{i}(M_{j+l-1})) 
  &\geq \tilde{\omega}_{j+l-1} 
  + \frac{1}{15} \frac{3}{4}98 \epsilon_0 - \epsilon_0\\
  & \geq -\bar{C} - 2(\epsilon_0 -\epsilon_{j+l}|_D)
    + \frac{1}{15} \frac{3}{4}98 \epsilon_0 - \epsilon_0 \geq -\bar{C} \mbox{ for } i\in\{1,2,3\},
  \end{align*} 
  which shows the claimed estimate (\ref{eq:skewcontrol}) with $\epsilon_j=\epsilon_0$. For $\tilde{M}_4(M_{j+l-1})$ we have that
  \begin{align*}
  \omega(\tilde{M}_4(M_{j+l-1}))&\geq -\bar{C} - 2(\epsilon_0 - \epsilon_{j+l-1}|_D) 
  - \epsilon_{j+l-1}|_D\\
   & \geq -\bar{C} - 2(\epsilon_0 - \epsilon_{j+l}|_D),
  \end{align*}
  which also proves the desired result. This concludes the proof of (\ref{eq:skewcontrol}) in the fixed sign case. \\
  
\emph{Step 3: Sign change.}  
  Let $j+l+1$ be the first index, in which the sign of the difference of the skew parts changes according to Algorithm \ref{alg:skew}. By Assumption \ref{as:1} and by the definition of our Algorithms \ref{alg:construction}, \ref{alg:skew}, this can only be the case if $\tilde{\omega}_{j+l}\geq 0$. The definition of $\bar{C}$ ensures that the $\tilde{M}_4(M_{j+l+1})$ obeys the upper bound
\begin{align*}
\omega(\tilde{M}_4(M_{j+l+1}))
\leq \frac{\bar{C}}{2}+\epsilon_0 
\leq \bar{C}.
\end{align*}  
For the pushed out parts, $\tilde{M}_1(M_{j+l+1}),\tilde{M}_2(M_{j+l+1}),\tilde{M}_3(M_{j+l+1})$, we argue similarly as we did in Step 1, but now with a change of signs: By the intercept theorem, the resulting skew parts become strictly smaller than the one of $\tilde{\omega}_{j+l}$ (potentially they even become negative). This then improves the upper bound (\ref{eq:upper_bound}). For the lower bound we argue as in Step 1 but with reversed sign in Assumption \ref{as:1}. This concludes the proof of (\ref{eq:skewcontrol}).\\

\emph{Step 4: Proof of (\ref{eq:skewcontrol1}).}  
In order to obtain the estimate (\ref{eq:skewcontrol1}), we notice that the skew parts associated with values of $e(\nabla u_j)\in K$ may on the one hand be strictly larger than the bound given in (\ref{eq:skewcontrol}). But on the other hand, they are derived as an $\tilde{M}_0$ matrix in one of the Conti constructions, in which matrices satisfying (\ref{eq:skewcontrol}) are modified. This implies that at most a gain of $5$ in the modulus of the corresponding skew part is possible, which yields the bound (\ref{eq:skewcontrol1}). As these domains are not further modified in the convex integration algorithm this bound cannot deteriorate in the course of the application of Algorithms \ref{alg:construction} and \ref{alg:skew}.
\end{proof}

\subsection{Existence of convex integration solutions}
\label{sec:exist}
Finally, in this last subsection, we show that Algorithms \ref{alg:construction}, \ref{alg:skew} can be used to deduce the existence of solutions to our problem (\ref{eq:incl}).

\begin{prop}[Convex integration solutions]
  \label{prop:convex_int}
   Let $M \in \R^{2\times 2}$ with $e(M) \in \intconv(K)$. Let $\Omega \subset \R^{2}$ be open and bounded. Then there exists a Lipschitz function $u: \R^{2} \rightarrow \R^{2}$
    such that
    \begin{align*}
     & \nabla u =M \text{ a.e. in }\R^{2}\setminus \Omega, \\
     & e(\nabla u) \in K \text{ a.e. in } \Omega.
    \end{align*}
\end{prop}

\begin{proof}
We apply Algorithm \ref{alg:construction} with $\bar{M}:= M-\omega(M)$.
By the results of Propositions \ref{prop:symcontrol} and \ref{prop:skewcontrol1} this algorithm is well-defined and can be iterated with $j \rightarrow \infty$. This yields a sequence of functions $u_j:\R^2 \rightarrow \R^2$ with bounded gradient (with $\|\nabla u_j\|_{L^{\infty}(\R^2)}$ depending on $d_K$, c.f. Lemma \ref{lem:skew1}).
We prove the convergence of this sequence and show that the limiting function $u_0$ solves (\ref{eq:K3}) with boundary data $\bar{M}$.\\
We note that for $k\geq j$
\begin{align}
\label{eq:unchanged}
  \nabla u_k(x) = \nabla u_{j}(x) \mbox{ for a.e. $x$ in } \Omega \setminus\Omega_j.
\end{align}
Moreover, 
\begin{align*}
\nabla u_k = \bar{M} \mbox{ a.e. in } \R^2 \setminus \overline{\Omega}.  
\end{align*}
By construction $\nabla u_j$ is bounded, hence $\nabla u_j \rightharpoonup \nabla u_0$ in the $L^{\infty}_{loc}(\R^2)$ weak-$\ast$ and the $L^2_{loc}(\R^2)$ weak topologies. By Poincar\'e's inequality
$u_j \rightarrow u_0$ in $L^2_{loc}(\R^2)$. 
We observe that Step 2 in Algorithm \ref{alg:construction} decreases the
total volume of the $\Omega_j$, i.e. of the part of $\Omega$, on
which $e(\nabla u_j)$ does not yet attain one of the wells:
\begin{align*}
|\Omega_j|= |\{ x \in U: e(\nabla u_j)\not \in K\}| \leq \left(1-v_0\frac{7}{8}\right)^{j}|\Omega|.
\end{align*}
Combined with \eqref{eq:unchanged} and the $L^{\infty}$ bound, this implies the desired
convergence $\nabla u_j \rightarrow \nabla u_0 $ with respect to the $L^2_{loc}(\R^2)$ topology, where 
$\nabla u_0 \in L^{\infty}(\R^2)$ is a solution to the problem \eqref{eq:K3} with
boundary data $\bar{M}$.\\
Defining $u(x):=u_0(x) + \omega(M)$ hence concludes the proof of 
Proposition \ref{prop:convex_int}.
\end{proof}

In Sections \ref{sec:covering} and \ref{sec:quant} we present a more refined
analysis of this construction algorithm. In particular, we give an explicit
quantitative construction for the covering procedure from Step 2 in  Algorithm \ref{alg:construction}.

\section{Covering Constructions}
\label{sec:covering}

In the following section we present the details of the coverings, which we use in the Algorithms \ref{alg:construction}, \ref{alg:skew}. Here we pursue two (partially) competing objectives: Given a triangle $D$, 
\begin{itemize}
\item we seek to cover an as large as possible volume fraction of it, but at least a given fixed volume fraction, $v_0>0$.
\item We have to control the perimeters of the triangles in the resulting new covering.
\end{itemize}
In the context of these considerations, it turns out that the parallel and the rotated cases (c.f. Definition \ref{defi:parallel_rot}) differ quantitatively
and hence have to be discussed separately. This can be understood when
considering possible coverings of rectangles by parallel or rotated rectangles.

\begin{figure}[t]
  \centering
  \includegraphics[width=0.9\linewidth, page=47]{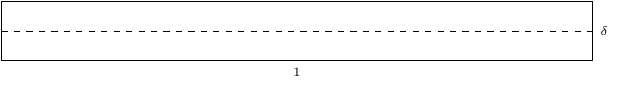}
  \includegraphics[width=0.9\linewidth, page=48]{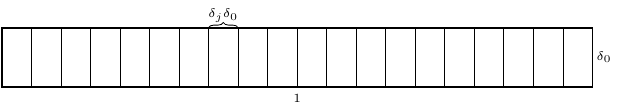}
  \caption{Covering a rectangle $R_{1,\delta_0}$ of side lengths $1$ and $\delta_0$ by (a) a parallel rectangle of half its aspect
   ratio, (b) an orthogonal rectangle of aspect ratio $r=\delta_j$ (which could for instance be $r=\delta_0/2$).  }  
    \label{fig:rectangle_coverings}
\end{figure}

We illustrate this in two extreme situations (c.f. Figure
\ref{fig:rectangle_coverings}): Given a rectangle $R_{1,\delta_0}$ with sides
of length $1$ and $\delta_0$, we seek to cover it with rectangles, which have a
fixed side ratio $r$ and whose long sides are either parallel or orthogonal to
the long side of the original rectangle $R_{1,\delta_0}$. In order to
illustrate the differences between these situations, we for instance assume that
$r=\delta_0/2$. In the situation, in which the original rectangle
$R_{1,\delta_0}$ is covered by rectangles, whose long side is \emph{parallel}
to the long side of $R_{1,\delta_0}$, the covering can be achieved by
splitting $R_{1,\delta_0}$ along its central line as illustrated in Figure
\ref{fig:rectangle_coverings} (a). Thus, the resulting perimeter (we view it as a measure of the $BV$ energy of the characteristic functions in the Conti
covering), which is necessary to cover the volume of $R_{1,\delta_0}$ is
bounded by twice the perimeter of $R_{1,\delta_0}$.
If the long sides of the covering rectangles of ratio $\delta_0/2$ are however \emph{orthogonal} to the long side of $R_{1,\delta_0}$, the covering of $R_{1,\delta_0}$ can only be achieved by $2\delta_0^{-2}$ small rectangles of side lengths $\delta_0$ and $\delta_0^2/2$ (c.f. Figure \ref{fig:rectangle_coverings} (b)). The necessary perimeter for this covering is thus proportional to $\delta_0^{-1} \Per(R_{1,\delta_0})$. \\
For a small value of $\delta_0$ this makes a substantial difference and accounts for the losses in the estimates for the rotated situation.\\
The difference of the parallel and the rotated situation become even more apparent, if we consider a sequence of coverings:
Here we start with the rectangle $R_{1,\delta_0}$ and first consider an iterative covering of it by \emph{parallel} rectangles, which in the $j$-th iteration step are of side ratio $\delta_j:=2^{-j}\delta_0$ (and such that the long side is parallel to the long side of $R_{1,\delta_0}$). The desired covering of $R_{1,\delta_0}$ in the iteration step $k$ can be achieved by splitting the rectangles from the covering at the iteration step $k-1$ along their central lines. In each iteration step the overall perimeter increases at most by a factor two, so that after $j$ iteration steps the overall perimeter can be estimated by
\begin{align*}
2^{j} \Per(R_{1,\delta_0}).
\end{align*}
If in comparison, we consider the case, in which the covering rectangles are rotated in every step by $\pi/2$ with respect to the preceding rectangles and again choose a ratio $\delta_j:= 2^{-j}\delta_0$ in the $j$-th iteration step, we inductively obtain a bound of the form
\begin{align*}
\left(\prod\limits_{l=1}^{j} \delta_l^{-1}\right) \Per(R_{1,\delta_0})
= \delta_0^{-j}\left(\prod\limits_{l=1}^{j} 2^{l}\right)\Per(R_{1,\delta_0})
\end{align*}
for the overall perimeter after the $j$-th step. In contrast to the parallel situation this has \emph{superexponential} behavior in $j$.\\
If we consider the $\pi/2$ rotated situation with fixed ratio $\delta_j=\delta_0$, this bound improves to an \emph{exponential} bound of the form
\begin{align*}
\delta_0^{-j} \Per(R_{1,\delta_0}).
\end{align*}
Hence, the estimates in the rotated situation are substantially worse 
than the ones in the parallel situation. In order to avoid 
superexponential behavior, we have to take care that the rotated case 
can only occur, if the value of $\delta_j$ is controlled from below. 
These heuristics a posteriori justify our careful choice of $\epsilon_j$ and $
\omega_j$ in Algorithms \ref{alg:construction}, \ref{alg:skew}.\\

Although the level sets of the Conti construction consist of triangles and hence our coverings $\{\Omega_{j,k}\}_{k\in\{1,\dots,J_j\}}$ will be coverings of triangles by triangles (instead of the previously described rectangular coverings), the heuristics from above still persist.  \\

Motivated by these heuristic considerations, in the sequel we seek to provide covering results and associated $BV$ bounds, which can be applied in Algorithms \ref{alg:construction}, \ref{alg:skew}. 
We organize the discussion of this as follows: In Section  \ref{sec:pre1}, we introduce some of the fundamental objects (c.f. Definitions \ref{defi:goodtriangle}, \ref{defi:classes}) and formulate the main covering result (Proposition \ref{thm:covering}). Here we consider a similar distinction into a parallel and a rotated situation as described in the above heuristics (c.f. Definition \ref{defi:classes}). With the class of triangles from Definition \ref{defi:goodtriangle} at hand we distinguish several different cases and discuss different covering scenarios. The respective coverings are tailored to the specific situation and are made such that we do not leave our class of triangles during the iteration. Their discussion is the content of Sections \ref{sec:cov_Cont}-\ref{sec:rota}. Finally, the various different cases are combined in Section \ref{sec:thm} to provide the proof of Proposition \ref{thm:covering}.

\subsection{Preliminaries}
\label{sec:pre1}
In this section we introduce the central objects of our covering (c.f. Definition \ref{defi:goodtriangle}, \ref{defi:classes}) and state our main covering result (Proposition \ref{thm:covering}). \\

As a preparation for the main part of this section, we begin by discussing auxiliary results on matrix space geometry. We first estimate the angle formed in strain space between two matrices:

\begin{lem}
\label{lem:angle_1}
Let $\tilde{d}>0$.
  Let $M^{(1)},M^{(2)}$ be matrices with $d(e(M^{(i)}),e^{(1)})\geq \tilde{d}$ and let $a^{(i)} \otimes n^{(i)}$ be
  the associated rank-one connection to $e^{(1)}+ \hat{S}_i$, where the skew matrices $\hat{S}_i$ are as in Corollary \ref{cor:plusminus_construction}.
  Suppose that $|M^{(1)}-M^{(2)}|\leq \epsilon_{j} \ll \tilde{d}$, then the angle
  $\alpha_1$ between $n^{(1)}$ and $n^{(2)}$ satisfies
  \begin{align*}
    |\alpha_1| \leq 3 \frac{\epsilon_{j}}{\tilde{d}}.
  \end{align*}
\end{lem}

\begin{rmk}
\label{rmk:angle_1}
Applied to a triangle $D \in \{\Omega_{j,k}\}_{k\in\{1,\dots,J_j\}}$ in the parallel case (c.f. Definition \ref{defi:parallel_rot}), Lemma \ref{lem:angle_1} implies that the rotation angle $\alpha_1$, with which the consecutive Conti constructions are rotated with respect to each other (and which is defined as in Lemma \ref{lem:angle_1}), is bounded by
\begin{align*}
|\alpha_1| \leq 3 \frac{\epsilon_j|_D}{\dist(e(\nabla u_j)|_{D}, K)} \leq 300 \delta_j|_D.
\end{align*}
Here $\epsilon_j, \delta_j$ are the functions from the convex integration Algorithm 
\ref{alg:construction}.
\end{rmk}

\begin{figure}[t]
  \centering
  \includegraphics[width=0.4\linewidth, page=1]{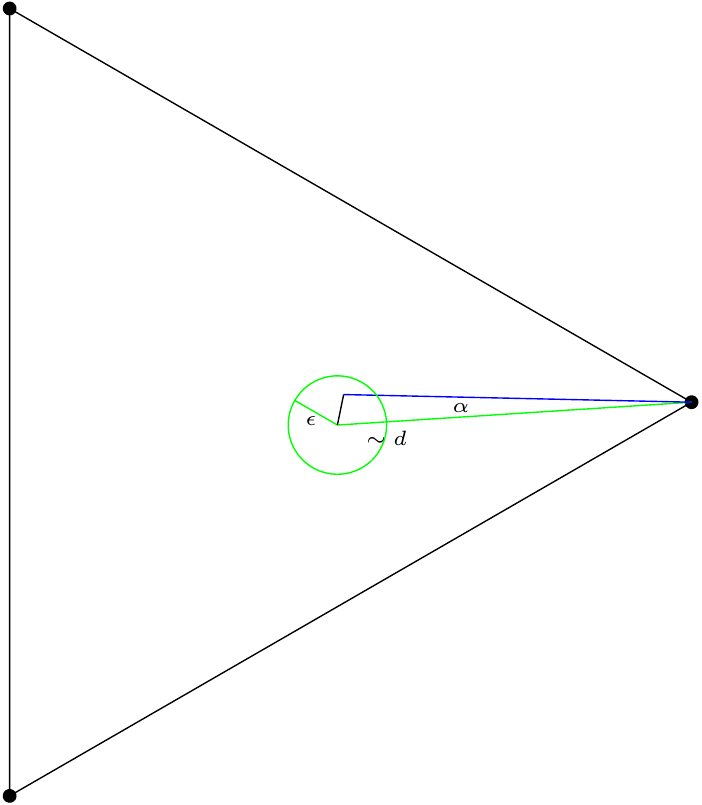}
  \caption{The angle between rank-one connections of nearby matrices is small.
    In particular, the associated rectangle constructions from Lemma \ref{lem:conti_deformed} are close to being
    parallel.}
  \label{fig:angle_nearby_matrices}
\end{figure}

\begin{proof}[Proof of Lemma \ref{lem:angle_1}]
  As sketched in Figure \ref{fig:angle_nearby_matrices}, we may estimate
  \begin{align*}
    |\alpha_1| \leq 1.5 |\tan(\alpha_1)| \leq 3 \frac{\epsilon_{j}}{\tilde{d}}, 
  \end{align*}
  where $\epsilon_j$ is the error in matrix space. Here the first estimate follows by a Taylor approximation and by noting that $|\alpha_1|$ is small, so that in particular
  $|\alpha_1| \leq \frac{\pi}{6}$ (in which range the tangent is invertible and for which the Taylor expansion is valid).
\end{proof}

Next we observe the following bounds on the rotation angles:

\begin{lem}[Angles]
\label{lem:angles}
Let $D\in \{\Omega_{j,k}\}_{k\in\{1,\dots,J_{j}\}}$ for $j\geq 1$. 
Let $\hat{D} \in \mathcal{D}_1(D)$. Assume that  the triangle $\hat{D}$ is in
the rotated case (c.f. Definition \ref{defi:parallel_rot}). Let $\alpha_1$ denote the angle between the long sides of the current and the following Conti constructions.
Then, we have that
\begin{align*}
0<C\delta_0<|\alpha_1| \leq \pi-C\delta_0.
\end{align*}
\end{lem}

\begin{proof}
This is an immediate consequence of Lemma \ref{lem:angle_a}.
\end{proof}

With these auxiliary results at hand, we proceed to the discussion of our central covering objects.
In order to define our set of covering triangles, $\{\Omega_{j,k}\}_{k\in\{1,\dots,J_j\}}$, we consider a subclass of triangles with, for our purposes, suitable properties.
To this end, we can not ensure that all domains appearing in our covering argument are right angle
triangles (due to the presence of the green triangles in the Conti construction in Figure \ref{fig:good_triangles_conti}), for which one of the other angles is approximately of size $\delta_{j}$. However, the following definition
provides a family of sets with similar properties. This will allow us to formulate a precise, iterative covering result.

\begin{defi}
\label{defi:goodtriangle}
Let $\delta_0$ be as in Algorithm \ref{alg:construction}.
  The triangle $D$ is said to be
  \emph{$\delta$-good with respect to a reference direction $n\in S^1$}, for $
  \delta \in (0,\delta_0]$, if
  \begin{enumerate}
  \item One angle, $\alpha$, satisfies $\alpha \in \delta[\frac{1}{10},1000]$, 
  \item The other two angles are contained in $\frac{\pi}{2} + 2 \delta[-1000, 1000]$, 
  \item One of the long sides encloses an angle in $\delta[-1000,1000]$
    with $n$.
  \end{enumerate}
  We refer to the long side of the triangle, which satisfies the requirement of 3.) as the \emph{direction of $D$}. We also say that the triangle $D$ is \emph{oriented parallel to $n$} or that \emph{$D$ is aligned to $n$.}\\
If a triangle $D$ satisfies 1.) and 2.) but not necessarily
   3.) we call it \emph{$\delta$-good} (which allows for a possible change of
   orientation).
 \end{defi}

\begin{rmk}
\label{defi:triangle}
We note that if $\alpha$ is small, both long sides could satisfy condition 3.) at the same time. In this case both directions are valid as \emph{directions of $D$}.
\end{rmk}

 In our construction one prominent reference direction is obtained from Conti's
 construction, as detailed in the following definition.

 \begin{defi}
 \label{defi:Conti_direction}
   Let $D$ be a level set of $\nabla u_j$ and let $e= e^{(p)}_j|_D$ be the
   reference well. Let further $n \in S^1$ be the direction of the long side of
   the Conti rectangle from Step 2 in Algorithm \ref{alg:skew}.
   We say that $n$ is the \emph{direction of the relevant Conti construction (at step $j$)}.
   A $\delta$-good triangle is \emph{parallel to Conti's
     construction} if one of its long sides is parallel to $n$.
 \end{defi}

In the sequel, we will give a precise covering result, which shows that in the
$j$-th step of our convex integration Algorithms \ref{alg:construction} and
\ref{alg:skew}, we may assume that only very specific triangles are present in the collection $\{\Omega_{j,k}\}_{k\in\{1,\dots,J_j\}}$ as
(parts of) level sets of $\nabla u_j$. To this purpose we define the following classes of triangles:

\begin{defi}
\label{defi:classes}
Let $SP_j$ be as in the Algorithms \ref{alg:construction}, \ref{alg:skew}.  Then, a triangle $D_j \in \{\Omega_{j,k}\}_{k\in\{1,\dots,J_j\}}$ is in the case:
\begin{itemize}
\item[(P1),] if it is $\delta_j|_{D_j}$-good with direction $n \in S^1$, where $n$ denotes the direction of the
  relevant Conti construction.  
\item[(P2),] if $\delta_j|_{D_j} = \delta_0$ and $\delta_{j-1}|_{D_{j}}\neq \delta_0$ and if $D_j$ is $\delta_{j-1}|_{D_j}$-good with direction $n \in S^1$, where $n$ denotes the direction of the
  relevant Conti construction.
\item[(R1),] if $\delta_j|_{D_j}=\delta_0$, the triangle is $\delta_0$-good and if it forms an angle $\beta$ with $C\delta_0 \leq \beta \leq \frac{\pi}{2}- C \delta_0$ with respect to the direction of the relevant Conti construction (c.f. Lemma \ref{lem:angles}). 
\item[(R2),] if $\delta_j|_{D_j}=\delta_0$, the triangle is $\delta_{j-1}|_{D_j}$-good with $\delta_{j-1}|_{D_j} \neq \delta_0$ and if it forms an angle $\beta$ with $C\delta_0 \leq \beta \leq \frac{\pi}{2}- C \delta_0$ with respect to the direction of the relevant Conti construction. 
\item[(R3),] if $\delta_j|_{D_j}=\delta_0$ and if the triangle is right angled and such that 
\begin{itemize}
\item[(a)] the other angles are bounded from below and above by $C \delta_0$ and $\frac{\pi}{2}- C \delta_0$,
\item[(b)] one of its sides is parallel to the orientation of the relevant Conti construction.
\end{itemize}
\end{itemize}
\end{defi}

  \begin{rmk}
\label{rmk:nonunique}
    The cases above, as stated, are not distinct since we allow for a factor in
    our definition of being $\delta$-good (c.f. Definition \ref{defi:goodtriangle}). For instance, there might be triangles
    which are in both case $(P1)$ and $(P2)$. However, in such situations also the
    constructions and perimeter estimates are comparable.
    In situations, in which the estimates would differ significantly and where $\delta_{j-1} \ll
    \delta_j=\delta_0$, the above definitions yield distinct cases.
  \end{rmk}

Let us comment on this classification: The basic distinction criterion
separating the triangles into the different cases is given by checking whether
the corresponding triangles are roughly aligned (as in the cases (P1), (P2)) or
whether they are substantially rotated (as in the cases (R1), (R2)) with respect to the direction of the relevant Conti construction (the case (R3) is a special ``error situation", which does not entirely fit into this heuristic consideration). Roughly speaking, this determines whether we are in a situation analogous to the first or to the second picture in Figure \ref{fig:rectangle_coverings}. This distinction is necessary, as else a control of the arising surface energy is not possible in a, for our purposes, sufficiently strong form.\\
This distinction (essentially) coincides with our definition of the parallel and the rotated cases (c.f. Definition \ref{defi:parallel_rot}): If a triangle ${D_j} \in\{\Omega_{j,k}\}_{k\in\{1,\dots,J_j\}}$ is in the \emph{parallel case} in step $j$, then the directions of the Conti construction, which gave rise to ${D_j}$, and of the relevant Conti construction at step $j$ (i.e. the construction, by which ${D_j}$ is (in part) covered) are essentially parallel (c.f. Lemma \ref{lem:angle_1} and Remark \ref{rmk:angle_1}). Letting $j_0\in \N$ denote the index from Definition \ref{defi:parallel_rot} and assuming that $j_0=1$, there are three possible scenarios for the relation of $\delta_j|_{D_j}$ and $\delta_{j-1}|_{D_j}$:
\begin{itemize}
\item[(i)] $\delta_j|_{D_j} = \delta_{j-1}|_{D_j}/2$. In this case $\nabla u_j|_{{D_j}}$ was produced as the stagnant matrix (c.f. Notation \ref{nota:conti}) in the iteration step $j-1$. In this case, our covering construction will ensure that ${D_j}$ is in case (P1) (not exclusively, c.f. Remark \ref{rmk:nonunique}, but as one option).
\item[(ii)] $\delta_j|_{D_j}= \delta_0$ but $\delta_{j-1}|_{D_j}\neq \delta_0$. This can for instance occur in a parallel push-out step. In this case, our covering construction will ensure that ${D_j}$ is in case (P2) (not exclusively (depending on the value of $\delta_{j-1}$), c.f. Remark \ref{rmk:nonunique}, but as one option).
\item[(iii)] $\delta_j|_{D_j}= \delta_0= \delta_{j-1}|_{{D_j}}$. This case can for instance occur in two successive push-out steps. In this case, our covering ensures that ${D_j}$ is in the case (P1).
\end{itemize}
If a triangle ${D_j} \in\{\Omega_{j,k}\}_{k\in\{1,\dots,J_j\}}$ is in the \emph{rotated case} in step $j$, then the direction of the Conti construction, which gave rise to ${D_j}$, and the direction of the relevant Conti construction at step $j$ are necessarily substantially rotated with respect to each other (c.f. Lemma \ref{lem:angles}). Again assuming that the index $j_0=1$ (where $j_0$ denotes the index from Definition \ref{defi:parallel_rot}), we now distinguish two cases for the relation between $\delta_j|_{{D_j}}$ and $\delta_{j-1}|_{D_j}$: Here we first note that necessarily (by definition of the rotated case, as occurring only after a push-out step) we have $\delta_j|_{D_j} = \delta_0$. Then there are two options for $\delta_{j-1}|_{D_j}$:
\begin{itemize}
\item[(i)] $\delta_{j-1}|_{D_j} = \delta_0$. This case can for instance occur in the situation of two successive push-out steps. In this case our covering ensures that ${D_j}$ can be taken to be in the case (R1).
\item[(ii)] $\delta_{j-1}|_{D_j} \neq \delta_0$. This case can for instance occur in the case, in which $\nabla u_{j-1}|_{D_j}$ is produced in a stagnant and $\nabla u_j|_{D_j}$ in a push-out step. In this situation our covering ensures that ${D_j}$ can be taken to be in the case (R2). 
\end{itemize} 
The case (R3) only occurs as an error case as a consequence of our specific covering procedure for the triangles of the types (R1) and (R2).\\
We relate the different cases to the heuristics given at the beginning of Section \ref{sec:covering} (c.f. Figure \ref{fig:rectangle_coverings}).
We view the cases (P1) and (R1) as the ``model cases" without and with substantial rotation and corresponding to the parallel and orthogonal (triangular) situation depicted in Figure \ref{fig:rectangle_coverings}. In both cases (P1) and (R1) the aspect ratio of the given triangle $D_j$ is roughly of order $\delta_j|_{D_j}$ (i.e. the quotient of its shortest and of its longest sides are roughly of that order) and we seek to cover it with a Conti construction of comparable ratio $\delta_j|_{D_j}$.
\\
The cases (P2) and (R2) are situations, in which the underlying triangle $D_j$ is roughly of side ratio $\delta_{j-1}|_{D_j}$ (i.e. the quotient of its shortest and of its longest sides are roughly of that order), where we however seek to cover the triangle with Conti constructions with ratio $\delta_0$.
This mismatch is a consequences of our construction of the function $\delta_j|_{D_j}$ in Algorithm \ref{alg:construction}: Here we prescribe that the matrices, which are pushed out (c.f. Notation \ref{nota:conti}), are allowed to have an error tolerance of $\delta_0$. 
In particular, it may occur that $\delta_{j-1}|_{D_j}\ll \delta_j|_{D_j}=\delta_0$, which is the situation described in (P2), (R2) either without or with substantial rotation. \\
The case (R3) is a consequence of how we deal with ``remainders" in our covering constructions for the cases (R1), (R2). \\

Our main result of the present section states that it is possible to find a covering of the level sets, which respects Algorithms \ref{alg:construction}, \ref{alg:skew}, such that only the specific triangles from Definition \ref{defi:classes} occur. Moreover, we provide bounds for the remaining uncovered ``bad" volume and the resulting perimeters.

\begin{prop}[Covering]
\label{thm:covering}
Let $\Omega=Q_{\beta}[0,1]^2$, where $\beta$ is the rotation of the Conti construction adapted to the matrix $M$ from Algorithm \ref{alg:construction}. Let $u_j$ be as in Algorithms \ref{alg:construction}, \ref{alg:skew}. Then, there exists a covering 
$\{\Omega_{j,k}\}_{k\in\{1,\dots,J_{j}\}}$ such that only triangles of the classes (P1), (P2) and (R1)-(R3) occur and such that
\begin{itemize}
\item[(i)] $|\Omega \setminus \Omega_j| \leq (1-\frac{7}{8}v_0)^{j}|\Omega|$, 
\item[(ii)] $\sum\limits_{k=1}^{J_{j+1}}\Per(\Omega_{j+1,k}) 
   \leq C \delta_0^{-1} \sum\limits_{k=1}^{J_{j}}\Per(\Omega_{j,k}) .$
\end{itemize}
Here $v_0 \in (0,1)$ is a small constant, which is independent of $\epsilon_0$, $d_0$ and $d_K$.
\end{prop}

In the remainder of this section we seek to prove this result and to construct the associated covering. To this end, in Section \ref{sec:cov_Cont} we first explain that the ``natural covering" of the Conti construction, which is achieved by splitting it into its level sets, satisfies the requirements of Proposition \ref{thm:covering}. In particular, this implies that the covering, which is obtained in Step 1 of Algorithm \ref{alg:construction}, satisfies the properties of Proposition \ref{thm:covering} (the resulting triangles are of the types (P1), (P2) or (R1), (R2)). Hence, in the remaining part of the section, it suffices to prove that given a triangle of the type (P1)-(R3), we can construct a covering for it, which obeys the claims of Proposition \ref{thm:covering}. 
To this end, in Section \ref{sec:build}, we first describe a general construction, on which we heavily rely in the sequel. With this construction at hand, in Section \ref{sec:parallel} and its subsections we then deal with the cases (P1), (P2), in which there is no substantial rotation involved. Subsequently, we discuss the cases (R1)-(R3) with non-negligible rotations in Section \ref{sec:rota}. Finally, in Section \ref{sec:thm} we provide the proof of Proposition \ref{thm:covering}.\\

The generalization to more generic domains is detailed in Section \ref{sec:generaldomains}.

\subsection{Covering the Conti construction by triangles}
\label{sec:cov_Cont}

We begin with our covering construction by explaining that a Conti construction of ratio $\delta_j|_{D_j}$ can be divided into a finite number of triangles, which are all of the types (P1), (P2) and (R1)-(R3).

\begin{lem}
\label{lem:cov_Cont}
Let $SP_j$ be as in Algorithm \ref{alg:construction} and let $D_j \in \{\Omega_{j,k}\}_{k\in\{1,\dots,J_j\}}$.
Suppose that $R\subset D_j$ is a Conti rectangle of ratio $\delta_{j}|_{D_j}$. Let $M_1,\dots,M_4$ denote the gradients occurring in the Conti construction with the same convention as in Notation \ref{nota:conti}. Then all level sets in $R$, on which $ M_4$ is attained, 
  can be decomposed into (at most two) triangles, which are of the type (P1).
  The level sets with $M_1, M_2, M_3$ can be decomposed into triangles of the
  type (P1), (P2) or (R1), (R2). 
\end{lem}

\begin{figure}[t]
  \centering
  \includegraphics[width=0.4\linewidth, page=15]{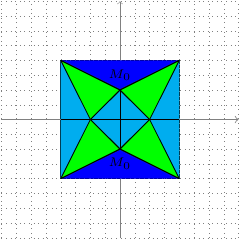}
  \caption{Triangles in the undeformed Conti construction}
  \label{fig:good_triangles_conti}
\end{figure}

\begin{proof}
  We recall that (after a suitable splitting into in total 16 triangles as depicted in Figure \ref{fig:good_triangles_conti}) all except for four triangles in the undeformed Conti construction (c.f. Lemma \ref{cor:var_conti})
  are axis-parallel and have aspect ratio
  approximately $1:\delta_{j}|_{D_j}$ (with a factor depending on $\lambda$; for
  $\lambda=\frac{1}{4}$ a factor in the interval $(1/4,4)$ is more than
  sufficient). 
  After rescaling the $x_2$-axis by $\delta_{j}|_{D_j}$ (as in Lemma \ref{lem:conti_deformed}), these aspect ratios are then
  comparable to $1:\frac{\delta_{j}|_{D_j}}{2}$, i.e. 
  $1: \delta_{j+1}|_{R\setminus F}$, where $F$ denotes the union of the non-axis-parallel level sets in the deformed configuration. Hence, 
  all the axis-parallel triangles are of the type (P1) or (R1). The triangles, in which $M_4$ is attained, are of type (P1), as the rotation angle of the next Conti construction in the parallel case is controlled by virtue of Lemma \ref{lem:angles}
  \\
  
  It remains to discuss the remaining triangles contained in $F$ (green in Figure
  \ref{fig:good_triangles_conti}). These are again $\delta_{j}|_{D_j}$-good by a
  similar estimate on the aspect ratios, and by an estimate on the angle 
  of rotation with respect to the $x_1$-axis. As by our convex integration Algorithm \ref{alg:construction}, Step 2 (b), $\delta_{j+1}|_{F}=\delta_0$, they are in general of the type (P2) or (R2), if $\delta_j|_{D_j} \neq \delta_0$, but could also be of the
  type (P1) or (R1), if $\delta_j|_{D_j} = \delta_0$.
\end{proof}

\begin{rmk}
\label{rmk:initial}
We observe that Lemma \ref{lem:cov_Cont} in particular implies that the triangles, which are obtained in Step 1 in Algorithm \ref{alg:construction}, all satisfy the claim of Proposition \ref{thm:covering}. Hence, in the sequel it suffices to provide a covering algorithm, which preserves this property.
\end{rmk}

\subsection{A basic building block}
\label{sec:build}

We begin our iterative covering statements by presenting a general building block, which we will frequently use in the sequel. Given a triangle $D$ we seek to reduce the discussion to that of a rectangle $R_2$, whose long side is aligned with the direction of $D$ and which is of similar volume as the original triangle. Only in the covering of this rectangle will the situations (P1), (P2) and (R1), (R2) differ. For the case (R3) we argue differently.

\begin{prop}
\label{prop:cov_P1}
  Let $D$ be a $\delta$-good triangle with $0< \delta \leq \delta_0$.
  Let further $\tilde{R}_2$ be a rectangle of aspect ratio $r\in [\frac{1}{10},1000] \delta$, whose long side is aligned to the direction of $D$. Then there exists
  a rescaled and translated copy $R_2\subset D$ of the rectangle $\tilde{R}_2 $ (of aspect ratio $1:r$), for which three of its corners lie on $\partial D$ and such that:
  \begin{enumerate}
  \item[(i)] $|R_2| \geq 10^{-6}|D|$.
  \item[(ii)] One corner divides a side of the triangle in the ratio $\frac{2}{3}:\frac{1}{3}$.
  \item[(iii)] The set $D \setminus R_2$ consists of at most $100$
    $\delta$-good triangles, which are aligned with the direction of $D$. 
  \end{enumerate}
\end{prop}

\begin{figure}[t]
  \centering
  \includegraphics[width=0.9\linewidth, page=23]{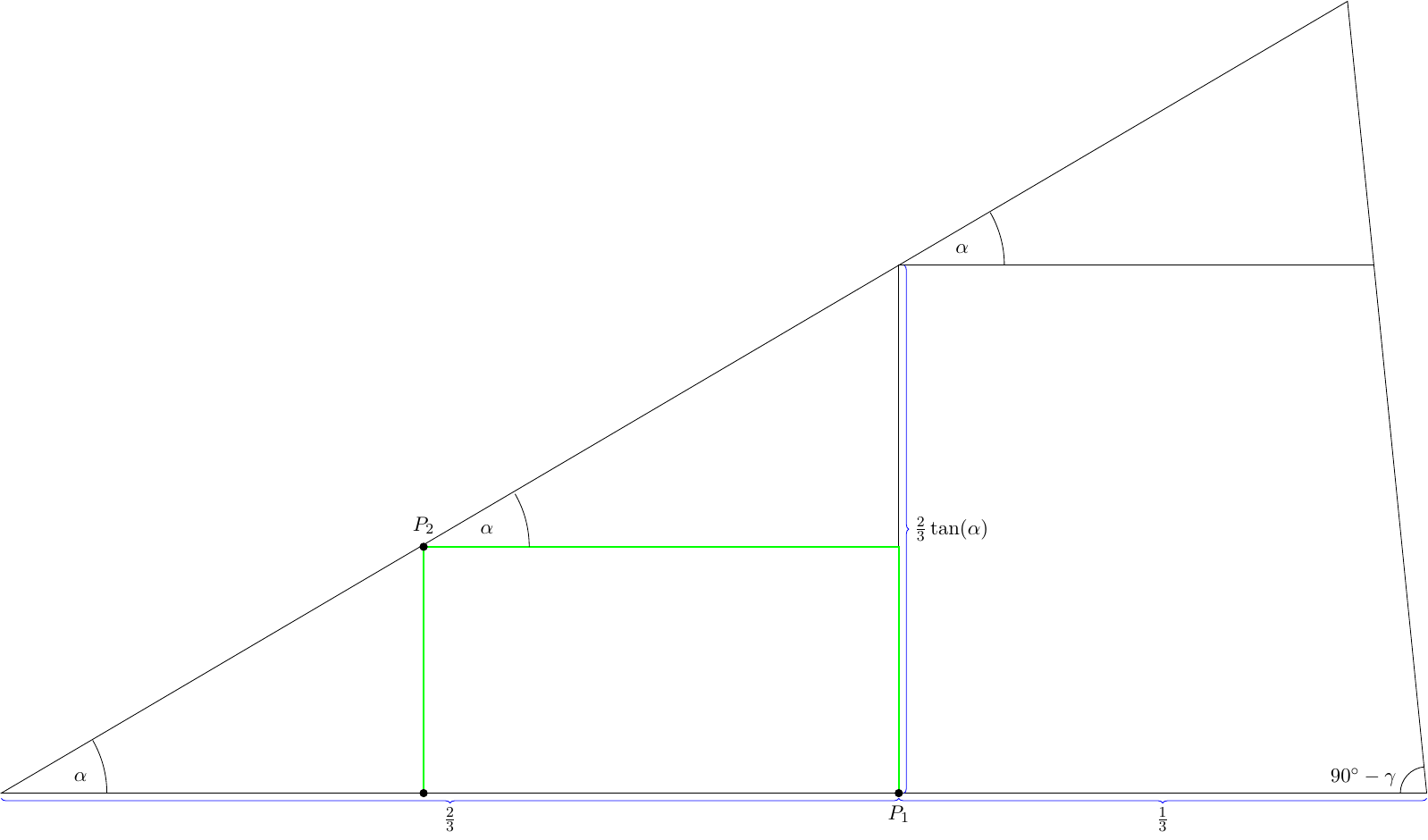}
  \caption{Fitting the parallel rectangle $R_2$ (green box) into $D$. The
    resulting remaining triangles are by construction again
    $\delta$-good. The partition of the box is shown in Figure \ref{fig:heuristic_right_block}.}
  \label{fig:heuristic_triangle}
\end{figure}

\begin{figure}[t]
  \centering
    \includegraphics[width=0.6\linewidth, page=18]{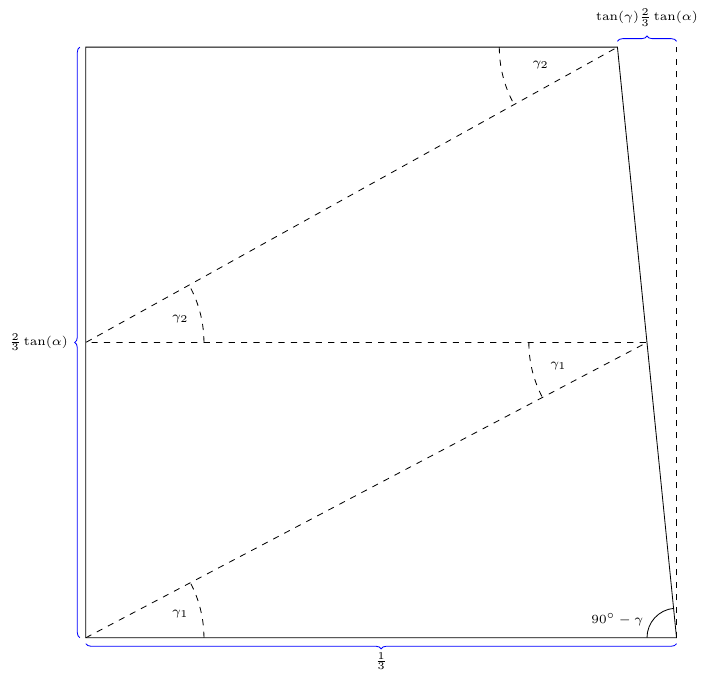}
  \caption{ As the angle at the bottom right is very close to $\frac{\pi}{2}$, the
    box is well approximated by a rectangle of side-lengths
    $\frac{1}{3}:\frac{2}{3}\tan(\alpha)$.
    Partitioning the box into $N$ slices of the same height, an estimate on the
    $\tan$ of the opening angles $\gamma_j$ of the corresponding rectangles shows
    that $\tan(\gamma_j) \approx \frac{2}{N} \tan(\alpha)$.  
   Choosing $N \in \{1,2,3\}$ appropriately, we thus
    obtain $\delta$-good triangles.}
  \label{fig:heuristic_right_block}
\end{figure}

\begin{proof}
Let $D$ be a given $\delta$-good triangle and let $\alpha$ denote the
corresponding angle from Definition \ref{defi:family} (1). Without loss of generality we may assume that the triangle
$D$ is aligned  with the $x_1$-axis, that the tip of the
triangle lies at the origin and that (after rescaling) the $x_1$-axis-parallel side is given
by the interval $[0,1]\times \{0\}$ (c.f. Figure \ref{fig:heuristic_triangle}).

Let $P_1=(\frac{2}{3},0)$ and let $g$ be the line of slope $-r$ through $P_1$.
Then $g$ intersects $\p D$ in exactly one other point $P_2$.
Being aligned along the $x_1$-axis, the rectangle $R_2$ is then uniquely determined by requiring that
$P_1$ and $P_2$ are two of its corners. By construction it has aspect ratio $1:r$.\\
In order to infer the bound on the volume, we compute the coordinates of $P_2=(\frac{2r}{\tan(\alpha) + r}, \frac{2 r \tan(\alpha)}{3(\tan(\alpha) + r)})$.
Hence the volume of $R_2$ is given by
\begin{align*}
|R_2|=  \left( \frac{2}{3}- \frac{2r}{\tan(\alpha) + r}\right)  \frac{2 r \tan(\alpha)}{3(\tan(\alpha) + r)} = \frac{2}{3} \frac{\tan(\alpha)}{r+\tan(\alpha)} \frac{2}{3} \frac{r\tan(\alpha)}{r+\tan(\alpha)}.
\end{align*}
Since the volume of $D$ is comparable to $\frac{\tan(\alpha)}{2}$, this
  results in a volume fraction of approximately
  \begin{align*}
  \frac{|R_2|}{|D|} \geq & \frac{ \frac{2}{3} \frac{\tan(\alpha)}{r+\tan(\alpha)} \frac{2}{3} \frac{r\tan(\alpha)}{r+\tan(\alpha)}}{\frac{\tan(\alpha)}{2}} 
   = \frac{1}{9} \frac{r \tan(\alpha)}{(\frac{r+\tan(\alpha)}{2})^2} \\ &\geq \frac{1}{9} \frac{\min(r,\tan(\alpha))}{\max(r, \tan(\alpha))}.
  \end{align*}
  Using the fact that $r\geq \frac{\delta}{10}$ and $\alpha,\arctan(r) \leq
  1000\delta$, we infer the desired estimate on the volume
  fraction. 
In particular, we note that it is independent of $\delta$.

Adding a vertical line through $P_1$ and a horizontal line through
$P_1+(0,\frac{2}{3}\tan(\alpha)) \in \p D$, we obtain an axis parallel triangle
of opening angle $\alpha$ to the left of $R_2$, another axis parallel triangle
of opening angle $\alpha$ above $R_2$, a four-sided box $B$ to the right of $R_2$
and triangle self-similar to $D$ above the box (c.f. Figure
\ref{fig:heuristic_triangle}).

As $D$ is $\delta$-good, so are the above mentioned three triangles and it
hence remains to discuss the box $B$ on the
right of $R_2$ (c.f. Figure \ref{fig:heuristic_right_block}).
By construction the bottom side of $B$ is axis parallel and of length
$\frac{1}{3}$ and the left side is also axis-parallel and of length
$\frac{2}{3}\tan(\alpha)$.
Furthermore, since $D$ is $\delta$-good, the angle on the bottom right of $B$ is
given by $\frac{\pi}{2} - \gamma$ for some $ \gamma  \in \delta [-2000,2000]$ and in particular $|\gamma| \leq \frac{1}{10}$.
Hence, the length of the axis-parallel top side differs from $\frac{1}{3}$ by
$|\tan(\gamma) \frac{2}{3}\tan(\alpha)| \leq \frac{1}{10}$.

Introducing further horizontal lines, we may partition $B$ into $N$ boxes with
three axis-parallel sides of height $\frac{2}{3N}\tan(\alpha)$ and length close
to $\frac{1}{3}$ (c.f. Figure \ref{fig:heuristic_right_block}).
Bisecting along the diagonals, we hence obtain opening angles $\gamma_j$ with
\begin{align*}
  \tan(\gamma_j)\approx \frac{\frac{2}{3N}\tan(\alpha)}{\frac{1}{3}} = \frac{2}{N} \tan(\alpha).
\end{align*}
Choosing $N \in \{1,2,3\}$ appropriately and noting that the remaining angle is
either $\frac{\pi}{2} -\gamma$ (same as $D$) or a right-angle, all obtained triangles
are $\delta$-good with respect to $n$.
\end{proof}

\begin{rmk}
  We note that the above quotient
  \begin{align*}
    \frac{\min(r,\tan(\alpha))}{\max(r, \tan(\alpha))}
  \end{align*}
  is symmetric in $\tan(\alpha),r$ and punishes
  them being of different size. A similar mechanism can be observed when trying
  to fit axis parallel rectangles of different aspect ratios $r_1,r_2$ inside each other.
  Letting $1:r_1$ be the lengths of the exterior rectangle, the interior
  rectangle has lengths $a:ar_2$, where $a \leq 1$ limits the volume ratio to
  \begin{align*}
    \frac{a^2 r_2}{r_1} \leq \frac{r_2}{r_1},
  \end{align*}
  and $ar_2 \leq r_1 \Leftrightarrow a \leq \frac{r_1}{r_2}$
  limits the volume ratio to
  \begin{align*}
    \frac{a^2 r_2}{r_1} \leq \frac{\frac{r_{1}^2}{r_{2}^2} r_2 }{r_1}= \frac{r_1}{r_2}.
  \end{align*}
This illustrates that the triangular situation is comparable to the rectangular situation, which we introduced as our heuristic model situation in the beginning of Section \ref{sec:covering}.
\end{rmk}

We further explain how, given a box $R$ with some rotation angle with respect to the $x_1$-axis, we construct a block of the type $\tilde{R}_2$.

\begin{lem}
\label{lem:cov_P1}
Let $0<\delta \leq \delta_0$ and let $R$ be a rectangle of aspect ratio
$1:r_0$ for $r_0 \in \delta [\frac{1}{10},10]$. 
Suppose further that the direction $n \in S^1$ of the long side of $R$ encloses
an angle $\beta \in
\delta[-1000,1000]$ with the $x_1$-axis.
Then there exist an axis-parallel rectangle $\tilde{R}_2$, which is parallel to
the $x_1$-axis and of aspect ratio $1:r$, where $r \in \delta[\frac{1}{10}, 1000]$, and a translated and rescaled copy $\tilde{R}$ of $R$ such that
  \begin{itemize}
  \item[(i)] $\tilde{R} \subset \tilde{R}_{2}$ and 
  $|\tilde{R}| \geq 10^{-6}|\tilde{R}_{2}|$, 
  \item[(ii)] The set $\tilde{R}_{2} \setminus \tilde{R}$ can be
    decomposed into at most 100 $\delta$-good triangles, whose direction is either $n$ or
    $e_1$.
  \end{itemize}
\end{lem}

\begin{figure}[t]
  \centering
  \includegraphics[height=5cm, page=16]{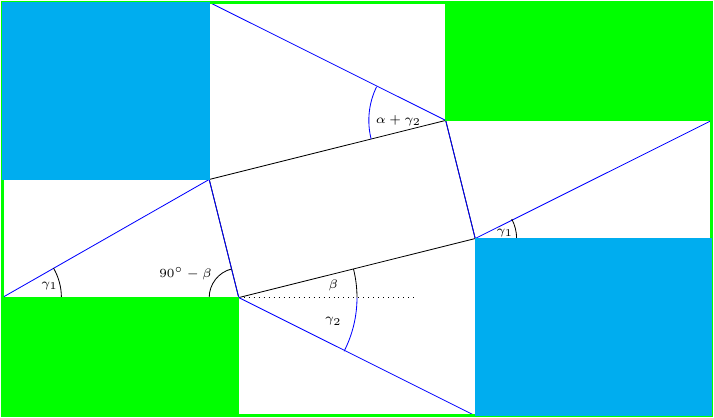}
  \caption{Constructing an axis parallel rectangle starting from $R$.
We begin with the inner rectangle $R=\tilde{R}$ (rotated with an angle $\beta$ with respect to the $x_1$-axis) and successively add the eight outer rectangles to obtain the box $\tilde{R}_2$ from Lemma \ref{lem:cov_P1}. The explicit construction of the four white boxes, which together with $R$ form the inner ``cross" are described in detail in Figures \ref{fig:heuristic_sides} and
   \ref{fig:heuristic_top}. All the outer rectangles are of aspect ratio approximately
   $1:\delta_{j}$ and can hence be decomposed into $\delta_{j}$-good triangles.
  }
  \label{fig:heuristic_box}
\end{figure}

\begin{figure}[t]
    \centering
    \includegraphics[width=0.8\linewidth,page=17]{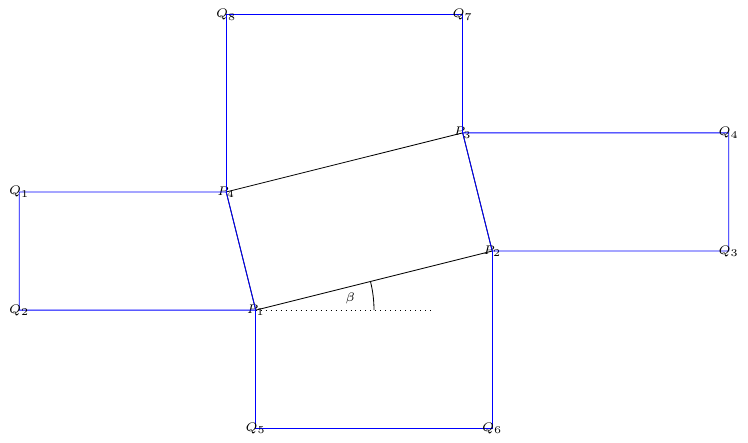}
    \caption{Labeling of points. The figure illustrates the successive addition of further points, which result in the inner cross structure. In this construction we have to choose the points in a way, which ensures that the triangles, which are formed by bisecting the boxes along the diagonals, still remain of the types (P1), (P2) and (R1), (R2).}
    \label{fig:heuristic_box_points}
  \end{figure}

\begin{figure}[t]
  \centering
  \includegraphics[width=0.8\linewidth, page=19]{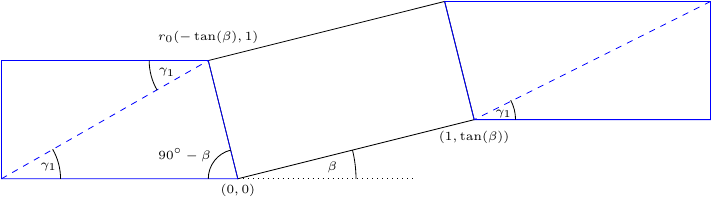}
  \caption{Adding $\delta$-good triangles on the left and right, we can
    achieve axis-parallel boundaries.}
  \label{fig:heuristic_sides}
\end{figure}

\begin{figure}[t]
  \centering
  \includegraphics[height=5cm, page=21]{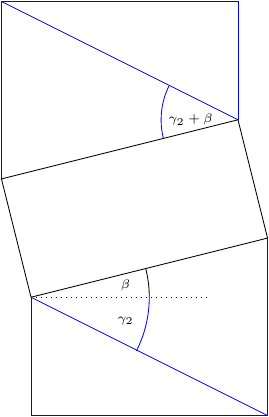}
  \caption{Adding $\delta_{j}$-good triangles on the top and bottom, we can
    achieve axis-parallel boundaries. Since $\beta$ might include a large or
    small factor in the definition, we may either allow and opening angle
    $\gamma_2+\beta$ or introduce a horizontal line to obtain two triangles with
    opening angle $\gamma_2$ and $\beta$.}
  \label{fig:heuristic_top}
\end{figure}

\begin{proof}
  The construction is sketched in Figures \ref{fig:heuristic_box}, \ref{fig:heuristic_box_points}, \ref{fig:heuristic_sides} and \ref{fig:heuristic_top}.

  After a translation, rescaling and reflection with respect to the $x_2$-axis, we may assume
  that $\beta \geq 0$ and that the corners of $\tilde{R}$ are given by
  \begin{align*}
    P_1&=(0,0),\\ P_2&=(1,\tan(\beta)), \\ P_3&=(1,\tan(\beta))+r_0 \sqrt{1+\tan(\beta)^2}(-\sin(\beta),\cos(\beta))\\
    &= (1,\tan(\beta)) + r_{0} (-\tan(\beta),1), \\ P_4&=r_0(-\tan(\beta),1),
  \end{align*}
  where we used that
  \begin{align*}
    \cos(\beta) \sqrt{1+\tan(\beta)^2}&= 1, \\
    \sin(\beta) \sqrt{1+\tan(\beta)^2}&= \tan(\beta).
  \end{align*}

  In the following, we add quadrilaterals with three axis-parallel sides
  with aspect ratios $r_{0}$ and $ r_{2}=\tan(\gamma_{2})$ for 
  \begin{align}
  \label{eq:gamma_2}
    \gamma_2 \in \delta (\frac{3}{10}, \frac{1000}{3}),
  \end{align}
   to be chosen later.
  
  We begin by adding quadrilaterals on the left and right by inserting the following four points (c.f. Figure \ref{fig:heuristic_box_points})
  \begin{align*}
    Q_1&= P_4 - (1,0), \\
    Q_2&= Q_1 - (0, r_0)= P_1 - (1 + r_0\tan(\beta),0), \\
    Q_3&= P_2 + (1,0),\\
    Q_4&= Q_3 + (0, r_0)= P_3 + (1 + r_0 \tan(\beta),0).
  \end{align*}
 
  We, in particular, note that the lines $\overline{Q_1 Q_2}$ and
  $\overline{Q_3Q_4}$ are parallel to the $x_2$-axis.
  Furthermore, the triangles $Q_2P_1P_4$ and $P_4Q_1Q_2$ have opening angles
  $\arctan(r_0)$, are parallel to the $x_1$-axis and are either right-angled or have an
  angle $\frac{\pi}{2} - \beta$. Hence all of these triangles are $\delta$-good with direction $e_1$.
  Similar observations hold for the triangles, which are constructed from $P_2,P_3,Q_2,Q_3$.\\

  Following a similar approach, we add the points
  \begin{align*}
    Q_5&= P_1 - (0,r_{2}), \\
    Q_6&= Q_5 + (1, 0)= P_2 - (0, r_{2} + \tan(\beta)), \\
    Q_7&= P_2 + (0,r_{2}) ,\\
    Q_8&= Q_7 - (1,0)= P_3 + (0, r_{2}+ \tan(\beta)).
  \end{align*}
  Here, the aspect ratio $r_2$ is chosen flexibly to account for the facts
    that our construction is horizontally of a length, which is slightly larger than $3$, and that the rotated
    rectangle has height $r_0+\tan(\beta)\geq r_0$.
  By symmetry, we may restrict ourselves to discussing the rectangle $P_1Q_5Q_6P_2$.
  The axis-parallel right angled triangle $P_{1}Q_{5}Q_{6}$ has opening angle
  $\gamma_{2}$ (as defined in (\ref{eq:gamma_2})) and is thus $\delta$-good.
  For the remaining triangle $P_{1}Q_{6}P_{2}$, we distinguish two cases:
  \begin{itemize}
  \item
  If $\beta \in \delta [\frac{1}{10}, 1000]$, we additionally
  introduce the point $Q_{9}= (1,0)$ and note that $P_1Q_{9}P_2$ is $\delta$-good
  and axis-parallel, as are $P_1Q_{9}Q_6$ and $Q_6Q_5P_1$ (which both have an opening angle $\gamma_2$).
 
\item  If $0\leq \beta \leq \delta \frac{1}{10}$, we note that by our
  restriction on $\gamma_2$,
  \begin{align*}
     \beta+ \gamma_{2} \in \delta [\frac{1}{10}, 300],
  \end{align*}
  which ensures that $P_{1}Q_{6}P_{2}$ is $\delta$-good and parallel 
  to the long side $n$ of $\tilde{R}$.
  \end{itemize}

  Finally, we complete our thus far roughly cross-shaped construction to the
  desired axis-parallel rectangle $R_2$ by adding four rectangles as in Figure
  \ref{fig:heuristic_box} (the green rectangles there).
  These rectangles have side lengths
  \begin{align}
  \label{eq:r1}
    (1+ r_0\tan(\beta)) &:  r_{2}, \\
    \label{eq:r2}
    1 &: (r_{2}+\tan(\beta)).
  \end{align}
  We consider the first rectangle with side ratio as in (\ref{eq:r1}).
  Since $r_0 \tan(\beta) \leq 2 (1000\delta)^2 < 0.1$ (which follows from the bounds for $\delta_0$), we can estimate the aspect ratio from
  above and below by $1:\frac{r_{2}}{1.1}$ and $1:r_{2}$, respectively.
  Bisecting this rectangle along the diagonals, then results in $\delta$-good axis-parallel right triangles, provided
  \begin{align*}
    r_{2} \in [1.1 \arctan(\delta\frac{1}{10}), \arctan(1000 \delta )].
  \end{align*}
  This is satisfied due the assumptions on $\gamma_2$, since $\frac{x}{1.1}
    \leq \arctan(x) \leq x$ on the considered domain.
  
  For the second rectangle with ratio as in (\ref{eq:r2}), we again distinguish two cases:
  \begin{itemize}
  \item If $\beta \in \delta[\frac{1}{10}, 1000]$, we divide the rectangle by a
    horizontal line through $P_{1}$, which yields two rectangles of lengths
    $1:\tan(\beta)$, $1:r_{2}$, which are $\delta$-good.
  \item If $\beta \in \delta [0,\frac{1}{10}]$, we note that by the same
    argument as above
    \begin{align*}
      r_{2} + \tan(\beta) \in [\arctan(\delta\frac{1}{10}), \arctan(1000\delta)].
    \end{align*}
    Thus the aspect ratio $1: (r_{2}+\tan(\beta))$ results in
    $\delta$-good axis-parallel triangles.
  \end{itemize}

  We conclude by noting that the resulting axis-parallel rectangle $R_{2}$ (which is the entire rectangle in Figure \ref{fig:heuristic_box_points}) has
  side lengths
  \begin{align*}
    (1+   2 +r_0\tan(\beta))  : (r_{0}+ 2 r_{2}+ \tan(\beta)).
  \end{align*}  
  Again estimating $r_{0}\tan(\beta)< 0.1$, this yields suitable triangles,
  provided
  \begin{align*}
    r_{0}+ 2 r_{2}+ \tan(\beta) \in [3.1 \arctan(\delta\frac{1}{10}), 3\arctan(1000 \delta)].
  \end{align*}
 Again this is satisfied due to our restrictions on $\gamma_2, r_0$ and $\beta$.

 We thus obtain a large family of admissible values $r_{2}$.
We note that the aspect
  ratio of $R_2$ is comparable to $2r_2+ \tan(\beta)+r_0$, as is the area of $R_2$.
 Hence, as a particular choice, we may take $r_2$ comparable to
  $\frac{r_0+\tan(\beta)}{2}$ (within a factor $3$ to ensure that $\gamma_2$
  satisfies the above restriction). Then the aspect ratio is comparable to
  \begin{align*}
    1: \frac{r_0+\tan(\beta)}{2},
  \end{align*}
  and the volume ratio is comparable to
  \begin{align*}
  \frac{|\tilde{R}|}{|\tilde{R}_2|}\geq
    \frac{r_0}{3 \cdot 2(r_0+\tan(\beta))}= \frac{1}{6} \frac{r_0}{r_0+\tan(\beta)}.
  \end{align*}
  Since $r_0 \geq \arctan(\delta \frac{1}{10})$ and $\beta \leq 1000 \delta$,
  this quotient may be estimated from below by $\frac{1}{10000}$.
\end{proof}

\begin{rmk}
\label{rmk:interp}
One should think of the rectangle $R$ in Lemma \ref{lem:cov_P1} as a Conti construction, which we seek to fit into a triangle $D$ as in Proposition \ref{prop:cov_P1}. In doing so, we however have to be careful, since in the cases (P1), (P2) we have to avoid creating new triangles, which are substantially rotated with respect to the original one (c.f. Figure \ref{fig:bad_triangle} and the explanations at the beginning of the next section). The box construction of Lemma \ref{lem:cov_P1} ensures this.
\end{rmk}

\subsection{Covering in the cases (P1), (P2)}
\label{sec:parallel}

In this section we explain how, given a triangle $D_j \in \{\Omega_{j,k}\}_{k\in\{1,\dots,J_j\}}$, which is of type (P1) or (P2), we can cover it by a combination of the relevant Conti constructions and some remaining triangles, which are again of the types (P1), (P2) and (R1), (R2). Moreover, we seek to achieve two partially competing objectives: On the one hand, we have to
control the volume of $D_j$, which is covered by Conti constructions, from below. On the other hand, we aim at keeping the resulting overall perimeter of the new covering geometry as small as possible. The construction of a covering, which balances these two objectives, is the content of Proposition \ref{prop:parallel}, which is the main result of this section.\\

Motivated by the heuristic considerations at the beginning of Section \ref{sec:covering} (c.f. Figure \ref{fig:rectangle_coverings} (a)), we expect that in the cases (P1) and (P2), in which there is no substantial rotation with respect to the
relevant Conti construction, the two competing objectives of sufficient volume coverage (Proposition \ref{prop:parallel} (1)) and of a good perimeter bound (Proposition \ref{prop:parallel} (3)), can be satisfied with a surface energy, which
is independent of $\delta_j$ and $\delta_0$. Indeed, it is possible to show that in the situation without substantial rotation, in each iteration step the overall perimeter of the covering of a triangle is comparable to the perimeter of the original triangle up to a loss of a controlled universal factor.

\begin{prop}
\label{prop:parallel}
Let $D_j$ be as in (P1)-(P2) with $j\geq 1$. 
Then there exists a covering and a constant $C>1$
   (independent of $\delta_0$) such that:
  \begin{enumerate}
  \item A volume fraction of at least $10^{-12}|D_j|$ is covered by finitely many
   rescaled and translated Conti constructions from Lemma
   \ref{lem:conti_deformed}. The Conti constructions can
   again be covered by finitely many 
   triangles of the types occurring in the cases (P1)-(P2) and
   (R1)-(R2), where $j$ is replaced by $j+1$.
  \item The complement of the Conti constructions is covered by finitely many 
  	triangles occurring in the cases (P1)-(P2), where $j$ is replaced by $j+1$.
   \item The overall surface energy of the new triangles 
   $D_{j+1,l}\in \mathcal{D}_1(D_k)$, is controlled by
   \begin{align*}
   \sum\limits_{D_{j+1,l}\in \mathcal{D}_1(D_j)}\Per(D_{j+1,l}) 
   \leq C \Per(D_j).
   \end{align*}
  \end{enumerate}
\end{prop}

In the proof of Proposition \ref{prop:parallel} we have to be careful in the choice of the covering, in order
to keep all the resulting triangles parallel to the direction of $D_j$ or parallel to the
  relevant Conti construction (c.f. Definition \ref{defi:Conti_direction}). This is necessary to ensure a covering such that the sum of the resulting perimeters is comparable to the original perimeter; in particular no factor of $\delta_0$ occurs here.
We emphasize that this alignment with the directions of the original triangle or the relevant Conti construction is a central point, since
if a (substantial) rotation angle with respect to these directions were to be obtained (e.g. as illustrated in Figure
\ref{fig:bad_triangle}, where the covering gives rise to triangles which are rotated by an angle of $\frac{\pi}{2}$), we would inevitably fall into cases similar as the situations described in (R1), (R2), however with a ratio $\delta_j$, which might be substantially smaller than $\delta_0$. As explained at the beginning of Section \ref{sec:covering}, this would entail a growth of the perimeters of the covering by a factor $\delta_j$. As a consequence our BV estimate from Section \ref{sec:quant} would become a superexponential bound, which could no longer be compensated by the only exponential $L^1$ decay. This would hence destroy all hopes of deducing good higher regularity estimates for the convex integration solutions.\\

\begin{figure}[t]
  \centering
  \includegraphics[width=0.4\linewidth, page=2]{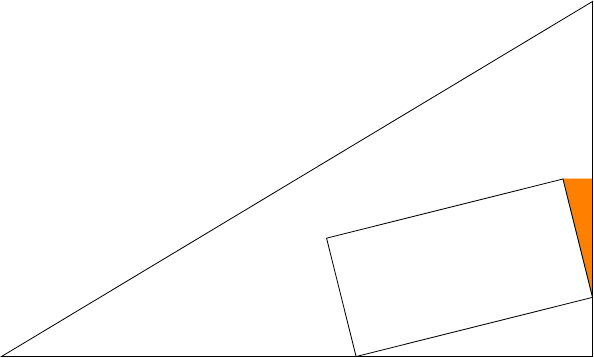}
  \caption{Problems which could arise in the covering algorithm: In our parallel covering result we have to avoid rotated triangles as the aspect ratios are very small for these.}
  \label{fig:bad_triangle}
\end{figure}

The remainder of this section is organized into three parts: We first discuss the covering constructions for the cases (P1) and (P2) separately in Sections \ref{sec:P1},
\ref{sec:P2}. Then in Section \ref{sec:P} we combine these cases, in order to provide the proof of Proposition \ref{prop:parallel}.

\subsubsection{The case (P1)}
\label{sec:P1}

We begin by explaining the covering in the case (P1). 

\begin{lem}
\label{prop:p11}
Let $D$ be a $\delta_j$-good triangle oriented along the $x_1$-axis. Let $R$ be a rectangle of aspect ratio $1:\delta_{j}/2$ such that its long axis is rotated by an angle of $\beta \in \delta_j[-10,10]$ with respect to the $x_1$-axis. Then there exists a covering by 
\begin{itemize}
\item[(i)] $K_j$, with $K_j\in [1,100]$, $\delta_j$-good, up to null-sets disjoint triangles, $D_l$, $l\in\{1,\dots,K_j\}$, which are either oriented along the $x_1$-axis or the long side of $R$,
\item[(ii)] a rescaled and translated copy $\tilde{R}$ of $R$, such that
$|\tilde{R}| \geq 10^{-6}|D|$.
\end{itemize}
Moreover,
\begin{align*}
\Per(\tilde{R}) + \sum\limits_{l=1}^{K_j} \Per(D_l) \leq C\Per(D),
\end{align*}
where $C>1$ is a universal constant (in particular independent of $\delta_j$ and $\delta_0$).
\end{lem}

\begin{proof}
We first invoke Lemma \ref{lem:cov_P1} with $R$ and $\delta= \delta_j$. This yields an axis-parallel box $\tilde{R}_2$ of side ratio $r\in [\frac{1}{10},10]\delta_j$. This box $\tilde{R}_2$ is admissible in Proposition \ref{prop:cov_P1}. An application of this proposition with $D$, $\tilde{R}_2$ and $\delta=\delta_j$ hence yields a covering of $D$ by $\delta_j$-good triangles, which all have $e_1$ as their direction, and a box $R_2$, which is covered as described in Lemma \ref{lem:cov_P1}. We note that the triangles within $R_2$ are thus also $\delta_j$-good and have as their directions either $e_1$ or the long side of $R$. The estimate on the perimeter follows, since all the covering triangles have perimeter controlled by $\Per(D)$ and as $K_j \leq 100$. The estimate on the volume fraction is a consequence of Proposition \ref{prop:cov_P1} (i) and Lemma \ref{lem:cov_P1} (i).
\end{proof}

\subsubsection{The case (P2)}
\label{sec:P2}

As in the case (P1) we have the following main covering result:

\begin{lem}
\label{prop:cov_P2}
  Let $D$ be a $\delta_{j-1}$-good, axis-parallel triangle. Let $R$ be
  a rectangle of aspect ratio $1: \delta_{0}$, whose long side is rotated with respect to the axis by an angle $\beta \in \delta_0[-1000,1000]$. Then there exists a covering of $D$ into 
\begin{itemize} 
\item[(i)] $M_j$, with $M_j \in [1,100]$, $\delta_{j-1}$-good triangles $D_l$, $l\in\{1,\dots,M_j\}$, which are parallel to the $x_1$-axis,
\item[(ii)] $K_j:= \frac{\delta_{0}}{\delta_{j-1}}$ translated, disjoint and rescaled copies $\tilde{R}_k$ of $R$ with the property that
$$|\bigcup\limits_{k=1}^{K_j}\tilde{R}_k| \geq 10^{-6}|D|,$$
\item[(iii)] $\tilde{M}_j$, with $\tilde{M}_j \in [1,100]$, $\delta_{0}$-good triangles $\tilde{D}_l$, $l\in\{1,\dots,\tilde{M}_j\}$, which are either parallel to the $x_1$-axis or parallel to the long side of $\tilde{R}$.
\end{itemize} 
Moreover,
\begin{align*}
\sum\limits_{k=1}^{K_j}\Per(\tilde{R}_{k}) + \sum\limits_{l=1}^{M_j} \Per(D_l) 
+ \sum\limits_{l=1}^{\tilde{M}_j} \Per(\tilde{D}_l) \leq C\Per(D),
\end{align*}
where $C>1$ is a universal constant (in particular independent of $\delta_j$ and $\delta_0$).
\end{lem}

\begin{proof}
We apply Lemma \ref{lem:cov_P1} with $\delta=\delta_0$ and the box $R$. This yields a box $\tilde{R}_2$ of ratio approximately $\delta_0$ and a box $\tilde{R}\subset \tilde{R}_2$, which is a translated and rescaled copy of $R$ of volume comparable to the volume of $\tilde{R}_2$. Stacking $K_j:= \frac{\delta_0}{\delta_{j-1}}$ translated copies of the boxes $\tilde{R}_2$ along the $x_1$-axis next to each other and denoting the individual boxes by $\tilde{R}_{2,k}$ (each containing a translated copy $\tilde{R}_k$ of $\tilde{R}$), yields as their union a new box $\bar{R}_2$ of aspect ratio $1:\delta_{j-1}$ (c.f. Figure \ref{fig:teilen}). With respect to this rectangle $\bar{R}_2$ and with $\delta=\delta_{j-1}$ we now apply Proposition \ref{prop:cov_P1}, which yields a rectangle $R_2$ of the same aspect ratio as that of $\bar{R}_2$. As the volume of each $\tilde{R}_k$ is comparable to the volume of $\tilde{R}_{2,k}$, the claim (ii) of Lemma \ref{prop:cov_P2} follows from Proposition \ref{prop:cov_P1}, since this 
ensures that $R_2$ has volume comparable to $D$.\\
It remains to bound the perimeters. Here we only estimate the sum of the perimeters of the rectangles $\tilde{R}_{2,k}$, as the remaining parts of the covering are controlled by a multiple of this. We note that
\begin{itemize}
\item each rectangle $\tilde{R}_{2,k}$ has perimeter bounded by
\begin{align*}
\Per(\tilde{R}_{2,k})\leq C \frac{\delta_{j-1}}{\delta_0}\Per(D),
\end{align*}
\item there are $\frac{\delta_0}{\delta_{j-1}}$-many axis parallel boxes $\tilde{R}_{2,k}$.
\end{itemize}
Hence,
\begin{align*}
\sum\limits_{k=1}^{K_j}  \Per(\tilde{R}_{2,k})\leq 
 C \frac{\delta_{j-1}}{\delta_0} \frac{\delta_0}{\delta_{j-1}} \Per(D) \leq C \Per(D).
\end{align*}
This concludes the proof.
\end{proof}

The main difference of Lemma \ref{prop:cov_P2} with respect to Proposition \ref{prop:cov_P1}
is the step, in which we bridge the mismatch in the ratios of the triangle $D$ (ratio $\delta_{j-1}$) and the given box $R$ (ratio $\delta_0$). Here we pass from a box of ratio approximately $\delta_0$ (which is prescribed for $R$ and hence for $\tilde{R}_2$) to a box with ratio approximately $\delta_j$ (for $\bar{R}_2$) by stacking translates of the boxes $\tilde{R}_{2,k}$ next to each other.

\begin{figure}[t]
\centering
\includegraphics[scale=0.6, page=49]{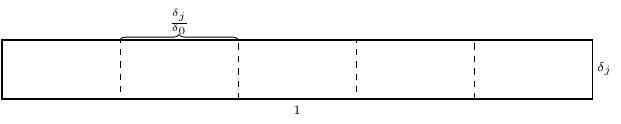}
\caption{The stacking construction of the boxes. Each of the smaller boxes is a suitably translated copy of the rectangle $\tilde{R}_2$, which is roughly of aspect ratio $1:\delta_0$. In each of these we insert a rescaled and translated version of the construction from Lemma \ref{lem:cov_P1}, c.f. Figure \ref{fig:heuristic_box}.}
\label{fig:teilen}
\end{figure}

\subsubsection{Proof of Proposition \ref{prop:parallel}}
\label{sec:P}
Using the results from Sections \ref{sec:P1}, \ref{sec:P2} we can now address the proof of Proposition \ref{prop:parallel}.

\begin{proof}
The first property of the Proposition follows from Lemma \ref{lem:cov_Cont} in combination with Lemma \ref{prop:p11} (ii)  (in the case (P1)) or Lemma \ref{prop:cov_P2} (ii) (in the case (P2)). In particular, by Lemma \ref{lem:cov_Cont} all the triangles, which are used to cover the Conti constructions, are $\delta_{j+1}$-good with respect to the relevant Conti construction.
The second property is a consequence of Lemma \ref{prop:p11} (i) combined with Lemma \ref{lem:cov_P1} (in the case (P1)) or Lemma \ref{prop:cov_P2} (i), (iii) (in the case (P2)). We emphasize that all these triangles are either parallel to the original triangle $D$ or to the relevant Conti construction, implying that both the angles and the orientations are within the admissible margins.
Finally, the bound on the perimeters follows from the corresponding claims in Lemmata \ref{prop:p11} and \ref{prop:cov_P2}.
\end{proof}

\subsection{Covering in the cases (R1)-(R3)}
\label{sec:rota}

In this section we deal with the covering in the cases (R1)-(R3). As in Section \ref{sec:parallel} we seek to simultaneously control the perimeter of the resulting covering and the volume of the domain, which is covered by Conti constructions. Motivated by the discussion from the beginning of Section \ref{sec:covering}, we however expect that it is unavoidable to produce estimates, in which the ratio $\delta_0$ appears.\\

With this expectation, we are less careful in our covering constructions and for instance do not seek to preserve the direction $n$, in which the corresponding $\delta_j$-good triangles are oriented. Yet, we still heavily rely on Proposition \ref{prop:cov_P1} and only modify the construction within the block $R_2$. This will give rise to certain new ``error triangles", which are of the type (R3).
In analogy to Proposition \ref{prop:parallel} we have:

\begin{prop}
\label{prop:rot}
Let $D_j$ be as in (R1)-(R3) with $j\geq 1$. 
Then there exists a covering and a constant $C>0$ independent of $\delta_0$ such that:
  \begin{enumerate}
  \item A volume fraction of $10^{-6}|D_j|$ is covered by finitely many
   rescaled and translated Conti constructions. The Conti-constructions can
   again be covered by finitely many 
   triangles of the types occurring in the cases (P1)-(P2) and
   (R1)-(R3), where $j$ is replaced by $j+1$.
  \item The complement of the Conti constructions is covered by finitely many 
  	triangles occurring in the cases (P1),(P2) and (R1)-(R3), where $j$ is replaced by $j+1$.
   \item The overall surface energy of the new triangles 
   $D_{j+1,l}\in \mathcal{D}_1(D_j)$, is controlled by
   \begin{align*}
   \sum\limits_{D_{j+1,l}\in \mathcal{D}_1(D_j)}\Per(D_{j+1,l}) 
   \leq C \delta_0^{-1} \Per(D_j).
   \end{align*}
  \end{enumerate}
\end{prop}
 
As in Proposition \ref{prop:parallel} the proof of this statement is based on separate discussions of the cases (R1), (R2), (R3) and can be deduced by combining the results of Lemmas \ref{lem:cov_R1}, \ref{lem:cov_R2}, \ref{lem:cov_R3}, \ref{lem:cov_Cont} and Proposition \ref{prop:cov_P1}.
Since this does not involve new ingredients, we restrict our attention to the discussion
of the cases (R1)-(R3) and omit the details of the proof of Proposition \ref{prop:parallel}. The analysis of the cases (R1)-(R3) is the content of the following subsections.

\subsubsection{The case (R1)}
The covering result for the case (R1) is very similar to the one from the case (P1). It only deviates from this by the construction within the rectangle $R_2$: 

\begin{lem}
\label{lem:cov_R1} 
Let $\beta \in (C \delta_0, \frac{\pi}{2}-C\delta_0)$.
Assume that $D$ is a $\delta_0$-good triangle oriented parallel to the $x_1$-axis.
Let $\bar{R}$ be a rectangle of aspect ratio $1:\delta_0$, which encloses an angle $\beta$ with respect to the orientation of $D$. Then $D$ can be covered by the union of
\begin{itemize}
\item[(i)] $M_j$, with $M_j \in [1,100]$, $\delta_0$-good triangles, $D_{1,k}$, which are aligned with the direction of $D$,
\item[(ii)] $0<K_j \leq C \delta_0^{-2}$ many translated, up to null-sets
  disjoint and rescaled copies $R_{2,k}$ of the rectangle $\bar{R}$, whose
    union covers a volume of size at least $10^{-6}|D|$,
\item[(iii)] $0\leq L_j \leq C \delta_0^{-2}$ many triangles $D_{2,k}$, which are of the type (R3).
\end{itemize}
Here $C>1$ is a universal constant. 
The overall perimeter of the resulting triangles and rectangles is controlled by
\begin{align}
\label{eq:Per}
 \sum\limits_{k=1}^{M_j}\Per(D_{1,k})+\sum\limits_{k=1}^{K_j}\Per(R_{2,k}) + \sum\limits_{k=1}^{L_j}\Per(D_{2,k}) \leq C \delta_0^{-1} \Per(D).
\end{align}
\end{lem}

\begin{figure}[t]
\centering
\includegraphics[width=0.8 \textwidth, page=50]{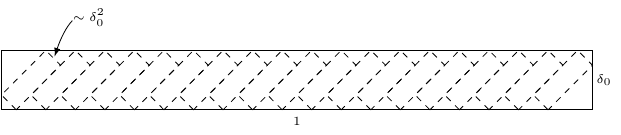}
\caption{The covering of the box $R_2$ of ratio $1:\delta_0$. The dashed rectangles correspond to the $K_j$ stacked (rescaled and translated) copies of $\bar{R}$, which we denote by $R_{2,k}$. Their envelope is a (rescaled) copy of $\bar{R}_2$, which we denote by $R_2$. It is the rectangle, which is returned as the output of Proposition \ref{prop:cov_P1}. The parts of $R_2$, which are not covered by the rectangles $R_{2,k}$, consist of the triangles $D_{2,k}$.}
\label{fig:teilen_R}
\end{figure}

\begin{rmk}
By Lemma \ref{lem:angles} the angles, which occur in our constructions, always satisfy the bound $\beta \in (C \delta_0, \frac{\pi}{2}-C\delta_0)$.
\end{rmk}

\begin{proof}
Let $c_1\in(1/4,1)$, $c_2 \in(1,4)$.
We begin by stacking $K_j \in [c_1, c_2]\delta_0^{-2}\cap \N$ many rectangles
$\tilde{R}_{2,k}$, which are translated copies of $\bar{R}$, next to each other
in such a way that their lowest corners lie on the $x_1$-axis (c.f. Figure
\ref{fig:teilen_R}). Let $\tilde{R}_2$ denote the enveloping axis-parallel
rectangle. 
By adapting the constants $c_1,c_2$ we can arrange that $\tilde{R}_2$
has an aspect ratio $r$ allowing for an application of Proposition \ref{prop:cov_P1} with $\delta = \delta_0$, $\tilde{R}_2$ and $r$. This yields a rectangle $R_2$ of aspect ratio $r$. Thus, the set $D\setminus R_2$ consists of the triangles described in (i). Moreover $R_2$ is covered as in Figure \ref{fig:teilen_R} by $K_j$-many rectangles $R_{2,k}$ with  aspect ratio $\delta_0$
and by a comparable number of ``error" triangles $D_{2,k}$. By definition, the constant $K_j$ satisfies the bounds in (ii). Using elementary geometry, we calculate that 
\begin{align}
\label{eq:vol}
| \bigcup\limits_{k=1}^{K_j}R_{2,k}|
 \geq  \frac{1}{10}|R_2|.
\end{align}
This implies the claim of (ii). \\
We note that the error triangles $D_{2,k}$ are all right angle triangles. Moreover, one of
the other angles coincides with the rotation angle $\beta$. At least one of triangles' sides is parallel to the orientation of the rectangles $R_{2,k}$. Thus the triangles $D_{2,k}$ are of the type (R3).\\
The  bound on the perimeters of the rectangles $R_{2,k}$ and of the triangles $D_{2,k}$ results from the following observations:
\begin{itemize}
\item The number $K_j$ of rectangles $R_{2,k}$ and the number $L_j$ of triangles $D_{2,k}$ are bounded by $C \delta_0^{-2}$.
\item The perimeter of each of the rectangles and each of the triangles is controlled: $\Per(R_{2,k}) + \Per(D_{2,k}) \leq C \delta_0 \Per(D)$.
\item There are at most 100 triangles $D_{1,k}$, each of which has a perimeter controlled by $\Per(D)$.
\end{itemize}
Thus,
\begin{align*}
\sum\limits_{k=1}^{M_j}\Per(D_{1,k}) +
\sum\limits_{k=1}^{K_j}\Per(R_{2,k}) + \sum\limits_{k=1}^{L_j}\Per(D_{2,k}) 
&\leq C \delta_0\Per(D) \delta_0^{-2}\\
& \leq C \delta_0^{-1} \Per(D).
\end{align*}
This concludes the proof.
\end{proof}

\subsubsection{The case (R2)}

The case (R2) is the rotated analogue of the case (P2). As we are in a
  rotated case, we have to be less careful about preserving orientations, and
  proceed similarly as in the case (R1).
Again the main issue is the covering of the rectangle $R_2$. However, in contrast to the case (R1) we now have to deal with a mismatch between the ratio of the triangle $D$ (with ratio $\delta_{j} \neq \delta_0$) and the ratio of the Conti construction (with ratio $\delta_0$). Similarly as in the case (P2) we overcome this issue by a ``stacking construction", which compensates the mismatch.

\begin{lem}
\label{lem:cov_R2} 
Let $\beta \in (C \delta_0, \frac{\pi}{2}-C \delta_0)$.
Assume that $D$ is a $\delta_{j}$-good triangle with direction parallel to the $x_1$-axis. Let $\bar{R}$ be a rectangle of side ratio $\delta_0$, which encloses an angle $\beta$ with respect to the long side of $D$. Then $D$ can be covered by the union of
\begin{itemize}
\item[(i)] $M_j$, with $M_j \in [1,100]$, $\delta_0$-good triangles, $D_{1,k}$, which are aligned with the direction of $D$,
\item[(ii)] $0<K_j\leq C \delta_0^{-1}\delta_j^{-1}$ many translated and
  rescaled copies $R_{2,k}$ of the rectangle $\bar{R}$, whose union covers
    a volume of size at least $10^{-6}|D|$,
\item[(iii)] $0<L_j\leq C \delta_0^{-1}\delta_j^{-1}$ many triangles $D_{2,k}$, which are of the type (R3).
\end{itemize}
There exists a universal constant $C>1$ such that the perimeter of the resulting triangles and rectangles is bounded by
\begin{align*}
\sum\limits_{k=1}^{M_j}\Per(D_{1,k}) + \sum\limits_{k=1}^{K_j}\Per(R_{2,k}) + \sum\limits_{k=1}^{L_j}\Per(D_{2,k})  \leq C \delta_0^{-1} \Per(D).
\end{align*}
\end{lem}

\begin{figure}[t]
\centering
\includegraphics[width=0.8 \textwidth, page=51]{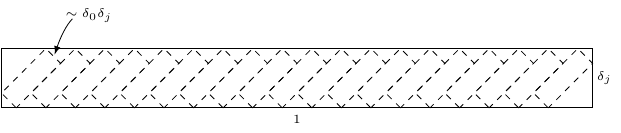}
\caption{The figure shows a (rescaled) copy of the enveloping rectangle $R_2$ and the stacked (and rescaled) copies of the rectangle $\bar{R}$, which we denote by $R_{2,k}$. The triangles correspond to the ones, which we denote by $D_{2,k}$ in Lemma \ref{lem:cov_R2} (iii).}
\label{fig:teilen_R2}
\end{figure}

\begin{proof}
We construct a box $\tilde{R}_{2}$ as in the case (R1) but now by stacking $K_j
\in [c_1,c_2](\delta_0 \delta_j)^{-1}\cap \N$ many of the boxes $\bar{R}$
next to each other, where $c_1\in(1/4,1)$ and $c_2\in (1,4)$. We denote these
stacked boxes by $\tilde{R}_{2,k}$ and define $\tilde{R}_2$ as the enveloping
axis-parallel rectangle. By adapting the values of $c_1, c_2$ it is
possible to obtain a ratio $r$ for $\tilde{R}_2$ (c.f. Figure \ref{fig:teilen_R2}), which is admissible in applying Proposition \ref{prop:cov_P1} with $\delta=\delta_j$, $\tilde{R}_2$ and $r$. This yields a box $R_2$, which is covered by rescaled copies $R_{2,k}$ of the rectangles $\tilde{R}_{2,k}$ and by ``error" triangles $D_{2,k}$. By construction and by elementary geometry (as in (\ref{eq:vol})) these satisfy the requirements in (ii), (iii). By Proposition \ref{prop:cov_P1} also (i) holds true.\\
It remains to estimate the perimeter of the union of the rectangles $R_{2,k}$ and the triangles $D_{2,k}$. To this end, we note that:
\begin{itemize}
\item Each rectangle $R_{2,k}$ has perimeter controlled by $C\delta_{j} \Per(D)$.
\item There are $C \delta_{j}^{-1}\delta_0^{-1} $ many such rectangles $R_{2,k}$.
\item The perimeters of the error triangles $D_{2,k}$ are up to a factor controlled by the perimeters of the rectangles $R_{2,k}$.
\end{itemize}
Thus, the resulting perimeter is up to a constant bounded by
\begin{align*}
 &\sum\limits_{k=1}^{M_j}\Per(D_{1,k}) + \sum\limits_{k=1}^{K_j}\Per(R_{2,k}) + \sum\limits_{k=1}^{L_j}\Per(D_{2,k}) \\
 & \leq C \delta_{j} \Per(D) (\delta_{j}^{-1} \delta_0^{-1}) = C\delta_0^{-1} \Per(D).
\end{align*}
This concludes the argument of the lemma.
\end{proof}

\subsubsection{The case (R3)}
We deal with the error triangles from the previous step. All of them are right
angle triangles, in which the other two angles are bounded from below and above
by $C \delta_0$ and $\frac{\pi}{2}-C \delta_0$. We show that in this situation we can reduce to two model cases, which we discuss below. This allows us to obtain the following result:

\begin{lem}
\label{lem:cov_R3}
Let $D$ be a triangle of type (R3). Let $R$ be a rectangle of side ratio $\delta_0$, which is parallel to one of the sides of $D$. Then it is possible to cover $D$ by finitely many scaled and translated copies of itself and by finitely many translated and scaled copies $R_k$ of the rectangle $R$ such that
\begin{itemize}
\item[(i)] $|\bigcup\limits_{k=1}^{K_j} R_{k}| \geq 10^{-6}|D|$,
\item[(ii)] $\sum\limits_{k=1}^{K_j} \Per(R_{k}) \leq \frac{C}{\delta_0}\Per(D)$.
\end{itemize}
\end{lem}

Our main ingredient in proving this is the following lemma:

\begin{lem}[Covering of triangles by rectangles]
\label{lem:cov_R3a}
Let $D_{1,m}$ denote a right angle triangle, in which the sides enclosing the right angle are of side lengths $1$ and $m$. Assume that
$m\in(0,50 \delta^{-1})$.
Let $R$ be a rectangle, which has side ratio $\delta\in(0,1)$. Assume that the longer side of $R$ is parallel to the side of the triangle $D_{1,m}$, which is of length $m$.\\
Then there exist a number $L=L(m)$ and disjoint, rescaled and translated copies $R_{k}$ of $R$ with the properties that:
\begin{itemize}
\item[(i)] $|D_{1,m}\setminus\bigcup\limits_{k=1}^{L}R_k| \geq 10^{-6}|D_{1,m}|$.
\item[(ii)] The sum of the perimeters of the rectangles $R_k$ satisfies 
\begin{align*}
\sum\limits_{k=1}^{L}\Per(R_k)\leq \frac{C(1+m\delta)}{\delta}\leq\frac{C}{\delta}\Per(D_{m,1}).
\end{align*}
\end{itemize}
\end{lem}

\begin{rmk}
We remark that for our application, the bound on $m$ does not impose an additional requirement. Indeed, the triangles of type (R3) only occur as artifacts of the coverings in Lemmas \ref{lem:cov_R1}, \ref{lem:cov_R2}. Here we may estimate $m$ by $\tan(\beta)$ for the error triangles in . For
  $\beta=\frac{\pi}{2}-\delta$, a Taylor expansion of
  $\frac{\sin(\frac{\pi}{2}-\delta)}{\cos(\frac{\pi}{2}-\delta)}$ entails the
  desired estimate.
\end{rmk}

Before explaining Lemma \ref{lem:cov_R3a}, we show how our main covering result, Lemma \ref{lem:cov_R3}, can be reduced to the situation of Lemma \ref{lem:cov_R3a}.

\begin{proof}[Proof of Lemma \ref{lem:cov_R3}]
We first claim that without loss of generality $D$ can be assumed to be of type (R3) with $R$ being parallel to one of the \emph{short} sides of the triangle $D$. Indeed, if $R$ is parallel to the long side of the triangle $D$, then this side is opposite of the right angle of the triangle. In this case, we split the triangle $D$ into two smaller triangles $D^{(1)}$, $D^{(2)}$ by connecting the corner, at which $D$ has its right angle, by the shortest line to the long side. The resulting triangles $D^{(1)}$, $D^{(2)}$ have the same angles as the original triangle (and in particular satisfy the non-degeneracy conditions for the angles, which are required in condition (R3)), but are now such that $R$ is parallel to one of their \emph{short} sides.\\
After this reduction, we seek to apply Lemma \ref{lem:cov_R3a} with $\delta=\delta_0$ for each of the triangles $D^{(1)}, D^{(2)}$. To this end we note that as $\beta_0 \geq C\delta_0$, we have that $m\leq C \delta_0^{-1}$. As a consequence, Lemma \ref{lem:cov_R3a} yields the desired result (by observing that $\Per(D^{(1)})+ \Per(D^{(2)}) \leq 2 \Per(D)$).
\end{proof}

\begin{figure}[t]
  \centering
  \includegraphics[page=7]{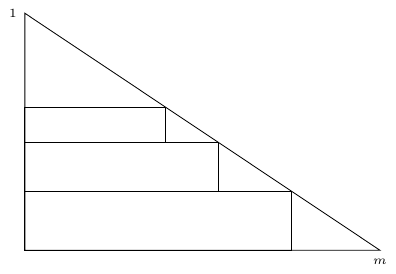}
  \includegraphics[page=8]{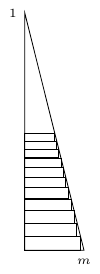}
  \caption{The triangles $D_{1,m}$ for $m\sim 1$ (left) and $m\sim \delta$ (right). Coverings of at least half the volume for $m>1$ and $m\sim \delta<1$.}
    \label{fig:cover_triangle}
\end{figure}

\begin{proof}[Proof of Lemma \ref{lem:cov_R3a}]
We construct the desired covering by a ``greedy" type algorithm. We
begin by fitting in the largest possible copy $R_1$ of $R$, which touches the side of 
the triangle $D_{1,m}$, which is of length $m$ (c.f. Figure \ref{fig:cover_triangle}). As the rectangle $R_1$ has to have a side ratio of $1:\delta$, its side lengths can be computed explicitly to be $l_1= \frac{m}{1+m\delta}, l_2 = \delta l_1$. Choosing $R_1$ as the first rectangle in the desired covering, we have created a decomposition of the triangle $D_{1,m}$ into three parts, the rectangle $R_1$ and two triangles, which are self-similar to the original triangle:
\begin{align*}
D_{1,m}= R_1 \cup D_{a_1, a_1 m} \cup D_{(1-a_1),(1-a_1)m}.
\end{align*}
Here $a_1:= l_2$ and hence $1-a_1= \frac{1}{1+\delta m}$ denote the similarity factors with respect to the original triangle $D_{1,m}$.
We iterate this procedure in the new triangle $D_{(1-a_1),(1-a_1)m}$, while ignoring the (smaller) triangle $D_{a_1,a_1 m}$. After $L$ steps of this algorithm we have obtained $L$ (up to null sets) disjoint rectangles $R_1,...,R_L$. We claim that if $L$ is chosen sufficiently large, the covering $\bigcup \limits_{k=1}^{L}R_k$ has the desired properties. Indeed, we choose $L$ such that $\left(\frac{1}{1+\delta m}\right)^{L}\in(\frac{1}{4}, \frac{1}{2})$ and first note that the construction of the rectangles $R_k$ is based on a self-similar iterative process with similarity factor $\lambda = \frac{1}{1+\delta m}$. Thus, we infer that
\begin{align*}
|\bigcup\limits_{k=1}^{L}R_k| 
&= |R_1| \sum\limits_{k=1}^{L} \lambda^{2 k} 
= |R_1| \frac{1-\lambda^{2(L+1)}}{1-\lambda^2} 
\geq |R_1| \frac{1}{2}\frac{1}{1-\lambda^2} 
= \frac{1}{2} \frac{m}{2+\delta m}\\
&\geq \frac{1}{200} m.
\end{align*}
Here we used the disjoint construction of the covering, the choice of $L$, the value of $\lambda$ and $\delta m \leq 50$. As $|D_{1,m}|= \frac{1}{2}m$, this yields the first claim. Similarly,
\begin{align*}
\sum\limits_{k=1}^{L}\Per(R_k)
&= \Per(R_1) \sum\limits_{k=1}^{L} \lambda^{ k} 
= \Per(R_1) \frac{1-\lambda^{L+1}}{1-\lambda} 
\leq \Per(R_1) 2\frac{1}{1-\lambda} \\
&\leq  4 \frac{1}{ \delta} 
\leq 4 \frac{1}{ \delta} \Per(D_{m,1}).
\end{align*}
Here we used that $\Per(R_1)\leq 2 l_1$ and that $\Per(D_{m,1})\geq 1$.
\end{proof}

\subsection{Proof of Proposition \ref{thm:covering}}
\label{sec:thm}

We initialize the construction by applying Step 1 in the Algorithms \ref{alg:construction}, \ref{alg:skew}. As these initial triangles are obtained as the level sets of a Conti construction with ratio $\delta_0$, they all form $\delta_0$ good triangles (c.f Lemma \ref{lem:cov_Cont}) and hence satisfy the properties of the theorem. It therefore remains to argue that this is preserved in our constructions from Sections \ref{sec:parallel} and \ref{sec:rota}. Given one of the triangles $D_j$ as in the theorem, the results of Propositions \ref{prop:parallel} and \ref{prop:rot} ensure this, once the rotation angle of the successive Conti constructions is controlled. This however is the achieved by virtue of Remark \ref{rmk:angle_1} and Lemma \ref{lem:angles}.

\section{Quantitative analysis}
\label{sec:quant}

After having recalled the qualitative construction of convex integration solutions in Section \ref{sec:Conti}, we now focus on controlling the scheme quantitatively. Here we rely on the quantitative covering results from Section \ref{sec:covering} (c.f. Propositions \ref{prop:parallel} and \ref{prop:rot}), which allow us to obtain bounds on the $BV$ norm of the iterates $u_k$ and the corresponding characteristic functions associated with the well $e^{(i)}\in K$ (Lemma \ref{lem:BV}). Combined with an $L^1$ estimate and the interpolation inequality from Theorem \ref{thm:interpol} or from Corollary \ref{cor:int}, this then yields the desired $W^{s,q}$ regularity of the characteristic function of the phases.\\

As in Section \ref{sec:alg}, given a matrix $M$ with $e(M)\in \intconv(K)$, we here assume that $\Omega:= Q_{\beta}[0,1]^2$, where $Q_{\beta}$ is the rotation, which describes how the Conti construction with respect to $M$ and $e^{(p)}_0$ is rotated with respect to the $x_1$-axis. This special case will play the role of a crucial building block in the situation of more general domains (c.f. Section \ref{sec:generaldomains}).\\

We begin by defining the characteristic functions associated with the corresponding wells:

\begin{defi}[Characteristic functions] 
\label{defi:characteristic}
We define the characteristic functions, $\chi_{k}^{(1)}, \chi_k^{(2)}, \chi_k^{(3)}$ associated with $e^{(1)}, e^{(2)}, e^{(3)}$ in the $k$-th step of the Conti construction as 
\begin{align*}
\chi_{k}^{(i)}(x)= \left\{
\begin{array}{ll}
1 & \mbox{ if } e(\nabla u_k)(x) = e^{(i)},\\
0 & \mbox{ else,}  
\end{array}\right. 
i\in\{1,2,3\}.
\end{align*}
We denote their point-wise a.e. limits as $k\rightarrow \infty$ by $\chi^{(i)}$, $i\in\{1,2,3\}$.
\end{defi}

We emphasize that these point-wise limits exists, since for a.e. point $x\in \Omega$ there exists an index $k_x\in \N$ such $x\in \Omega\setminus \Omega_{k_x}$. By our convex integration algorithm and by Definition \ref{defi:characteristic}, the value of $\chi_l(x)$ remains fixed for $l\geq k_x$.\\

Using the covering results from Section \ref{sec:covering}, we can address the $BV$ bounds for the characteristic functions $\chi_k^{(i)}$, $i\in\{1,2,3\}$:

\begin{lem}[$BV$ control]
\label{lem:BV}
Let $u_j:\Omega \rightarrow \R^2, \chi_j^{(1)},\chi_j^{(2)},\chi^{(3)}_j$ denote the deformation and characteristic functions, which are obtained in the $j$-th step of the convex integration scheme from Proposition \ref{prop:convex_int}. Let $\delta_0$ be as in Algorithm \ref{alg:construction}, Step 0 (b).
Then, there exist constants $C_0, C_1>0$ (independent of $\delta_0$) such that
\begin{align*}
\| \chi_j^{(i)}\|_{BV(\Omega)} \leq C_1(C_0\delta_0^{-1})^j \mbox{ for } i\in\{1,2,3\}.
\end{align*}
\end{lem}

\begin{proof}
We deduce the following iterative bound for the size of the BV norm:
\begin{align}
\label{eq:BVbound}
\sum\limits_{k=1}^{J_{j}}\sum\limits_{\Omega_{j+1,l}\in \mathcal{D}_1(\Omega_{j,k})} \Per(\Omega_{j+1,k}) \leq C \delta_0^{-1} \sum\limits_{k=1}^{J_{j}}\Per(\Omega_{j,k}).
\end{align}
To this end, let $\Omega_{j,k}$ be a triangle from the covering $\{\Omega_{j,k}\}_{j\in\{1,\dots,J_j\}}$ at the $j$-th iteration step. In particular, $e(\nabla u_j)$ is constant on $\Omega_{j,k}$. 
We apply Algorithms \ref{alg:construction}, \ref{alg:skew} and in these specify the choice of our covering to be the one of Proposition \ref{prop:parallel} or the one of Proposition \ref{prop:rot}.
In order to bound the resulting BV norm, we distinguish two cases:
\begin{itemize}
\item[(a)] \emph{The parallel case.} Assume that
$\Omega_{j,k}$ is of the type (P1) or (P2) (which, by the explanations below Definition \ref{defi:classes} holds in the parallel case). 
Thus, the Conti construction in the $j$-th and $(j+1)$-th step are nearly aligned in the direction of their degeneracy.
In this case, Proposition \ref{prop:parallel} is applicable and implies that 
\begin{align}
\label{eq:it1}
\sum\limits_{\Omega_{j+1,l}\in \mathcal{D}_1(\Omega_{j,k})} \Per(\Omega_{j+1,k}) \leq C \Per(\Omega_{j,k}),
\end{align}
for some absolute constant $C>0$.
\item[(b)] \emph{The rotated case.} Assume that $\Omega_{j,k}$ is of the type (R1)-(R3) (which, by the explanations below Definition \ref{defi:classes} holds in the rotated case).
Thus, the Conti construction in the $j$-th and $(j+1)$-th step are not aligned. By Lemma \ref{lem:angles} there are even lower bounds on the degree of alignment.
In this case, Proposition \ref{prop:rot} is applicable and yields that 
\begin{align}
\label{eq:it2}
\sum\limits_{\Omega_{j+1,l}\in \mathcal{D}_1(\Omega_{j,k})} \Per(\Omega_{j+1,k}) \leq C \delta_0^{-1} \Per(\Omega_{j,k}).
\end{align}
\end{itemize}

Combining both estimates (\ref{eq:it1}), (\ref{eq:it2}) and summing over all domains $\Omega_{j,k}$ for fixed $j$ implies (\ref{eq:BVbound}). From this we infer that 
\begin{align*}
  \sum\limits_{k=1}^{J_{j}}\Per(\Omega_{j,k}) \leq C \delta_0^{-j}.
\end{align*}
As
by construction
\begin{align*}
\sum\limits_{i=1}^{3}\|\chi_{j}^{(i)}\|_{BV(\Omega)} \leq C  \sum\limits_{k=1}^{J_{j}}\Per(\Omega_{j,k})\leq C \delta_0^{-j},
\end{align*}
we therefore obtain the statement of the Lemma.
\end{proof}

Using the explicit construction of our convex integration scheme, we further estimate the difference of two successive iterates in the $L^1$ norm:

\begin{lem}[$L^1$ control]
\label{lem:L1}
Let $u_k:\Omega \rightarrow \R^2, \chi_k^{(1)},\chi_k^{(2)},\chi^{(3)}_k$ denote the deformation and characteristic functions, which are obtained in the $k$-th step of the convex integration scheme from Lemma \ref{lem:convex_int}. 
Then,
\begin{align*}
\| \chi_k^{(i)}-  \chi_{k+1}^{(i)}\|_{L^1(\Omega)} \leq C\left(1-  \frac{7}{8} v_0 \right)^k  \mbox{ for } i\in\{1,2,3\}.
\end{align*}
\end{lem}

\begin{rmk}
In our realization of the covering argument, which is described in Section \ref{sec:covering}, we have chosen $v_0 = 10^{-6}$. In particular, it is independent of the boundary condition $M$ in (\ref{eq:incl}).
\end{rmk}

\begin{proof}
The proof follows immediately from the Conti construction and the observations that 
\begin{align*}
|\Omega_j|\leq C \left(1-\frac{7}{8}v_0 \right)^{j} |\Omega|,
\end{align*}
and that $\chi_k^{(i)}(x) = \chi_j^{(i)}(x)$ for a.e. $x\in \Omega\setminus
\Omega_j$, if $k\geq j$.
\end{proof}

Combining Lemma \ref{lem:BV} and \ref{lem:L1} with Theorem \ref{thm:interpol} or Corollary \ref{cor:int} yields the following regularity result:

\begin{prop}[Regularity of convex integration solutions]
\label{prop:reg}
Let $M$ with $e(M)\in \intconv(K)$ and assume that $\Omega=Q_{\beta}[0,1]^2$. Let $\delta_0>0$ be as in Step 0(b) in Algorithm \ref{alg:construction}.
Let $u:\Omega \rightarrow \R^2$ be a convex integration solution obtained according to the Algorithms \ref{alg:construction}, \ref{alg:skew} and described in Proposition \ref{prop:convex_int}. Then it is possible to obtain
\begin{align*}
\chi^{(i)} \in W^{s,q}
\end{align*}
for all $s \in(0,1),q\in(1,\infty)$ with $0<s q < \theta_0$ and
$\theta_0=\frac{\ln(1-\frac{7}{8}v_0)}{\ln(1-\frac{7}{8}v_0)+\ln(\delta_0)-\ln(C_0)}$.
Here $C_0>0$ is an absolute constant, which does not depend on $\delta_0$ and $M$.
\end{prop}

\begin{proof}
By Remark \ref{rmk:prod} it suffices to control a weighted product of the $L^1$ and the $BV$ norms.
Combining Lemmata \ref{lem:BV} and \ref{lem:L1} to this end yields that
\begin{align}
\label{eq:L1BV}
\|\chi_{j+1}^{(i)}-\chi_{j}^{(i)}\|_{L^{1}(\R^2)}^{1-\theta}\|\chi_{j+1}^{(i)}-\chi_{j}^{(i)}\|_{BV(\R^2)}^{\theta}
 \leq C \left(1-\frac{7}{8}v_0\right)^{j (1-\theta)}\left(\frac{C_0}{\delta_0}\right)^{j \theta}  .
\end{align}
Choosing $\theta \in(0,\theta_0)$ with $\theta_0:= \frac{\ln(1-\frac{7}{8}v_0)}{\ln(1-\frac{7}{8}v_0)+\ln(\delta_0)-\ln(C_0)} $ hence yields geometric decay for the right hand side of (\ref{eq:L1BV}). Invoking Remark \ref{rmk:prod} and using (\ref{eq:L1BV}) hence implies that
\begin{align}
\label{eq:L1BV1}
\|\chi_{j+1}^{(i)}-\chi_{j}^{(i)}\|_{W^{s,q}(\R^2)} 
&\leq  C_{s,q} \left(\left(1-\frac{7}{8}v_0\right)^{j (1-\theta_1)}\left(\frac{C_0}{\delta_0}\right)^{j \theta_1}\right)^{\frac{s}{\theta_1}}  
\end{align}
for any $\theta_1 \in (0,\theta_0)$ and any pair $(s,q)\in(0,1)\times (1,\infty)$ with $0<s q \leq \theta_1$.
Therefore, a telescope sum argument entails that the sequence $\{\chi_{j}^{(i)}\}_{j\in\N}$ forms a Cauchy sequence in $W^{s,q}(\R^2)$. Hence completeness yields that $\chi^{(i)}\in W^{s,q}$ for all $sq \in (0,\theta_0)$. 
\end{proof}

\begin{rmk}[Dependences]
\label{rmk:epsilon_dep}
We remark that the \emph{quantitative} dependences in Proposition \ref{prop:reg} are clearly non-optimal. Parameters, which could improve this are for instance:
\begin{itemize}
\item Varying the volume fraction $\lambda\in(0,1)$ in the Conti construction from Lemma \ref{cor:var_conti}.
\item Choosing a sharper relation between $\epsilon_j$ and $\delta_j$ and modifying the $j$-dependence of $\epsilon_j$ (for instance by only using summability for the stagnant matrices instead of the geometric decay, which is prescribed in Step 2(b) of Algorithm \ref{alg:construction}).
\end{itemize}
This would however not change the \emph{qualitative} behavior of the estimates.\\

\begin{figure}[h]
  \centering
  \includegraphics[width=0.7\linewidth, page=53]{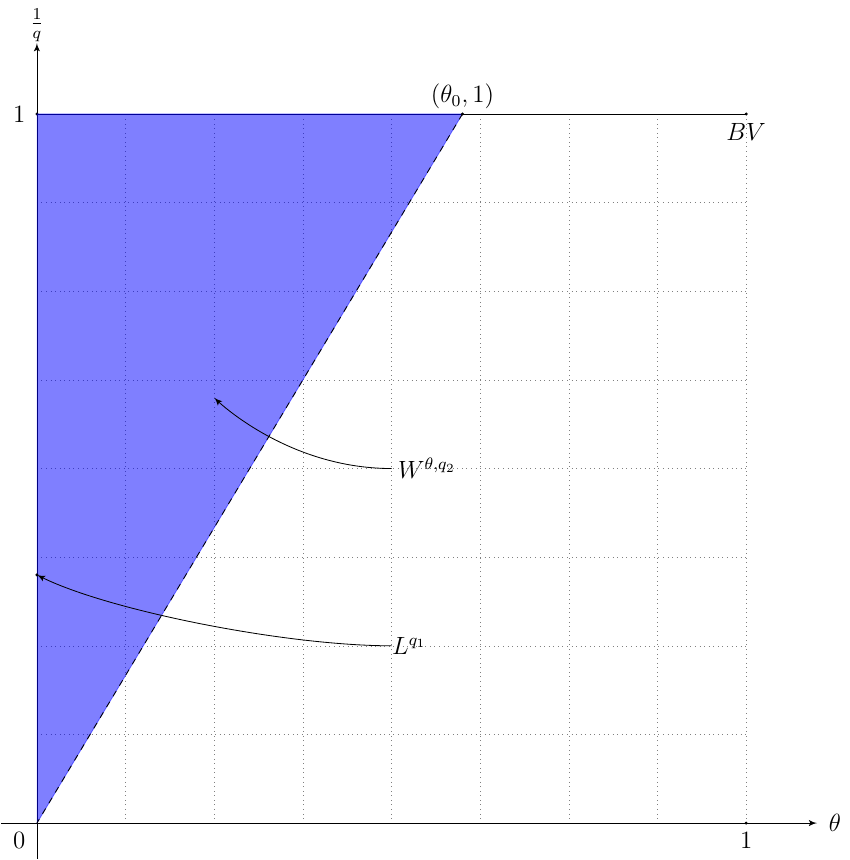}
  \caption{By interpolation, our decay and growth bounds for the $L^1$ and $BV$
    norms of the differences $\chi_{j+1}^{(i)}-\chi_j^{(i)}$ yield Cauchy sequences in the $W^{\theta,q}$ spaces
    inside the triangular region. Here, we use that the functions under
    consideration are characteristic functions and hence all $L^p$ norms with
    $1\leq p <\infty$ can be compared.}
  \label{fig:wthetaq}
\end{figure}

A \emph{qualitatively} different behavior would arise, if
in the proof of Lemma \ref{lem:BV} only case (a) occurred. Then
based on our construction in Proposition \ref{prop:parallel} and Lemma
\ref{lem:BV} the $\delta_0$ dependence would improve in Proposition
\ref{prop:reg}: In this case the choice of the product of the \emph{exponents}
$s,q$ in Proposition \ref{prop:reg} would \emph{not} depend on $\delta_0$, but would be \emph{uniform} in the whole triangle $\intconv(K)$. In this case only the \emph{value} of the $W^{s,q}$ norm would deteriorate with $\delta_0$. \\
We however remark that, as a matrix in $\intconv(K)$ is a convex combination of all three values of $e^{(1)}, e^{(2)}, e^{(3)}$, the described construction necessarily involves instances of case (b). It is however conceivable that by controlling the number of these steps, it could be possible to improve the dependence of $s,q$ on $\delta_0$. It is unclear (and maybe rather unlikely), whether it is possible to completely remove it with the described convex integration scheme.
\end{rmk}

\begin{rmk}[Fractal dimension] 
\label{rmk:frac1}
We emphasize that in accordance with Remark \ref{rmk:frac} the $W^{s,p}$ regularity of $\chi^{(i)}$ for $i\in\{1,2,3\}$ has direct implications on the (packing) dimension of the boundary of the sets $\{x\in \R^2: \chi^{(i)}(x)=1\}$.
\end{rmk}

Similarly, we obtain bounds on the deformation and the infinitesimal strain tensor:

\begin{prop}
\label{prop:reg1}
Let $\theta_0\in(0,1)$ be the exponent from Proposition \ref{prop:reg}. Then for all $s\in(0,1), p \in(1,\infty)$ with $sp <\theta_0$ there exist solutions $u:\Omega \rightarrow \R^2$ of (\ref{eq:incl}) with
\begin{align*}
\nabla u - M \in W^{s,p}(\R^{2}).
\end{align*}
\end{prop}

\begin{proof}
The proof is along the lines of the proof of Proposition \ref{prop:reg}. However, instead of estimating $\chi_{j+1}^{(i)} - \chi_{j}^{(i)}$, we bound $\nabla u_{j+1}-\nabla u_j$. Here the $BV$ bound follows from the bound for $\chi_{j+1}^{(i)}$ by noting the uniform boundedness of $\nabla u_j$ (c.f. Proposition \ref{prop:skewcontrol1}) and the fact that for the estimate for $\chi_{j+1}^{(i)}$ we used the whole resulting perimeter. Hence
\begin{align*}
\|\nabla u_{j+1}- \nabla u_j\|_{BV(\R^{2})} \leq C \| \chi_{j+1}^{(i)}- \chi_j^{(i)}\|_{BV(\R^{2})} .
\end{align*}
For the $L^1$ estimate we use the $L^{\infty}$ bound for $\nabla u_j$ (which follows from Proposition \ref{prop:skewcontrol1}) and the fact that $e(\nabla u_j) \in \overline{\conv(K)}$)
in combination with the fact that in the $j$-th iteration step $\nabla u_j$ is only changed on a volume fraction of $\left(1-\frac{7}{8}v_0\right)^{j}$. Thus,
\begin{align*}
\|\nabla u_{j+1}- \nabla u_j\|_{L^1(\Omega)} \leq C \max\limits_{x\in \R^2}|\nabla u_j(x)|\left(1-\frac{7}{8}v_0\right)^{j}.
\end{align*}
Hence the same interpolation as above yields the $W^{s,p}(\R^2)$ regularity of $\nabla u - M$,
 which implies
the desired result.
\end{proof}

\section{General Domains}
\label{sec:generaldomains}
In this section we explain how to construct the desired ``regular" convex integration solutions in arbitrary Lipschitz domains by using the bounds from the special cases, which were discussed in Section \ref{sec:quant}.
In this context our main result is the following:

\begin{prop}
\label{prop:cover_a}
Assume that $\Omega\subset \R^2$ is a bounded Lipschitz domain and suppose that $M \in
\R^{2\times 2}$ with $e(M)\in \intconv(K)$. Let $\beta\in [0,2\pi)$ be the angle, with which the Conti construction for $M$ is rotated with respect to the $x_1$-axis and let $\chi_k^{(i)}$ be defined as in Definition \ref{defi:characteristic}.
Let $\theta_0>0$ be the $W^{s,p}$ exponent for the regularity of $\chi^{(i)}$ with respect to the rotated unit square $Q_{\beta}[0,1]^2$ adapted to $M$, i.e. let $\theta_0$ be such that for all $s\in(0,1), p\in (1,\infty]$ with $0<sp <\theta_0$ and for some $\mu(s,p)\in(0,1)$ it holds
\begin{align}
\label{eq:control_a}
\|\chi^{(i)}_{k+1}-\chi_{k}^{(i)}\|_{BV(Q_{\beta}[0,1]^2)}^{\theta}\|\chi^{(i)}_{k+1}-\chi_{k}^{(i)}\|_{L^1(Q_{\beta}[0,1]^2)}^{1-\theta} \leq C(s,p)\mu(s,p)^{-k}.
\end{align}
Then there exists a constant $C(\Omega,M,s,p)$ and a family of subsets $\bar{\Omega}_k\subset \Omega$ such that
\begin{itemize}
\item[(i)] $\bar{\Omega}_k:= \bigcup\limits_{l=1}^{k}  \bigcup\limits_{m=1}^{K_l} Q_{l}^m$, where $Q_l^m:= \left( [0,\lambda_l]^2 + x_{l,m} \right)$ are (up to null-sets) disjoint cubes with $x_{l,m} \in \Omega$ and $\lambda_l:= 2^{-l}$ such that
\begin{align*}
\bar{\Omega}_k \nearrow \Omega \mbox{ in } L^1(\R^2)
\end{align*}
(in the sense of the convergences of their characteristic functions),
\item[(ii)] for $\tilde{\chi}^{(i)}_k(x):= \sum\limits_{l=1}^{k} \sum\limits_{m=1}^{K_l} \chi_{k}^{(i)}(\frac{x-x_{m,l}}{\lambda_l})\chi_{\bar{\Omega}_k}(x)$ the estimate (\ref{eq:control_a}) remains valid for all $s,p$ with $s\in(0,1), p\in(1,\infty]$ and $0<sp<\theta_0$. In the dependences the constant $C(s,p)$ however is replaced by $C(\Omega,M,s,p)$ and $\mu(s,p)$ replaced by $C(\theta)\mu(s,p)$. Here $\theta=\theta(s,p)$ is the interpolation exponent associated with $s,p>0$.
\end{itemize}
\end{prop}

As an immediate consequence we infer the following corollary:

\begin{cor}
\label{cor:domain1}
Suppose that $\Omega\subset \R^2$ is a bounded Lipschitz domain and assume that $M\in \intconv(K)$. Let $\beta\in[0,2\pi]$ be the angle, with which the Conti construction for $M$ is rotated with respect to the $x_1$-axis. Let $\theta_0>0$ be the limiting $W^{s,p}$ exponent for the regularity of $\chi^{(i)}$ with respect to the rotated unit square $Q_{\beta}[0,1]^2$ adapted to $M$. Then for all $s,p>0$ with $0<sp<\theta_0$
\begin{itemize}
\item[(i)] the point-wise limit $\tilde{\chi}^{(i)}$ of the functions   $\tilde{\chi}^{(i)}_k$ satisfies
\begin{align*}
\| \tilde{\chi}^{(i)}\|_{W^{s,p}(\Omega)} \leq C(\Omega, M, s,p).
\end{align*}
\item[(ii)] there exist solutions $u$ to (\ref{eq:incl}) with
\begin{align*}
\|\nabla u\|_{W^{s,p}(\Omega)} \leq C(\Omega, M, s,p).
\end{align*}
\end{itemize}
\end{cor}

\begin{proof}[Proof of Corollary \ref{cor:domain1}]
We first note that the point-wise limit $\tilde{\chi}^{(i)}$ exists, since $\bar{\Omega}_k \rightarrow \Omega$ and since $\chi_k^{(i)} \rightarrow \chi_k$ in a point-wise sense as $k\rightarrow \infty$.
With this at hand, the proof of Corollary \ref{cor:domain1} follows from Proposition \ref{prop:cover_a} by interpolation in an analogous way as explained in Proposition \ref{prop:reg}. 
We therefore omit the details of the proof of the corollary.
\end{proof}

\begin{figure}[t]
\centering
\includegraphics[width=6cm,page=14]{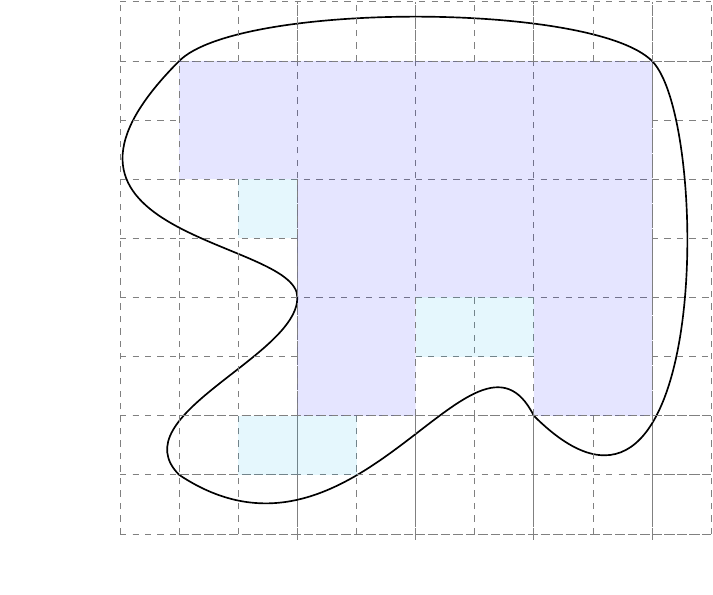}
\caption{The covering of $\Omega$ by squares of decreasing sizes defines the set $\Omega_k$.}
\label{fig:general_domain}
\end{figure}

We proceed to the proof of Proposition \ref{prop:cover_a}. Here we argue by covering our general domain $\Omega$ by the special domains from Section \ref{sec:quant} (Steps 1 and 2). On each of the special domains, we apply the construction from Section \ref{sec:quant} (c.f. also Algorithms \ref{alg:construction}, \ref{alg:skew}). In order to obtain a sequence with bounded $W^{s,p}$ norm, we however do not refine to arbitrarily fine scales immediately, but proceed iteratively (c.f. Step 3). A central point here is to control the necessary number of cubes at each scale (Claim \ref{claim:1}), since this has to be balanced with the corresponding energy contribution (c.f. Step 4). To this end, we use a ``volume argument", which by the Lipschitz regularity of the domain allows us to infer information on the number of cubes on each scale (c.f. Proof of Claim \ref{claim:1}).

\begin{proof}[Proof of Proposition \ref{prop:cover_a}]
\emph{Step 1: Covering of a general Lipschitz domain.} We may assume that
  $M=0$ and first consider the case of $\Omega$ being a domain, which is bounded
  by the $x_2$-axis, the segment $[0,1]\times \{0\}$, a Lipschitz graph
  $f:[0,1]\rightarrow \R$ and the segment $\{1\}\times [0,f(1)]$. By symmetry we may further assume that $f(x_1)\geq 0$ for all $x_1\in[0,1]$. For general
domains $\Omega$, by the compactness and Lipschitz regularity, we may locally
reduce to a similar case, where $f$ is a Lipschitz curve, but not necessarily a
graph. However, all arguments in the following extend to that case as well.
\\

\emph{Step 2: Counting cubes.}
Let $\tilde{\Omega}_l:= \bigcup\limits_{k=1}^{K_l} Q_l^k$, where $Q_l^k\subset \Omega$ are (up to zero sets) disjoint, grid cubes of an axis-parallel grid of grid size $\lambda_l:=2^{-l}$. We choose $K_l\in \N$ maximal. Thus, by definition we have that $\tilde{\Omega}_l \subset \tilde{\Omega}_{l+1}\subset \Omega$ for all $l\in \N$. In the limit $l\rightarrow \infty$ the sets $\tilde{\Omega}_l$ eventually cover the whole set $\Omega$ (which we assume to be as in Step 1). \\
We estimate the number of the cubes, which are contained in the sets $\tilde{\Omega}_{l+1}\setminus \tilde{\Omega}_l$. For these we claim:

\begin{claim}
\label{claim:1}
The set $\tilde{\Omega}_{l+1}\setminus \tilde{\Omega}_l$ contains at most $C_f \lambda_{l+1}^{-1}$ of the grid cubes $Q_{l+1}^k\subset \Omega$.
\end{claim}

\begin{proof}[Proof of Claim \ref{claim:1}]
Indeed, we first observe that for a sufficiently large constant $C_f$ (depending on $f$, c.f. Remark \ref{rmk:Lipconstant}) and for a sufficiently large value of $l\in \N$ every point $x\in\Omega$ in the subgraph $S(f,l)$ of $f-C_f \lambda_{l}$ is contained in a cube $Q_l^k\subset \tilde{\Omega}_l$. Hence at least a volume of size 
\begin{align*}
|S(f,l) \cap \Omega|:=\int\limits_{0}^{1}f(x)dx - C_f \lambda_l,
\end{align*}
is completely covered by cubes of size $\lambda_l$. There may be additional cubes of size $\lambda_l$ contained in $\tilde{\Omega}_l$. As however only a volume of size $C_f \lambda_l$ is left and as each cube has volume $\lambda_l^{-2}$, the number of these additional cubes is controlled by
\begin{align*}
\#\{\mbox{grid cubes of size $\lambda_l$ in } \tilde{\Omega}_l \setminus S(f,l)\} \leq 2 C_f \lambda_l \lambda_l^{-2} 
\leq 2 C_f \lambda_l^{-1}.
\end{align*}
Combining this with the observation that 
\begin{align*}
|\Omega\cap S(f,l+1)|-|\Omega \cap S(f,l)| = C_f (\lambda_{l}- \lambda_{l+1}) 
= C_f \lambda_{l+1},
\end{align*}
which implies that $S(f,l+1)$ has at most $C_f \lambda_{l+1}^{-1}$ more cubes of size $\lambda_{l+1}$ than $S(f,l)$,
we infer that 
$\tilde{\Omega}_{l+1}\setminus \tilde{\Omega}_{l}$
contains at most $4 C_f \lambda_{l+1}^{-1}$ cubes of size $\lambda_{l+1}$.
\end{proof}

\emph{Step 3: Definition of the algorithm.}
We use the following definitions 
\begin{align*}
\hat{\Omega}_1:=\tilde{\Omega}_1,\
\hat{\Omega}_l:= \Omega_l \setminus \bigcup\limits_{j=1}^{l-1} \tilde{\Omega}_l \mbox{ for } l\geq 2.
\end{align*}
With this we set (as illustrated in Figure \ref{fig:general_domain})
\begin{align*}
\bar{\Omega}_k:=\bigcup\limits_{j=1}^{k}\hat{\Omega}_j.
\end{align*}
We recall that by Claim \ref{claim:1} we have that
$
\bar{\Omega}_{k+1}\setminus \bar{\Omega}_k $
is a union of at most $C C_f \lambda_{k+1}^{-1}$ cubes of side lengths $\lambda_{k+1}$.\\
Denoting by $u_{k}$ the deformation in step $k$ of the Algorithms \ref{alg:construction}, \ref{alg:skew} with initialization as in Step 1, we define the deformation $\tilde{u}_k|_{\bar{\Omega}_k}$ on $\bar{\Omega}_k$ in the $k$-th step as
\begin{align*}
\tilde{u}_k(x):=
\left\{
\begin{array}{ll}
 \lambda_l u_{k}(\lambda_l^{-1} (x-x_{l,k})) &\mbox{ for } x \in Q_{l}^k \subset \hat{\Omega}_l \cap \bar{\Omega}_k,\\
0 &\mbox{ for } x \notin \bar{\Omega}_k,
\end{array} \right.
\end{align*}
where $x_{l,k}\in Q_{l}^k$ denotes the center of the cube $Q_l^k$. We observe that $\tilde{u}_{k}$ is a Lipschitz function (since $M=0$). We define $\tilde{\chi}_k^{(i)}$ as the associated characteristic function for the well $e^{(i)}$, i.e.
\begin{align*}
\tilde{\chi}_k(x):=
\left\{
\begin{array}{ll}
1 &\mbox{ if } e(\nabla u)(x) = e^{(i)},\\
0 &\mbox{ else}.
\end{array} \right.
\end{align*}

\emph{Step 4: Energy estimate.}
We note that if 
\begin{align*}
E_{k,1}:=\|\chi_{k+1}^{(i)}-\chi_k^{(i)}\|_{BV([0,1]^2)}^{\theta}\|\chi_{k+1}^{(i)}-\chi_k^{(i)}\|_{L^1([0,1]^2)}^{1-\theta} \leq C \mu(s,p)^{k},
\end{align*}
scaling implies that
\begin{align*}
E_{k,l}:=\|\chi_{k+1}^{(i)}-\chi_k^{(i)}\|_{BV([0,\lambda_l]^2)}^{\theta}\|\chi_{k+1}^{(i)}-\chi_k^{(i)}\|_{L^1([0,\lambda_l]^2)}^{1-\theta} \leq C \mu(s,p)^{k} \lambda_l^{2-\theta}.
\end{align*}
Hence, we estimate
\begin{align*}
E_k&:=\|\tilde{\chi}_{k+1}^{(i)}-\tilde{\chi}_k^{(i)}\|_{BV(\Omega_{k+1})}^{\theta}\|\tilde{\chi}_{k+1}^{(i)}-\tilde{\chi}_k^{(i)}\|_{L^1(\Omega_{k+1})}^{1-\theta}\\
& \leq \sum \limits_{l=1}^{k} E_{k,l} \#\{Q_l^{k} \subset \hat{\Omega}_l\cap \bar{\Omega}_k: Q_l^k \mbox{ is a grid cube of size } \lambda_l\}\\
& \stackrel{\text{Claim \ref{claim:1}}}{\leq} C C_f \sum \limits_{l=1}^{k} E_{k,l} \lambda_l^{-1}= C C_f \sum \limits_{l=1}^{k} \mu(s,p)^{k} \lambda_l^{2-\theta} \lambda_l^{-1}\\
&= C C_f \mu(s,p)^k\sum \limits_{l=1}^{k} 2^{-l(1-\theta)} \\
& \leq C(\theta) C_f \mu(s,p)^k  \rightarrow 0 \mbox{ as } k \rightarrow \infty.
\end{align*}
Thus, for $s,p$ as above, the sequences $\tilde{\chi}_{k}^{(i)}$ are still Cauchy in $W^{s,p}$. This concludes the proof.
\end{proof}

\begin{rmk}
\label{rmk:Lipconstant}
The constant $C_f$ from Claim \ref{claim:1} can be controlled by
$C [\nabla f]_{C^{0,1}(\Omega)}$, for some universal constant $C>1$.
\end{rmk}

\section*{Acknowledgements}
We would like to thank Sergio Conti for suggesting the problem and for useful discussions. Further, we would like to thank Felix Otto for making his Minneapolis lecture notes available to us and for helpful discussions on the project.

\section{Appendix: A Construction Using Symmetry}
\label{sec:append}

In this final section we recall a special solution to (\ref{eq:incl}) with $M=0$, which enjoys $BV$ regularity (c.f. \cite{C}, \cite{Pompe}, \cite{CPL14}). This construction crucially relies on symmetry. It hence gives rise to the question whether it is possible to exploit symmetry in a more systematic way in constructing ``regular" convex integration solutions.\\

In describing this particular solution, we first present all possible zero homogeneous solutions to the differential inclusion (\ref{eq:incl_1}) in Section \ref{sec:exact}. In Section \ref{sec:BV} we then rely on this to construct the desired solution with $BV$ regularity.

\subsection{Exactly stress-free configurations for the hexagonal\hyp to\hyp rhombic phase transformation}
\label{sec:exact}

\begin{figure}[t]
\centering
\includegraphics[scale=0.5, page=24]{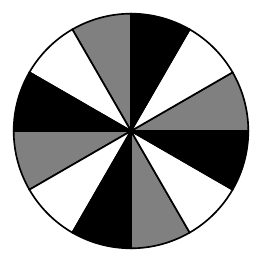}
\includegraphics[scale=0.5, page=31]{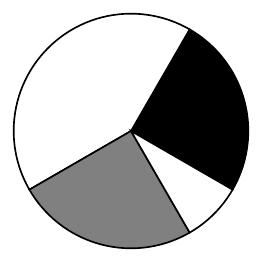}
\includegraphics[scale=0.5, page=32]{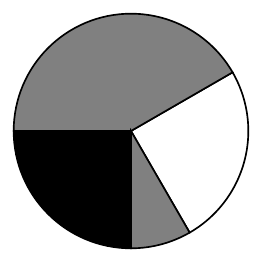}
\includegraphics[scale=0.5, page=33]{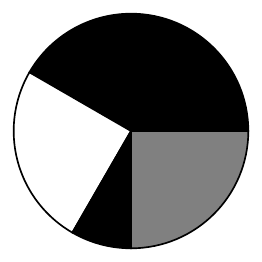}
\includegraphics[scale=0.4, page=25]{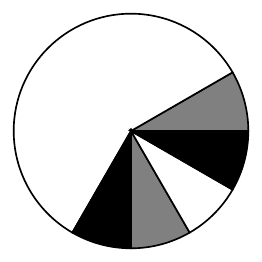}
\includegraphics[scale=0.4, page=26]{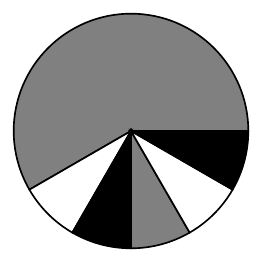}
\includegraphics[scale=0.4, page=27]{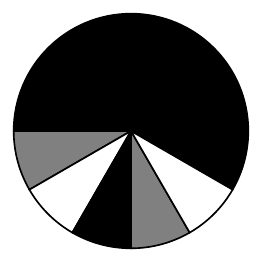}
\includegraphics[scale=0.4, page=28]{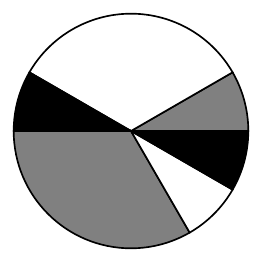}
\includegraphics[scale=0.4, page=29]{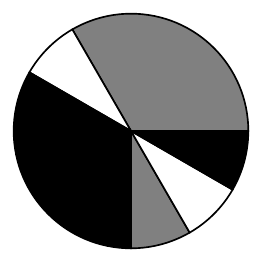}
\includegraphics[scale=0.4, page=30]{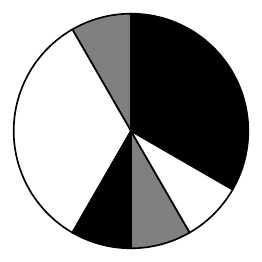}

\caption{The zero-homogeneous exactly stress-free configurations for the hexagonal-to-rhombic phase transition. The white sectors correspond to the variant $e^{(1)}$, the gray ones to $e^{(2)}$ and the black ones to $e^{(3)}$:
There is a twelve-fold corner, up to symmetry (i.e. rotations by $\frac{\pi}{2}$) a single variant of a four-fold corner and up to symmetry two different variants of a six-fold corner.}
\label{fig:zeroh}
\end{figure}

The high degree of non-rigidity of the hexagonal-to-rhombic phase transition is reflected in a comparably large number of possible solutions to (\ref{eq:incl_1}). There is already a large number of solutions with \emph{homogeneous} strain, i.e. solutions $u$ of (\ref{eq:incl_1}) such that all the phases intersect in a single point and the strains are zero-homogeneous functions $e(\nabla u)(\lambda x)= e(\nabla u)(x)$ (c.f. Figure \ref{fig:zeroh}). 
To verify this, we recall, c.f. Lemma 17 in \cite{R16}, that the following conditions are necessary and sufficient for the presence of such a corner.

\begin{lem}[Compatibility condition at a zero-homogeneous corner]
\label{lem:compat}
Let $e:\R^2 \rightarrow \R^{2\times 2}_{sym}$ be a zero-homogeneous tensor field. Let $A_1,\dots,A_m \in \R^{2\times 2}_{sym}$ be symmetric matrices such that $A_{j}\neq A_{j+1}$, where $j,j+1$ are considered modulo $m$. 
Assume that along a closed circle surrounding the origin, $e$ successively attains 
the values $A_1,\dots,A_m$ (as for instance in Figure \ref{fig:zeroh}). Then $e$ is a strain tensor, i.e. there exists a function $u\in W^{1,\infty}_{loc}(\R^2)$ such that $e=e(\nabla u)$, if and only if the following two conditions are satisfied:
\begin{enumerate}
\item There exist vectors $a_i\in \R^2\setminus \{0\}$, $n_i\in \mathbb{S}^1$ such that
\begin{align*}
A_{i}-A_{i+1} = \frac{1}{2}(a_i\otimes n_i + n_i \otimes a_i) \mbox{ for } i\in \{1,...,m\}.
\end{align*}
\item $\sum\limits_{i=1}^{m}a_i\otimes n_i = 0.$
\end{enumerate}
\end{lem}

Here the first condition corresponds to tangential continuity along the interfaces of the jumps. The second requirement ensures the compatibility of the skew symmetric part of the gradient.\\

Keeping this in mind, a symbolic Mathematica computation allows to determine all possible zero-homogeneous corners. This leads to the following classification result:
 
\begin{obs}
Apart from simple laminates (which trivially exist due to the symmetrized rank-one connectedness of the strains) there are the following compatible zero-homogeneous configurations of strains (c.f. Figure \ref{fig:zeroh}):
\begin{itemize}
\item a single configuration involving twelve strains, 
\item (up to symmetry) two types of configurations involving six strains, 
\item (up to symmetry) one configuration involving four strains.
\end{itemize} 
\end{obs}

These homogeneous configurations can further be combined to yield compatible ``zig-zag" configurations, c.f. Figure \ref{fig:cross}. The zero-homogeneous configurations will in the sequel serve as the building blocks of our constructions. 

\begin{figure}
\centering
\includegraphics[scale=0.5, page=34]{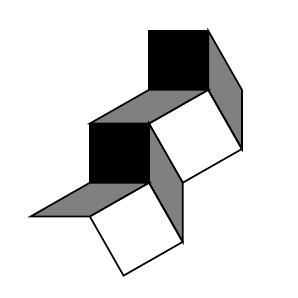}
\includegraphics[scale=0.5, page=35]{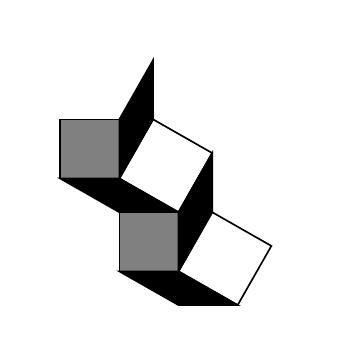}
\caption{The homogeneous corners can be combined to yield the above patterns (up to symmetry they correspond to a single pattern). As before white corresponds to the variant $e^{(1)}$, gray to the variant $e^{(2)}$ and black to the variant $e^{(3)}$.}
\label{fig:cross}
\end{figure}

\subsection{A BV construction for zero boundary data}
\label{sec:BV}
In this section we present an explicit construction of a solution to (\ref{eq:incl}) with $M=0$ and $e(\nabla u)\in BV$ (and with $\nabla u \in L^{p}(\R^2)$ for all $p\in(1,\infty)$ but $\nabla u \notin L^{\infty}(\R^2)$). In our construction (c.f. Proposition \ref{prop:zero}) we crucially rely on symmetry properties of the strains. As the wells form an equilateral triangle in strain space, it appears plausible to expect the best regularity properties arise in the center of the convex hull of $K$, i.e. for solutions with zero boundary data.\\

For an arbitrary domain this construction (which is motivated by the constructions of Conti \cite{C} and which was similarly already used in \cite{Pompe} and \cite{CPL14}) can be used to obtain a solution to (\ref{eq:incl}) with 
\begin{align*}
e(\nabla u) \in W^{s,p}(\R^2) \mbox{ for all } s\in(0,1), p\in(1,\infty) \mbox{ with } sp<1.
\end{align*}

As we are in two-dimensions, symmetrized rank-one connections exist between any pair of symmetric matrices with vanishing trace (c.f. Lemma \ref{lem:rk1}). In particular, all the strains $e^{(1)}, e^{(2)}, e^{(3)}$ are compatible with any (constant) boundary condition. Yet, a priori it is not clear, whether this can be turned into a global configuration involving as few strains as possible. In the case of zero boundary data, this is indeed possible, while preserving very good regularity properties:

\begin{prop}[Zero boundary data construction]
\label{prop:zero}
There exists a bounded domain $\Omega\subset \R^2$, $\Omega \neq \emptyset$, and a solution $u:\R^2 \rightarrow \R^2$ of (\ref{eq:incl}) with $M=0$, such that
\begin{align*}
e(\nabla u) \in BV(\R^2).
\end{align*}
\end{prop}

We emphasize that this construction is not new and that related constructions, using the symmetry of the domain, already appeared in \cite{C}, \cite{Pompe}, \cite{CPL14}. \\

\begin{figure}[t]
\centering
\includegraphics[scale=1, page=10]{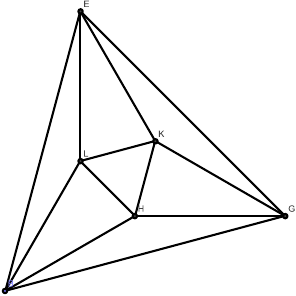}
\includegraphics[scale=1, page=11]{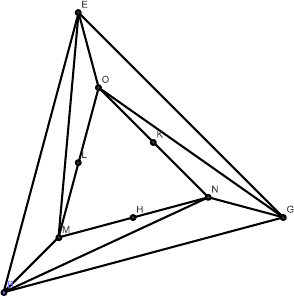}
\caption{The first iteration step in the deformation corresponding to zero boundary data. It is possible to interpret the construction as (a linearization at the identity of) the deformation depicted in the first two pictures: The undeformed reference configuration (left) is deformed into the configuration on the right.}
\label{fig:Def}
\end{figure}

\begin{proof}
In order to obtain the desired construction, we consider an equilateral triangle rotated by $\frac{\pi}{12}$ with respect to the $x_1$-axis and a self-similar copy of it which is homothetically positioned in the larger one at a length ratio $4-2\sqrt{3}$ and then rotated by $\frac{\pi}{3}$ with respect to the outer triangle (c.f. Figure \ref{fig:Def}, left). Now we rotate the inner triangle by $\frac{\pi}{3}$, so that it turns into a homothetically scaled version of the outer triangle and stretch it by a factor 2 while preserving the boundary of the larger triangle (c.f. Figure \ref{fig:Def}, right). This leads to the gradient distribution depicted in Figure \ref{fig:Def1}. In particular, the only strains, which are used consist of $e^{(1)}, e^{(2)}, e^{(3)}$ as well as the zero strain. This can be iterated in the respectively smaller triangles. As only the skew symmetric part of the strains grow, while the symmetric strains are fixed in the set of our wells and the zero matrix, and as the new interior 
triangle is a self-similar copy of the outer original triangle with a constant ratio of $2(4-2\sqrt{3})$, this yields the claimed energy contributions:
In each step the construction leads to a bound of the form
\begin{align*}
\|e(\nabla u_n)\|_{BV(\R^2)}\sim \sum\limits_{k=0}^{n} c^k \leq C,
\end{align*}
where $c\in(0,1)$ and $C>1$ is a universal constant. Thus, passing to the limit $n\rightarrow \infty$ implies the desired regularity result for the symmetrized part. 
\end{proof}

\begin{rmk}
\label{rmk:scaling1} 
We remark that this can easily be turned into a scaling result for associated elastic and surface energies on the domain $\Omega$. It gives an energy scaling for minimizers subject to $\nabla u = 0$ on the boundary of $\Omega$, which corresponds to a ``surface energy contribution". 
\end{rmk}

A second way of producing the construction of Proposition \ref{prop:zero} relies on the homogeneous building blocks, which were described in the previous section. We
\begin{itemize}
\item compute all the possible four-fold corners consisting of the strains $e^{(1)},e^{(2)},$ $e^{(3)},0$. There is an admissible four-fold corners with a large portion (more precisely involving an angle of $\frac{5 \pi}{3}$) of the zero phase (this is one of the corner depicted in the left picture in Figure \ref{fig:Def}),
\item act on this configuration via a rotation of $\frac{2 \pi}{3}$: This yields two new compatible four-fold corners (these are the other two corners in the left picture in Figure \ref{fig:Def}),
\item combine the corners in an equilateral triangle as in Figure \ref{fig:Def},
\item check the compatibility (by means of Lemma \ref{lem:compat}) of the resulting central corners with the strain $e=0$.
\end{itemize}

Similar constructions (but with different symmetries) can be applied in other two-dimensional configurations in matrix space (e.g. with the symmetries of a square), if the strains are arranged in a symmetric polygon in strain space (e.g. in square, in which case one possible solution would be the one given in Lemma \ref{lem:conti_undeformed}, c.f. also Figure \ref{fig:conti}).\\

\begin{figure}[t]
\centering
\includegraphics[scale=1,page=13]{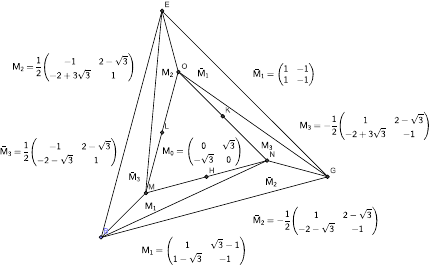}
\caption{The gradient distribution of the deformation. The construction exploits the symmetry of the wells in strain space. The figure depicts the deformation gradients (after linearization) in the image configuration. The matrices, which yield the same symmetrized gradient, are denoted by the same indeces, e.g. $M_1$ and $\bar{M}_1$.}
\label{fig:Def1}
\end{figure} 

In a general domain the construction from above can be applied with a (logarithmic) loss: 

\begin{prop}[General zero boundary data construction]
\label{prop:zero1}
Let $\Omega \subset \R^2$ be a non-empty Lipschitz domain. Then there exists a configuration such 
for all $s\in(0,1), p\in(1,\infty)$ with $sp <1$ we have
\begin{align*}
e(\nabla u) \in W^{s,p}(\R^2).
\end{align*}
\end{prop}

As the passage from the special domain to an arbitrary domain follows by a covering argument analogous to the one presented in Section \ref{sec:generaldomains}, we omit the proof here.

\begin{rmk}
\label{rmk:scaling2}
Similarly as explained in Remark \ref{rmk:scaling1}, Proposition \ref{prop:zero1} can also be transformed into a scaling result. In comparison to the ``surface energy scaling'', which is obtained from Proposition \ref{prop:zero}, we however lose a logarithmic factor in the corresponding construction. It is an interesting and challenging open problem to decide whether this logarithmic loss is necessary in a general domain.
\end{rmk}

\bibliographystyle{alpha}
\bibliography{citations}

\end{document}